\documentclass[10pt, a4paper]{amsart}
\usepackage{enumitem,changepage,amsmath,amstext,amscd, amssymb,txfonts, epsfig, psfrag, color, multicol, graphicx}
\usepackage[colorlinks=true, pdfstartview=FitV, linkcolor=blue, citecolor=blue, urlcolor=blue]{hyperref}
\usepackage[normalem]{ulem}
\usepackage[section]{algorithm}
\usepackage{algorithmicx,algpseudocode}
\usepackage{bm,epstopdf}

\algrenewcommand{\algorithmiccomment}[1]{\hfill  \emph{#1}}

\makeatletter
\def\@tocline#1#2#3#4#5#6#7{\relax
\ifnum #1>\c@tocdepth 
  \else 
    \par \addpenalty\@secpenalty\addvspace{#2}%
\begingroup \hyphenpenalty\@M
    \@ifempty{#4}{%
      \@tempdima\csname r@tocindent\number#1\endcsname\relax
 }{%
   \@tempdima#4\relax
 }%
 \parindent\z@ \leftskip#3\relax \advance\leftskip\@tempdima\relax
 \rightskip\@pnumwidth plus4em \parfillskip-\@pnumwidth
 #5\leavevmode\hskip-\@tempdima #6\nobreak\relax
 \ifnum#1<0\hfill\else\dotfill\fi\hbox to\@pnumwidth{\@tocpagenum{#7}}\par
 \nobreak
 \endgroup
  \fi}
\makeatother
\let\oldtocsection=\tocsection
\let\oldtocsubsection=\tocsubsection
\let\oldtocsubsubsection=\tocsubsubsection
\renewcommand{\tocsection}[2]{\hspace{0em}\oldtocsection{#1}{#2}}
\renewcommand{\tocsubsection}[2]{\hspace{1em}\oldtocsubsection{#1}{#2}}
\renewcommand{\tocsubsubsection}[2]{\hspace{2em}\oldtocsubsubsection{#1}{#2}}

\theoremstyle{plain}

\newtheorem{theorem}{Theorem}
\newtheorem{thm}{Theorem}[section]
\newtheorem{lemma}[thm]{Lemma}
\newtheorem{cor}[thm]{Corollary}
\newtheorem{prop}[thm]{Proposition}

\theoremstyle{definition}
\newtheorem{defn}[thm]{Definition}

\newtheorem{remark}[thm]{Remark}

\newtheorem{Example}[thm]{Example}

\newtheorem{Open questions}[thm]{Open questions}
\newtheorem{Open question}[thm]{Open question}
\newtheorem{Open problems}[thm]{Open problems}
\newtheorem{Open problem}[thm]{Open problem}

\newcommand{\tc}[2]{\textcolor{#1}{#2}}
\definecolor{magenta}{rgb}{.5,0,.5} 
\definecolor{dred}{rgb}{.5,0,0} 
\definecolor{green}{rgb}{0,.5,0} \newcommand{\green}[1]{\tc{green}{#1}}
\definecolor{blue}{rgb}{0,0,1} \newcommand{\blue}[1]{\tc{blue}{#1}}
 \definecolor{black}{rgb}{0,0,0} 
\definecolor{dgreen}{rgb}{0,.3,0} 
\definecolor{vdred}{rgb}{.3,0,0} 
\definecolor{red}{rgb}{1,0,0} 
\definecolor{gray}{rgb}{.5,.5,.5} \newcommand{\gray}[1]{\tc{gray}{#1}}
\definecolor{cerulean}{rgb}{0,.48,.65} 
\definecolor{gold}{rgb}{0.80,0.58,0.05} 
\definecolor{orange}{rgb}{1.00,0.50,0} 
\definecolor{pink}{rgb}{1.00,0.08,0.58}

\def\Bbb{\mathbb}

\DeclareMathOperator{\Rank}{rank}

\newcommand{\alg}[1]{\textbf{\texttt{\textup{#1}}}}

\def\Z{\Bbb{Z}}
\def\N{\Bbb{N}}
\def\N{\Bbb{N}}

\newcommand{\cyc}[1]{\langle #1 \rangle}
\def\ni{\noindent}

\def\Img{\hbox{\rm Img}}
\def\Area{\hbox{\rm Area}}

\def\Dist{\hbox{\rm Dist}}

\def\sgn{\hbox{\rm sgn}}

\def\F+L{\hbox{$\textup{F}\!_+\textup{L}$}}

\def\rank{\hbox{\rm rank}}
\def\ssm{\smallsetminus}

\def\Invalid{\textbf{invalid}}

\def\invalid{\textbf{invalid}}

\renewcommand{\H}{\mathcal{H}}

\def\ms{\medskip}

\def\onto{{\kern3pt\to\kern-8pt\to\kern3pt}}

\def\<{\langle}
\def\>{\rangle}
\def\|{{\ |\ }}

 \def\inv{^{-1}}

\def\inv{^{-1}}

\newcommand{\set}[1]{\left\{#1\right\}}

\newcommand{\abs}[1]{\left|#1\right|}
\renewcommand{\ni}{\noindent}

\renewcommand{\ms}{\medskip}

\def\*{^{\star}}

\renewcommand{\labelenumi}{\textup{(}\roman{enumi}\textup{)}}

\setlist[enumerate,1]{label=\arabic*.,ref=\arabic*}
\setlist[enumerate,2]{label=\arabic{enumi}.\arabic*.,ref=\arabic{enumi}.\arabic*}
\setlist[enumerate,3]{label=\arabic{enumi}.\arabic{enumii}.\arabic*.,ref=\arabic{enumi}.\arabic{enumii}.\arabic*}
\setlist[enumerate,4]{label=\arabic{enumi}.\arabic{enumii}.\arabic{enumiii}.\arabic*.,ref=\arabic{enumi}.\arabic{enumii}.\arabic{enumiii}.\arabic*}
\begin{document}

\title[Taming the hydra]{Taming the hydra: \\ the word problem and extreme integer compression}

\author{W.\ Dison, E.\ Einstein and T.R.\ Riley}

\thanks{We gratefully acknowledge partial  support from NSF grant DMS-1101651 (TR) and Simons Collaboration Grant 318301 (TR), and the hospitality of the Mathematical Institute, Oxford (EE \& TR),  and the Institute for Advanced Study, Princeton (TR)  during the writing of this article.  }

\begin{abstract}
For a finitely presented group, the word problem asks for an algorithm
which declares whether or not words on the generators represent the
identity.  The Dehn function is a complexity measure of a direct attack
on the word problem by applying the defining relations. Dison \& Riley showed that a ``hydra phenomenon'' gives rise to novel groups with extremely fast
growing (Ackermannian) Dehn functions.  Here we show that nevertheless, there are efficient (polynomial time) solutions to the word problems of these groups.  Our main innovation is a means of computing efficiently with enormous integers which are represented in compressed forms by  strings of Ackermann functions.

 \ms

\footnotesize{\ni \textbf{2010 Mathematics Subject
Classification:  20F10, 20F65, 68W32, 68Q17}  \\ \ni \emph{Key words and phrases:} Ackermann functions,  subgroup distortion, Dehn function, hydra, word problem, membership problem, polynomial time}
\end{abstract}

\date \today

\maketitle

\tableofcontents

\section{Introduction} \label{intro}

\subsection{Ackermann functions and compressed integers} \label{1.1}

Ackermann functions $A_i : \N \to \N$ are a family of increasingly fast-growing functions beginning $A_0: n \mapsto n+1$, $A_1: n \mapsto 2n$, and $A_2:n\mapsto 2^n$, and with subsequent $A_{i+1}$ defined recursively so that $A_{i+1}(n+1) =   A_iA_{i+1}(n)$ and $A_{i+1}(0)=1$. (More details follow in Section~\ref{Ackermann intro}.)

Starting with zero and successively applying a few such functions and their inverses  can produce an enormous integer.  For example, $$A_3 A_0 A_1^2 A_0(0) \ = \ A_3 A_0 A_1^2  (1) \ = \ A_3 A_0 A_1  (2) \ = \ A_3 A_0 (4) \ = \ A_3 (5) \ = \ 2^{65536}$$ because $$A_3(5)  \ = \  A_2^5A_3(0) \ = \  A_2^5 (1) \ = \  \mbox{\footnotesize{$2$}}^{\mbox{\footnotesize{$2$}}^{{\mbox{\footnotesize{$2$}}^{\mbox{\footnotesize{$2$}}^{\mbox{\footnotesize{$2$}}}}}}} \ = \ 2^{65536}.$$

In this way Ackermann functions provide highly compact representations for some very large numbers.  

In principle, we could compute with these representations  by evaluating the integers they represent and then using standard integer arithmetic, but this can be monumentally inefficient because of the   sizes of the integers.    We will explain how  to calculate efficiently  in a rudimentary way with such representations of integers:  

 \begin{theorem}  \label{Ackermann}
Fix an integer $k \geq 0$.  There is a polynomial-time algorithm, which on input a word $w$  on $A_0^{\pm1}, \ldots, A_k^{\pm1}$, declares whether or not $w(0)$ represents an integer, and if so whether  $w(0) <0$, $w(0)=0$ or $w(0)>0$.  
\end{theorem}

(The manner in which $w(0)$ might fail to represent an integer is that as it is  evaluated from right to left, an $A_i^{\pm 1}$ is applied to an integer outside its domain.  Details are in Section~\ref{Ackermann function preliminaries}.  In fact our algorithm halts in  time bounded above by a polynomial of degree $4+k$---see Section~\ref{Ackermann specs}.  We have not attempted to optimize the degrees of the polynomial bounds on time complexity  here or elsewhere in this article.)

\subsection{The word problem and Dehn functions}
 
Our interest in Theorem~\ref{Ackermann} originates in group theory.  Elements of  a group $\Gamma$ with a  generating set  $A$ can be represented by words---that is,  products of elements of $A$ and their inverses.  To work with $\Gamma$, it is useful to have an algorithm which, on input a word,  declares whether that word represents the identity element in $\Gamma$. After all, if we can recognize when a word represents the identity, then we can recognize when two words represent the the same group element, and thereby begin to compute in $\Gamma$.   
The issue of whether there is such an algorithm is known as the \emph{word problem} for $(\Gamma,A)$ and was first posed by Dehn \cite{Dehn2, Dehn} in 1912.  (He did not precisely ask for an algorithm, of course, rather  `\emph{eine Methode angeben, um mit einer endlichen Anzahl von Schritten zu entscheiden...}'---that is, `\emph{specify a method to decide in a finite number of steps...}.')

Suppose a group $\Gamma$ has a finite presentation 
$$\langle \, a_1, \ldots, a_m \mid r_1, \ldots, r_n \, \rangle.$$
The \emph{Dehn function} $\Area : \N \to \N$ quantifies the difficulty of  a \emph{direct attack} on the word problem: roughly speaking $\Area(n)$ is the minimal $N$ such that if a word of length at most $n$ represents the identity, then it does so `as a consequence of' at most $N$ defining relations.

Here is some notation that we will use to make this more precise.  Associated to a set  $\set {a_1, a_2, \ldots  }$ (an \emph{alphabet}) is the set of   inverse letters   $\set{a^{-1}_1, a^{-1}_2, \ldots  }$.  The inverse map is the involution defined on $\set{a_1^{\pm 1}, a_1^{\pm 2}, \ldots }$ that  maps $a_i \mapsto a_i^{-1}$ and $a_i^{-1} \mapsto a_i$ for all $i$.  Write $w= w(a_1, a_2, \ldots )$ when $w$ is a word on the letters $a_1^{\pm 1},  a_2^{\pm 1}, \ldots$.  The inverse map extends to  words by sending 
$w = x_1\cdots x_s \mapsto x_s^{-1} \cdots x_1^{-1} = w^{-1}$ when each $x_i \in \set {a_1^{\pm 1}, a_2^{\pm 1}, \ldots}$.    
 Words $u$ and $v$ are \emph{cyclic conjugates} when $u = \alpha \beta$ and  $v =  \beta \alpha$ for some subwords $\alpha$ and $\beta$.      \emph{Freely reducing}  a word means removing all $a_j^{\pm 1}a_j^{\mp 1}$ subwords.  
For $\Gamma$ presented as above,  \emph{applying a relation} to a word $w = w(a_1, \ldots, a_m)$ means  replacing some subword $\tau$ with another subword $\sigma$ such that some   \emph{cyclic conjugate} of $\tau \sigma^{-1}$ is one of  $r^{\pm 1}_1, \ldots, r^{\pm 1}_n$.   
 
For a word $w$ representing the identity in $\Gamma$,  $\Area(w)$ is the minimal $N \geq 0$ such that there is a sequence of \emph{freely reduced} words $w_0, \ldots, w_N$ with $w_0$ the freely reduced form of $w$, and $w_N$ is the empty word, such that for all $i$, $w_{i+1}$ can be obtained from $w_i$ by \emph{applying a relation} and then \emph{freely reducing}.  
	The \emph{Dehn function} $\Area : \N \to \N$  
 is  defined by   $$\Area(n)  \ := \ \max  \set{ \, \Area(w) \mid \textup{words } w \textup{ with } \ell(w) \leq n \textup{ and } w=1 \textup{ in } \Gamma \,  }.$$ 

 This is one of a number of equivalent definitions of the Dehn function.  While a Dehn function is defined for a particular finite presentation for a group, its growth type---quadratic, polynomial, exponential etc.---does not depend on this choice.  Dehn functions are important from a geometric point-of-view and have been studied extensively.  There are many places to find background, for example \cite{BRS, Bridson6, BrH, Dehn,  Gersten, Gromov, RileyDehn, Sapir}.

 If $\Area(n)$ is bounded above by a recursive function $f(n)$, then there is a `brute force' algorithm to solve the word problem: to tell whether or not a given word $w$ represents the identity,   search through all the possible  ways of applying at most $f(n)$ defining relations and see whether one reduces $w$ to the empty word.  (There are finitely presented groups for which there is no algorithm to solve the word problem \cite{Boone, Novikov}.)
Conversely, when a finitely presented group admits an algorithm to solve its word problem, $\Area(n)$ is bounded above by a recursive function (in fact $\Area(n)$ \emph{is} a recursive function) \cite{Gersten6}.

There are finitely presented groups for which an extrinsic algorithm is far more efficient than this intrinsic brute-force approach.  A simple example is $$\Z^2 \ = \  \langle \, a,b  \, \mid \, ab=ba \, \rangle$$ (which has Dehn function $\Area(n) \simeq n^2$).  Given a word made up of the letters  $a^{\pm 1}$ and $b^{\pm 1}$, the extrinsic approach amounts to  searching exhaustively through all the ways of shuffling letters $a^{\pm 1}$ past letters $b^{\pm 1}$ to see if there is one which brings each  $a^{\pm 1}$ together with an $a^{\mp 1}$ to be cancelled, and likewise each   $b^{\pm 1}$ together with a $b^{\mp 1}$.  It is much more efficient to read through the word and check that the number of $a$ is the same as the number of $a^{-1}$, and the   number of $b$ is the same as the number of $b^{-1}$.

There are more dramatic examples where  $\Area(n)$ is a fast growing recursive function  (so the `brute force' algorithm succeeds but is extremely inefficient), but there are efficient ways to solve the word problem.
Cohen, Madlener \& Otto built the first examples. in a series of  papers   \cite{CohenWisdom, CMO, MO}   where Dehn functions  were first defined.  They designed their groups in such a way that the `intrinsic'  method of solving the word problem involves running a very slow algorithm which has been suitably `embedded' in the presentation.  But running this algorithm is pointless as it is constructed to halt (eventually) on all inputs and so presents no obstacle to the word representing the identity.   Their examples all  admit algorithms to solve the word problem in running times that  are at most $n \mapsto \exp^{(\ell)}(n)$ for some $\ell$.  But for each  $k \in \N$ they have examples which have Dehn functions  growing like $n \mapsto A_k(n)$.  Indeed, better, they have examples with Dehn function growing like $n \mapsto A_n(n)$.

Recently, more extreme examples were constructed by  Kharlampovich, Miasnikov \& Sapir~\cite{KMS}.  By simulating Minsky machines in groups,  for every recursive function $f: \N \to \N$, they construct a finitely presented group (which also happens to be residually finite and solvable of class 3) with Dehn function growing faster than $f$, but with word problem solvable in polynomial time.  

There are also  `naturally arising' groups which have fast growing  Dehn function  but an efficient (that is, polynomial-time)  solution to  the word problem.  A first example  is $$\langle \, a,b \, \mid \, b^{-1} a b =a^2 \, \rangle.$$  Its Dehn function grows exponentially (see, for example, \cite{BRS}), but the group admits a faithful matrix representation  $$a \mapsto \left(\begin{array}{cc}1 & 1 \\0 & 1\end{array}\right), \qquad b \mapsto  \left(\begin{array}{cc}1/2 & 0 \\0 & 1\end{array}\right),$$ and so it is possible to check efficiently when a word on $a^{\pm 1}$ and $b^{\pm 1}$ represents the identity by multiplying out the corresponding string of matrices.
   
 A celebrated 1-relator group due to Baumslag  \cite{Baumslag2} provides a more dramatic example: 
 \begin{equation*} \label{Baumslag's group}
 \langle \  a,b  \ \mid \   (b^{-1}a^{-1} b)  \, a \, (b^{-1}a b) = a^{2} \  \rangle.
 \end{equation*}
 Platonov~\cite{Platonov} proved its Dehn function grows like $n \mapsto \stackrel{\lfloor \log_2 n \rfloor}{\overbrace{\exp_2(\, \exp_2 \cdots (\exp_2}(1))} \cdots)$, where $\exp_2(n):=2^n$.  
(Earlier results in this direction are in \cite{Bernasconi, Gersten6, Gersten}.)
   Nevertheless, Miasnikov,  Ushakov \& Won \cite{MUW1} solve its word problem in polynomial time.  (In unpublished work  I.~Kapovich and Schupp   showed it is solvable in exponential time  \cite{SchuppPersonal}.)

Higman's group 
$$\langle \  a, b, c, d  \ \mid \   b^{-1}a b = a^2,  \,   c^{-1}b c = b^2,  \,  d^{-1}c d = c^2,  \,  a^{-1} d a = d^2  \ \rangle$$
 from \cite{Higman} is another  example.    Diekert, Laun   \& Ushakov \cite{DLU} recently gave a polynomial time algorithm for  its word problem  and, citing a 2010 lecture of Bridson, claim it too has Dehn function growing like a tower of exponentials.  

The groups we focus on in this article are yet more extreme `natural examples'.   They arose in the study of \emph{hydra groups} by  Dison \& Riley~\cite{DR} .  Let $$\theta: F(a_1, \ldots, a_k) \to F(a_1, \ldots, a_k)$$ be the  automorphism of the free group of rank $k$ such that $\theta(a_1)=a_1$ and  $\theta(a_i) = a_i a_{i-1}$ for $i=2, \ldots, k$.   
  The family
 \begin{equation*}  
G_k \ := \ \langle \  a_1, \ldots, a_k,  t \  \mid \  t^{-1} {a_i} t=\theta(a_i) \ \, \forall i>1   \ \rangle,\end{equation*}
are called \emph{hydra groups}. Take \emph{HNN-extensions}
 $$ \Gamma_k \ := \  \langle \ a_1,\ldots,a_k,t,p \ \mid \ t\inv a_i t = \theta(a_i), \ [p,a_it]=1   \ \, \forall i>1 \ \rangle$$
of $G_k$  where the stable letter $p$ commutes with all elements of the  subgroup  $$H_k \ := \   \langle a_1t, \ldots, a_kt \rangle.$$  It is shown in \cite{DR} that for $k =1, 2, \ldots$,  the subgroup $H_k$ is free of rank $k$ and  $\Gamma_k$
 has   Dehn function growing like $n \mapsto A_k(n)$. Here we prove that nevertheless:
 
 \begin{theorem} \label{WP theorem}
For all $k$, the word problem of $\Gamma_k$ is solvable in polynomial time.   
\end{theorem}
 
(In fact, our algorithm halts within time bounded above by a polynomial of degree $3k^2+k+2$---see Section~\ref{Conclusion}.) 
 
 \subsection{The membership problem and subgroup distortion}

\emph{Distortion} is the root cause of the  Dehn function of $\Gamma_k$ growing like $n \mapsto A_k(n)$.  The massive gap between Dehn function and the time-complexity of the word problem for $\Gamma_k$ is attributable to a similarly massive gap between a \emph{distortion function} and the time-complexity of a \emph{membership problem}.  Here are more details.    

Suppose $H$ is a subgroup of a group $G$ and $G$ and $H$ have finite generating sets $S$ and $T$, respectively.   So $G$ has a \emph{word metric}  $d_S(g,h)$, the    length of a shortest word on $S^{\pm1}$ representing $g^{-1}h$, and   $H$ has a word metric $d_T$ similarly.  

The \emph{distortion} of $H$ in $G$ is  
\[\Dist^{G}_{H}(n) \ := \  \max\{ \, d_T(1,g) \  | \ g \in H \text{ with } d_S(1,g) \le n \,  \}.\] 
(Distortion is defined here with respect to specific  $S$ and $T$, but their choices do not affect the qualitative growth of $\Dist^{G}_{H}(n)$.)   A fast growing distortion function signifies that $H$ `folds back on itself' dramatically  as a metric subspace of $G$.  

The \emph{membership problem} for $H$ in $G$ is to find an algorithm which, on input of a word on $S^{\pm 1}$, declares whether or not it represents an element of $H$.   

If the word problem of $G$ is decidable (as it is for all $G_k$, because, for instance, they are free-by-cyclic) and   we have a recursive upper bound on $\Dist^{G}_{H}(n)$, then  there  is a  brute-force solution to the membership problem for $H$ in $G$.  
If the input word $w$ has length $n$, then search through all words on $T^{\pm 1}$ of length at most $\Dist^{G}_{H}(n)$ for one representing the same element as $w$.   This is, of course, likely to be extremely inefficient, and especially so for   $H_k$ in $G_k$ as the distortion $\Dist^{G_k}_{H_k}$ grows like  $n \mapsto A_k(n)$.  Nevertheless:

\begin{theorem} \label{MP theorem} For all $k$, the membership problem for $H_k$ in $G_k$ is solvable in polynomial time.  \end{theorem}

(Our algorithm actually halts within time bounded above by a polynomial of degree $3k^2+k$---see Section~\ref{Conclusion}.)   We will use this   to prove Theorem~\ref{WP theorem}.

\subsection{The hydra phenomenon} \label{hydra phenomenon}

The reason   $G_k$  are named \emph{hydra groups} is that the extreme distortion of $H_k$ in $G_k$ stems from a string-rewriting phenomenon which is a reimagining of the battle between Hercules and the Lernean Hydra, a mythical beast which grew two new heads for every one Hercules severed. 
Think of a \emph{hydra} as a word $w$ on $a_1, a_2, a_3, \ldots$.
Hercules fights  $w$ as follows.  He  removes its first letter, then the remaining letters regenerate in that for all $i>1$,  each remaining $a_i$ becomes $a_ia_{i-1}$ (and each remaining $a_1$ is unchanged).  This repeats.  An induction on the highest index present shows that every hydra eventually becomes the empty word.  (Details are  in   \cite{DR}.)  Hercules is then declared victorious.   For example, the hydra $a_2a_3a_1$ is annihilated in $5$ steps:
$$a_2 a_3 a_1 \ \to \   a_3 a_2  a_1 \ \to \   a_2 a_1 a_1 \  \to \   a_1 a_1 \  \to \  a_1 \  \to \  \textit{empty word}.$$

Define $\H(w)$ to be the number of steps required to reduce a hydra $w$ to the trivial word (so $\H(a_3a_3a_1)=5$).  Then, for $k=1, 2, \ldots$, define functions $\H_k : \N \to \N$ by  $\H_k(n) = \H(a_k^n)$.  It is shown in \cite{DR} that $\H_k$ and $A_k$ grow at the same rate  for all $k=1, 2, \ldots$  since the two families exhibit  a similar recursion relation.

Here is an outline of the argument from \cite{DR} as to why $\Dist_{H_k}^{G_k}$  grows at least as fast as $n \mapsto \mathcal{H}_k(n)$ (and so as fast as  $n \mapsto A_k(n)$).  When $k \geq 2$ and $n \geq 1$, there is a reduced word $u_{k,n}$ on $\set{a_1t, \ldots, a_kt}^{\pm 1}$ of length  $\H_k(n)$ representing $a_k^n t^{\H_k(n)}$ in $G_k$ on account of the hydra phenomenon.  (For example, 
$u_{2,3} \ =  \ (a_2t)^2 (a_1t)(a_2 t) (a_1t)^3$  equals $a_2^3 t^7$ in $G_2$ since $a_2$, $a_2$, $a_1$, $a_2$, $a_1$, $a_1$, and $a_1$ are the $\H_2(3)=7$ initial letters removed by Hercules as he vanquishes the hydra $a_2^3$.)   
This can be used to show that in $G_k$  
$$a_k^{n} a_2 \  t a_1 \  a_2^{-1}  a_k^{-n} \ =  \  u_{k,n} \, (a_2t) \,  (a_1 t) \, ({a_2}t)^{-1} \, {u_{k,n}}^{-1}.$$
 The word on the left  is a product of length  $2n +4$ of the generators $\set{a_1, \ldots, a_n,t}^{\pm 1}$ of $G_k$ and that on the right is a product of length  $2 \H_k(n) +3$ of the generators $\set{a_1t, \ldots, a_kt}^{\pm 1}$ of $H_k$.  As $H_k$ is free of rank $k$ and this word is reduced, it is not equal to any  shorter   word on these generators.

\subsection{The organization of this article and an outline of our strategies}   
 We prove Theorem~\ref{Ackermann} in Section~\ref{Ackermann intro}.  
  Here is an outline of the algorithm we construct.  Given a word $w(A_0, \ldots, A_k)$ we attempt to    pass   to  successive new words $w'$ that are \emph{equivalent}  to $w$ in that $w'(0)$ represents an integer  if and only if $w(0)$ does, and when they both do, $w(0) = w'(0)$. These words are obtained by making substitutions  that, for instance, replace a letter $A_{i+1}$ in $w$   by a subword $A_i A_{i+1}A_0^{-1}$ (this substitution stems from the recursion defining Ackermann functions), or we delete a subword $A_i A_i^{-1}$ or $A_i^{-1} A_i$.  The aim of these changes is to eliminate all the letters $A_1\inv, \ldots, A_k\inv$ in $w$, as these present the greatest obstacle to checking whether such a word represents an integer. 
Once no $A_1\inv, \ldots, A_k\inv$ remain in $w'$, when calculating $w'(0)$ letter-by-letter starting from the right, only $A_0^{\pm 1}$ can trigger decreases in absolute value. So to determine the sign of $w'(0)$ it suffices to evaluate $w'(0)$ letter-by-letter from the right, stopping if the integer calculated ever exceeds the length of $w'$.  

In order to reach such a $w'$ we `cancel' away letters $A_i^{-1}$ with some $A_i$ somewhere further to the right in the word.  We do this by manipulating suffixes of the form $A_i^{-1} u A_iv$ such that $u=u(A_0, \ldots, A_{i-1})$. Such suffixes either admit substitutions to make  a similar suffix with the $A_i^{-1}$ and $A_i$ eliminated, or they can be recognized  not to evaluate to an integer because $u$ cannot carry the element  $A_iv(0) \in \text{Img }A_{i}$  to another element of  $\text{Img }A_{i}$ since the gaps between elements of $\text{Img }A_{i}$ are large.  

 A number of difficulties arise.  For instance, there are exceptional cases when replacing $A_{i+1}$  by $A_i A_{i+1}A_0^{-1}$ fails  to preserve validity.  Another issue is that we must ensure that the process terminates, and so we may, for example, have to introduce an $A_i$ `artificially' to cancel with some $A_i^{-1}$.  
 
 To show that our algorithm halts in polynomial time, we argue that the lengths of the successive words  remain bounded by a constant times $\ell(w)$ (the length of $w$), and integer arithmetic operations performed only ever involve integers of absolute value at most  $3\ell(w)$.  
 
The group theory in this paper (specifically Theorem~\ref{MP theorem}) actually requires  a variant  of Theorem~\ref{Ackermann} (specifically, Proposition~\ref{Thm: Psi}).   Accordingly, in Section~\ref{phi function preliminaries} we introduce a  family  of functions which we call $\psi$-functions, which are closely related to Ackermann functions, and  we adapt the earlier results and proofs to these.  (We believe Theorem~\ref{Ackermann}   is of intrinsic interest because Ackermann functions are well-known and  efficient computation with this form of highly compressed integers is novel.  This is why we do not present Proposition~\ref{Thm: Psi} only.)

We give a polynomial-time solution to the membership problem for $H_k$ in $G_k$    in Section~\ref{MP problem intro}, proving Theorem~\ref{MP theorem}.  Here is an outline of our algorithm.  Suppose $w(a_1, \ldots, a_k, t)$ is a word representing an element of $G_k$.  To tell whether  or not $w$ represents an element of $H_k$, first collect all the $t^{\pm 1}$  at the front by shuffling them to the left through the word, applying $\theta^{\pm 1}$ as appropriate to the intervening $a_i$ so that the element of $G_k$ represented does not change.  The result is a word $t^r v$ where $\abs{r}  \leq \ell(w)$ and $v=v(a_1, \ldots, a_k)$ has length at most a constant times  $\ell(w)^k$.  Then carry the $t^r$ back through $v$ working from left to right,  converting (if possible) what lies to the left of the power of $t$ to a word on the generators  $a_1t, \ldots, a_kt$ of $H_k$.  Some examples can be found in Section~\ref{alg examples}.  

 The power of $t$ being carried along will vary as this proceeds and, in fact, can get extremely large as a result of the hydra phenomenon.  So instead of keeping track of the power directly, we record it as a word on $\psi$-functions.  Very roughly speaking, checking whether this process ever gets stuck (in which case $w \notin H_k$) amounts to checking whether an associated $\psi$-word is valid.  If the end of the word is reached, we then have a word on    $a_1t, \ldots, a_kt$  times some power of $t$, where the power is represented by a  $\psi$-word.  We then determine whether  or not $w \in H_k$   by checking whether or not that $\psi$-word represents $0$.   Both tasks can be accomplished suitably efficiently thanks to Proposition~\ref{Thm: Psi}.  

A complication is that the power of $t$ is not carried through from left to right one letter at a time.  Rather, $v$ is partitioned into subwords which we call \emph{pieces}.  These pieces are determined by the locations of the $a_k$ and $a_k^{-1}$ in $v$.  Each contains at most one $a_k$ and at most one $a_k^{-1}$, and if the $a_k$ is present in a piece, it is the first letter of that piece, and it the $a_k^{-1}$ is present, it is the last letter.   The power of $t$ is, in fact, carried through one piece at a time.  Whether it can be carried through a piece $a_k^{\varepsilon_1} u a_k^{-\varepsilon_2}$ (here, $\varepsilon_1, \varepsilon_2 \in \set{0,1}$ and $u=u(a_1, \ldots, a_{k-1})$ is reduced)  depends on $u$ in a manner that can be recursively analyzed by decomposing $u$ into pieces with respect to the locations of the  $a_{k-1}^{\pm 1}$ it contains.   The main technical result behind the correctness of our algorithm is the `Piece Criterion' (Proposition~\ref{Prop: Piece Criterion}), which also serves to determine whether a  power  $t^r$ can pass through a piece $\pi$---that is, whether $t^r\pi = \sigma t^{s}$ for some $\sigma \in H_k$ and some $s \in \Z$---and, if it can, how to represent $s$ by an $\psi$-word.
  
Reducing Theorem~\ref{WP theorem}  to Theorem~\ref{MP theorem} is relatively straight-forward.  It requires little more than a standard result about HNN-extensions, as we will explain in Section~\ref{MP to WP}.

\subsection{Comparison with power circuits and straight-line programs}  

Our methods compare and contrast with those used   to solve the word problem for Baumslag's group in \cite{MUW1} and Higman's group  in \cite{DLU}, where  \emph{power circuits} are the key tool.  Power circuits provide concise  representations of integers.  Those of size $n$ represent (some) integers up to size a height-$n$ tower of powers of $2$.  There are efficient algorithms to perform  addition,  subtraction,   and multiplication and division by $2$ with power-circuit representations of integers, and to declare which of two  power circuits represents the larger integer.  
 
 We too  use concise representations of large integers, but in place of power circuits we use   strings of Ackermann functions.  These have the advantage that they may  represent much larger integers.  After all,  $A_3(n) = \exp_2^{(n-1)}(1)$ already produces a tower of exponents, and the higher rank Ackermann functions grow far faster.
However, we are aware of fewer efficient algorithms to perform operations with strings of Ackermann functions than are available for power circuits: we only have Theorem~\ref{Ackermann}. 

Our methods also bear comparison with the work of Lohrey, Schleimer and their coauthors \cite{HL,  HLM,  Lohrey2, Lohrey,  LohreyMono, LS, Schleimer} on efficient computation in groups and monoids where words are given in compressed forms using   \emph{straight-line programs}  and are compared and manipulated using   polynomial-time algorithms due to   Hagenah, Plandowski and Lohrey.  For instance Schleimer obtained polynomial-time algorithms solving the word problem for free-by-cyclic groups and automorphism groups of free groups and   the membership problem for the
handlebody subgroup of the mapping class group in \cite{Schleimer}.

\section{Efficient calculation with Ackermann-compressed integers} \label{Ackermann intro}

\subsection{Preliminaries} \label{Ackermann function preliminaries}  

Let $\N = \set{0, 1, 2,\ldots }$.
\emph{Ackermann functions}  $A_0, A_1 : \Z \to \Z$ and $A_i: \N \to \N$ for $i=2, 3, \ldots$ are  defined recursively  by 
\begin{enumerate}
\renewcommand{\labelenumi}{(\roman{enumi}) }
\item $A_0(n) = n+1$ for all $n \in \Z$,
\item $A_1(n) = 2n$  for all $n \in \Z$,
\item $A_i(0) = 1$  for all $i \geq 2$, and
\item $A_{i+1}(n+1) =   A_iA_{i+1}(n)$ for all $n \ge 0$ and all $i \geq 1$.  
\end{enumerate}
Our choices of $\Z$ as the domains for $A_0$ and $A_1$ and our definition of $A_0$ represent  small variations on the standard definitions of Ackermann functions,  reflecting  the definitions of the functions $\psi_i$ to come in Section~\ref{MP problem intro}.  The following table, showing some values of $A_i(n)$, can be constructed by first inserting the $i=0,1$ rows and then $n=0$ column, and then filling in the subsequent rows left-to-right according to the recurrence relation.   
$$\begin{array}{r | c c c c c c c  c }
                           & 0 & 1 & 2 & 3 & 4 &  \cdots  & n &  \cdots   \\ \hline
A_0                        & 1 &  2 &  3 &  4 &  5 &  \cdots  & n+1 &  \cdots \\ 
A_1                        & 0 & 2 & 4 & 6 & 8 &  \cdots  & 2n  &  \cdots \\  
A_2                       & 1 & 2 & 4 & 8 & 16   &  \cdots  &   2^{n} &  \cdots \\  
A_3                        & 1 & 2 & 4 & 16 &  65536  & \cdots  &  \left. \parbox{9mm}{${2^{\mbox{2}^{{\reflectbox{$\ddots$}}}}}^{\mbox{2}}$}  \right\} n &  \cdots  \\ 
A_4                        & 1 & 2 & 4 & 65536 &    \left. \parbox{9mm}{${2^{\mbox{2}^{{\reflectbox{$\ddots$}}}}}^{\mbox{2}}$}  \right\} 65536  & \cdots  &   &   \\  
\vdots                & \vdots & \vdots & \vdots & \vdots  & \vdots  &     &   &   \\   
\end{array} $$

For all  $i \geq 2$ and $n \geq 1$,   $A_i(n)  =   A_{i-1}^{n}(1)$ by repeatedly applying $(\textup{iv})$ and using $A_i(0)=1$.  So for  all $n \geq 0$, $A_2(n)   =   2^n$ and $A_3(n)$ is a $n$-fold iterated power of $2$, in other words, a tower of powers of $2$ of height $n$.    The recursion $(\textup{iv})$ causes the functions' extraordinarily fast growth.  Indeed, because of the increasing nesting of the recursion,  the $A_i$ represent the successive graduations in a hierarchy of all primitive recursive functions due to Grzegorczyk.  
 
The functions $A_i$ are all strictly increasing and hence injective (see Lemma \ref{Properties of A_k}).  So they have partial inverses:
\begin{itemize}
\item[(I)] $A_0\inv:\Z\to \Z$ mapping $n\mapsto n-1$,
\item[(II)] $A_1\inv:2\Z \to \Z$ mapping  $n\mapsto  n/2$, and
\item[(III)]  $A_i\inv:\text{Img }A_{i} \to \N$ for all $i>1$.
\end{itemize}

Parts (\ref{A_k of 1}--\ref{A_k9} ) of the following lemma are adapted from Lemma~2.1 of \cite{DR} with  modifications to account for the fact that $A_0$ is defined as $n\mapsto n+1$ here rather than $n\mapsto n+2$.  Part \eqref{A_i gap}  quantifies the spareness of the image of $A_2, A_3, \ldots$ in a way that will be vital to our proof of Theorem~\ref{Ackermann} (specifically, in our proof the correctness of the subroutine $\alg{BasePinch}$).  It will tell us  that  if $u =u(A_1,  \ldots, A_{k-1})$ and  $u A_k(n) \in \Img \, A_k$ but $u A_k(n)  \neq A_k(n)$,  then $\ell(u)$ must be relatively large.

\begin{lemma} \label{Properties of A_k}
 
   \begin{align}
    A_i(1) \ &= \ 2  && \forall i \geq 0, \label{A_k of 1} \\
    A_i(2) \ &= \ 4  && \forall i \ge 1, \label{A_k of 2} \\
    A_i(n) \ &\leq \ A_{i+1}(n)  && \forall i \geq 1; n \geq 0,  \label{A_k7} \\
    A_i(n) \ &< \ A_i(n+1)   && \forall i,n \geq 0, \label{A_k5} \\
    n \ &\leq \ A_i(n)  && \forall i,n \geq 0, \label{A_k2} \\
    \intertext{\textup{(}with equality   in \eqref{A_k2} if and only if $i=1$ and $n  = 0$\textup{)}}
    A_i(n) + A_i(m) \ &\leq \ A_i(n + m)  && \forall i,n,m \geq 1, \label{A_k3} \\
    A_i(n) + m \ &\leq \ A_i(n+m)  &&\forall i,n,m \geq 0, \label{A_k9}  \\
    |A_i(n)-A_i(m)|   \ & \ge \ \dfrac12 A_i(n) & &  \forall i\ge 2  \text{ and }  n\ne m. \label{A_i gap}
  \end{align}

\begin{proof}
  Equations \eqref{A_k of 1} and \eqref{A_k of 2} follow from $A_{i+1}(n+1) =   A_iA_{i+1}(n)$  by induction on $i$.    It is easy to check that \eqref{A_k7} holds if $i=1$ or if $n=0$ and that \eqref{A_k5} and \eqref{A_k2} hold if $i=0$, if $i=1$ or if $n=0$.   It is clear \eqref{A_k3} holds if $i=1$. 
 The inequality \eqref{A_k9} holds if $i=0$, $i=1$ or $m=0$.  The inductive arguments for the above inequalities are then identical to the corresponding ones in Lemma~2.1 of \cite{DR}.  For \eqref{A_i gap}, note that the result is true when $i=2$ as $A_2(n) = 2^n$ for all $n \in \N$ and, given how each of the successive rows is constructed from those   preceding them, it follows that it is true for all $i \geq 2$.  
\end{proof}
\end{lemma}

When a word $w = w(A_0, \ldots, A_k)$ is non-empty, we let $\rank(w)$ denote  the maximum $i$ such that $A_i^{\pm 1}$ occurs in $w$ and $\eta(w)$ denote the number of $A_1\inv, \ldots, A_k\inv$ in $w$. 
For example, if $w= A_4\inv A_3 A_0\inv A_1\inv A_2$,  then $\rank(w)=4$ and $\eta(w) = 2$.

As we   said in Section~\ref{1.1}, strings of Ackermann functions offer a means of representing integers.
For $x_1, \ldots, x_n \in \{A_0^{\pm 1},\ldots,A_k^{\pm 1}\}$, we say the  word $w=x_nx_{n-1}\cdots x_1$ is \emph{valid} if    $x_mx_{m-1}\cdots x_1(0)$ is defined for all $0\le m \le n$.  That is,  if we evaluate  $w(0)$ by  proceeding through $w$ from right to left applying  successive  $x_i$, we never encounter the problem that we are trying to apply $x_i$ to an integer outside its domain, and so $w(0)$ is a well-defined integer. 

For example,  $w:=A_2\inv A_1 A_1 A_0$ is valid, and $w(0) = \log_2 (2\cdot 2\cdot (0+1)) = 2$. But $A_2 A_0\inv$ and $A_1 A_1\inv A_0$  are not valid because $A_0\inv(0) = -1$ is not in $\mathbb{N}$ (the domain of $A_2$) and because $A_0(0)=1$ is not in $2\Z$ (the domain of $A_1\inv$).   

  For $m \in \Z$, the \emph{sign} of $m$, denoted  $\sgn(m)$, is $-$, $0$, or $+$ depending on whether  $m<0$, $m=0$, or $m>0$, respectively.  So Theorem~\ref{Ackermann} states that there is a polynomial-time algorithm to test validity of $w(A_0, \ldots, A_k)$ and, when valid, to determine the sign of $w(0)$.

   We say $w(A_0, \ldots, A_k)$ and $w'(A_0, \ldots, A_k)$ are \emph{equivalent} and write $w \sim w'$ when $w$ and $w'$ are either both invalid, or are both  valid and $w(0)=w'(0)$.

\subsection{Examples and general strategy} \label{Ackermann Examples}

 We fix an integer $k \geq 0$ throughout the remainder of this article.

We will motivate and outline our design of our algorithm \alg{Ackermann} by means of some examples.  The   details of  \alg{Ackermann} and it subroutines (which we refer to parenthetically below)   follow in  Section~\ref{Ackermann specs}.  

First consider   the case where the word $w(A_0, \ldots, A_k)$ in question satisfies  $\eta(w)=0$---that is, contains no  $A_1\inv, \ldots, A_k\inv$.  Such $w$ are not hard to handle because, to check validity of $w$, we only need to make sure that no $A_i$ in $w$ with $i\ge 2$ takes a negative input when $w(0)$ is evaluated. (Such $w$ are handled by the subroutine \alg{Positive}.)  Here is an example.

\begin{Example}\label{Ex: Positive}
Let $w= A_0^{-6} A_1 A_0\inv A_5 A_0^{-4} A_2 A_1 A_2 A_0$, which is a word of length $17$ with $\eta(w)=0$.  We can evaluate directly working from right to left that, if valid, $w(0) =  A_0^{-6} A_1 A_0\inv A_5(12)$.  At this point we are reluctant to calculate $A_5(12)$ as it is enormous, and instead  recognize that $A_5(12)$ is larger than $\ell(w)=17$ ($\alg{Bounds}$), which as we will explain in a moment we can do suitably quickly.  We then deduce that $w$ is valid and $w(0) >0$, because $A_0\inv$ are the only letters further to the left which would lower the value, were the evaluation to continue, and there cannot be enough of them to reach $0$ or a negative number. 
\end{Example}

In general, if $\eta(w)=0$, our algorithm starts evaluating $w(0)$ working right to left. 
Let $w_j$ denote the length-$j$ suffix of $w$.  
The only letters in $w$ which could decrease absolute value are   $A_0^{\pm 1}$, so if $|w_j(0)|>\ell(w)$ for some $j$ and $w$ is valid, then $\sgn(w_j(0))=\sgn(w(0)).$
Moreover, if $\abs{w_j(0)} > \ell(w)$, then the only way $w$   fails to be valid is if $w_j(0) <0$ and the prefix of $w$ to the left of $w_j$ contains one of $A_2, A_3, \ldots$.  So after either exhausting $w$ or reaching such a $j$ and then scanning the remaining letters in $w$, the algorithm can halt and decide whether or not $w(0)$ is valid, and if so its sign.   
 
This technique adapts to compare $w(0)$ with a constant --

\begin{Example} \label{Ex: compare with constant}
Take $w$ as in Example \ref{Ex: Positive}.
We see that $w(0)>2$ by applying the same technique to find that $w(0)-2 = A_0^{-2} w(0)> 0$. Here, the size of $A_5(12)$ still dwarfs $\ell(A_0^{-2}w) = 19$, so the   computation carried out is essentially the same.
\end{Example}

So, how do we determine that $A_5(12)>17$ or, indeed, $A_5(12)>19$ for  Examples~\ref{Ex: Positive} and \ref{Ex: compare with constant}?
 The recursion   $A_{i+1}(n+1) = A_i A_{i+1}(n)$ implies that $\Img\, A_i \subseteq \Img\, A_2$ for all $i\ge 2$.  Suppose we wish to know whether $A_i(n)$ is less than some constant $c$.  The cases $i=0, 1$ are easy to handle as $A_0(n) = n+1$ and $A_1(n) = 2n$ for all $n$.  So are the cases $n=0,1,2$ as $A_i(0)=1$, $A_i(1)=2$, and $A_i(2)=4$ for all $i$. As for other values of $i$ and $n$, the recursion allows a subroutine (\alg{Bounds})  to list the $i\ge 2$ and $n\ge 3$ for which $A_i(n) < c$.

For instance, to  find the $i\ge 2$ and $n\ge 0$ for which $A_i(n) < 17$, first  calculate $A_2(n) = 2^n$ for all $n$ for which $A_2(n) <17$, filling in the first row of the following table.
\[ \begin{matrix}  & n=0 & n=1 & n=2 & n=3 & n=4 &  \\
							 A_2 & 1   &  2  & 4   & 8   &  16 &  \\
							 A_3 & 1  & 2 & 4 & 16 &  & \\
							 A_4 & 1 & 2 & 4 &  & & \\
 \end{matrix}\]
Now fill the  table one row at a time. We start with $A_3(0)=1$ and $A_3(1)=2$, and then  $A_3(2) = A_2A_3(0) = A_2(1) = 2$.  Then $A_3(2) = A_2A_3(1)$, which is $4$ because, as we already know, $A_3(1)=2$ and $A_2(2) = 4$. Similarly, $A_3(3) = 16$. 
And $A_3(4) = A_2A_3(3) =A_2(16)$, which must be greater than $16$ since  $A_2(16)$ is not in the table. 
We carry out the same process for $A_4$.  We discover that  $A_4(3) = A_3A_4(2) = A_3(4)$ is  at least $17$ since 
$ A_3(4)$ is not already in the table.  At this point we halt, reasoning that $A_j(3) \geq A_i(3) \geq  17$ for all $j>i$ (see Lemma~\ref{Properties of A_k}).  

  \alg{Ackermann}'s strategy, on input a word $w$,    is to reduce to the case $\eta(w)=0$ by progressing  through a sequence of equivalent words, facilitated  by:

\begin{lemma}\label{Lem: OneToZero Equiv}
Suppose $u=u(A_0, \ldots, A_k)$ and $v=v(A_0, \ldots, A_k)$. The following equivalences hold  if $v$ is invalid or if $v$ is valid and satisfies the further conditions indicated:
\begin{align*}
u A_{i+1} v  \ & \sim \ u A_i A_{i+1} A_0\inv v & v(0) >0 \text{ and } i\geq 1, \\ 
u A_{i+1}^{-1} v \ & \sim \ u A_0 A_{i+1}^{-1} A_{i}^{-1} v   & v(0) >1 \text{ and } i\geq 1,  \\ 
 u A_i^{-1} A_i  v \ &  \sim \  uv & v(0) \geq 0 \text{ and } i\geq 0.  
\end{align*}
\end{lemma}

\begin{proof} 
If $v$ is invalid, then any word  with suffix $v$ is invalid, so $uA_{i+1}v \sim uA_iA_{i+1}A_0\inv v$ and $uA_{i+1}\inv v \sim u A_0 A_{i+1}\inv A_i v$. 

Assume $v$ is valid. If $v(0)>0$, then $A_0\inv v(0)\ge 0$ so that $A_{i+1}v$ and $A_iA_{i+1} A_0\inv v$ are valid words and by the recursion defining the functions, 
\[ A_{i+1}v(0) \ =  \ A_i A_{i+1}(v(0)-1)  \ = \  A_iA_{i+1}A_0\inv v(0). \]
Thus $uA_{i+1}v \sim uA_iA_{i+1}A_0\inv v$ since their validity is  equivalent to the validity of $u$ on   input   $A_{i+1}v(0)$. 

Suppose $v(0)>1$.
If $v(0) = A_{i+1}(c)$ for some $c\in\Z$, then $c>0$ because $i\ge 1$, so $v(0) = A_{i}A_{i+1}(c-1)$.
Conversely, $v(0) = A_iA_{i+1}(c-1)$ implies $c\ge 1$. Thus
\[A_0 A_{i+1}\inv A_i\inv v(0) \ = \  c \  =  \ A_{i+1}\inv v(0),\]
and $uA_0 A_{i+1}\inv A_i\inv v  \sim u A_{i+1}\inv v$ because their validity is equivalent to validity of $u$ on input $A_{i+1}\inv v (0)$. 

That $u A_i^{-1} A_i  v \sim uv$ under the given assumptions is apparent because  the condition $v(0) \geq 0$ ensures $v(0)$ is in the domain of $A_i$, given that $i \geq 2$.
\end{proof} 

We will  frequently make tacit  use of this fact, which is immediate from the definitions:

\begin{lemma} 
If  $w(A_0, \ldots, A_k)$ and $w'(A_0, \ldots, A_k)$  can be expressed as $w=uv$ and $w'=uv'$ for some equivalent suffixes $v \sim v'$, then  $w \sim w'$
\end{lemma}

Here is an outline  of what  \alg{Ackermann} does on input a valid word $w$.  A description of how \alg{Ackermann} checks the hypotheses of Lemma~\ref{Lem: OneToZero Equiv}   and what   it does  when they fail is postponed until the end of the outline. \begin{enumerate}
\renewcommand{\labelenumi}{\arabic{enumi}. }
\item   \label{Ackermann outline step 1} 
Locate the rightmost $A_r\inv$ in $w$ for which $r\ge 1$.  We aim to eliminate this letter,  to get a word $w'$ with $\eta(w') < \eta(w)$ and $w \sim w'$ by `cancelling' it with an $A_{r}$ that lies somewhere to its right and with no higher rank letters in between.   However there may be no such $A_r$, in which case we  manufacture one.  Accordingly ---
\begin{enumerate} 
\item  If every letter to the right of $A_r\inv$ is of rank less than $r$, then  append either $A_0\inv A_r$ if $r>1$ or $A_1$ if $r=1$ to create an equivalent word ending in  $A_r$ .  
\item \label{1b}  Locate the first letter $A_{r'}$  that lies to the right of our $A_r\inv$ and has  $r' \geq r$. If $r'>r$,  substitute $A_{r'-1} A_{r'} A_0^{-1}$ for this $A_{r'}$, then $A_{r'-2} A_{r'-1} A_0^{-1}$ for the resulting $A_{r'-1}$, and so on, as per Lemma~\ref{Lem: OneToZero Equiv} until we have created an $A_r$ (\alg{Whole}). 
\end{enumerate}
 Thereby, obtain a word equivalent to $w$ which has suffix  $s = A_r\inv  u A_r v$ for some  $u$ and $v$ with $\eta(u)=\eta(v)=0$ and $\Rank(u)<r$.  (\alg{Reduce}.)

\item We now invoke a subroutine  ($\alg{Pinch}_r$) which will either declare $s$ (and so $w$) invalid, or will convert $s$ to an equivalent word $A_0^l v$ for some  $l \in \Z$. \label{Pinch Step in Ackermann Outline}   

Suppose first that $ \rank(u) =r-1 >0$.  We will  explain how to eliminate an $A_{r-1}$ from $u$.  On repetition, this will give a word  $A_0^m A_r\inv \tilde{u} A_r v \sim s$ such that $\Rank(\tilde{u}) \leq r-2$.  ($\alg{CutRank}_r$.)  
\begin{enumerate}
\item \label{CutRank Outline Step 1} Find the leftmost $A_{r-1}$ in $s$ and write  $$s  \ = \  A_{r}\inv u' A_{r-1} u'' A_r v$$ where $\Rank(u')< {r-1}$ and $\Rank(u'')\le {r-1}$. 
Substitute $A_0 A^{-1}_r A^{-1}_{r-1}$ for $A_r^{-1}$ as per Lemma~\ref{Lem: OneToZero Equiv} to give 
$$A_0 A^{-1}_r A^{-1}_{r-1} u' A_{r-1} u'' A_r v \ \sim \ s.$$
\item Apply $\alg{Pinch}_{r-1}$  to the suffix $A^{-1}_{r-1} u' A_{r-1} u'' A_r v$ to give an equivalent word $A_0^{l'}u'' A_r v$ for some $l' \in \Z$.  Thereby get $$A_0 A^{-1}_rA_0^{l'}u'' A_r v \ \sim \ s.$$

\item Likewise eliminate an $A_{r-1}$ from $u''$ in  $A^{-1}_rA_0^{l'}u'' A_r v$, and so on, until we arrive at   
 $$A_0^m A_r\inv \tilde{u} A_r v \ \sim \ s$$  such that $m \in \Z$ and  $\Rank(\tilde{u}) \leq r-2$.
\end{enumerate}

To reduce the rank of  the subword between the  $A_r\inv$ and the $A_r$  further we manufacture an $A_{r-1}\inv$ and an  $A_{r-1}$ and then proceed recursively.  Accordingly ---

\begin{enumerate}
\addtocounter{enumii}{3}
\item \label{the FinalPinch Step} Substitute for $A_r\inv$ and $A_r$ as per Lemma~\ref{Lem: OneToZero Equiv}  to get 
$$A_0^{m} \, (A_0 A_r\inv A_{r-1}\inv) \, \tilde{u} \, (A_{r-1} A_r A_0\inv) \, v \ \sim \ s.$$

\item Call $\alg{Pinch}_{r-1}$ on the suffix $A_{r-1}\inv  \tilde{u}  A_{r-1} A_r A_0\inv    v$ to obtain  
$$A_0^{m +1}  A_r\inv   A_0^{l''}    A_r A_0\inv    v \ \sim \ s$$
for some $l'' \in \Z$ 
($\alg{FinalPinch}_r$).
\end{enumerate}
\item 
Eliminate $A_r\inv$ and $A_r$ from the suffix $A_r\inv   A_0^{l''}    A_r A_0\inv    v$ using a method we will shortly explain via Example~\ref{Ex: Pinch0} to give  an equivalent suffix  $A_0^{l'''}   A_0\inv    v$   for some $l''' \in \Z$  ($\alg{BasePinch}$).    Thereby, if $w'$ is the word obtained from $w$ by substituting the suffix beginning with the final  $A_r\inv$ with  $A_0^{m +1}  A_0^{l'''}   A_0\inv    v$, then $w \sim w'$ and $\eta(w') < \eta(w)$, as required.      \label{Ackermann outline step 4}
\item Repeat steps \ref{Ackermann outline step 1}--\ref{Ackermann outline step 4} until we have an equivalent word with no $A_1\inv, \ldots, A_k\inv$.  
\item Use the strategy (\alg{Positive}) from Example~\ref{Ex: Positive} above. 
\end{enumerate}

To make legitimate substitutions as per Lemma~\ref{Lem: OneToZero Equiv} in Steps~\ref{1b}, \ref{CutRank Outline Step 1}, and \ref{the FinalPinch Step}, we  have to examine certain suffixes. 
In every instance   we are:
\begin{enumerate}
\item \label{1sub}  either substituting  $A_i A_{i+1} A_0^{-1}$ for an $A_{i+1}$, in which case we have to check that  the suffix $v$ (which has $\eta(v) =0$)  after that $A_{i+1}$ has  $v(0) >0$,
\item \label{2sub}  or substituting $A_0 A_{i+1}^{-1} A_i^{-1}$ for an $A_{i+1}^{-1}$, in which case we have to check that  the suffix $v$ after that  $A_{i+1}^{-1}$ (which again has $\eta(v) =0$) has   $v(0) >1$.
\end{enumerate}
So validity of $v$ and the hypothesis $v(0) >0$ or $v(0)>1$ (and indeed whether $v(0) < 0$, whether $v(0)=1$, or whether $v(0) \leq 0$,  which we will soon also  need) can be checked in the manner of Examples~\ref{Ex: Positive} and \ref{Ex: compare with constant}, and if  $v$ is invalid, then  $w$ is invalid.

Suppose, then,  we are in Case~i, $v$ is valid, but $v(0) \leq 0$.   
\begin{itemize}
\item If  $i > 0$ and  $v(0) < 0$, then $A_{i+1}v$, and so $w$, is invalid.  
\item  If $i >1$ and $v(0)=0$, then  $A_{i+1}v(0)=1$ and so, instead of making the planned substitution, the suffix  $A_{i+1}v$ can be replaced by the equivalent $A_iv$.  
\item If $i=1$ and  $v(0)=0$, then we  have a suffix $A_2v$ which we replace by the equivalent $A_0A_1(v)$.  
\item When $i=0$, no substitution is necessary because $A_1\inv u A_1 v$ is valid if and only if $u(0)$ is even.
If so $u=A_0^l$ for some even $l$ and $A_1\inv u A_1 v$ can be replaced by the equivalent $A_0^{l/2} v$.  
\end{itemize}

 Suppose, on the other hand, that  we are in Case~ii, $v$ is valid, but $v$ is valid and $v(0) \leq1$. The algorithm actually only tries to make substitutions for $A_{i+1}\inv$ when the input word has suffix $A_{i+1}\inv u A_{i+1} v_0$  for some subwords $u$ and $v_0$ such that $\eta(u)= \eta(v_0)=0$ and $\Rank(u)<i+1$ (and $v \equiv u A_{i+1} v_0$).
It proceeds as follows:    
\begin{itemize}
\item  If $v(0)=1$ and $i>0$, output the equivalent $A_0^{-v_0(0)}v_0$.
\item If $i=0$ use the fact that $A_1\inv u A_1 v_0$ is valid if and only if $u(0)$ is even. If $u(0)$ is even, $u=A_0^l$ for some even integer $l$ replace the suffix $A_1\inv u A_1 v_0$  by the  equivalent  $A_0^{l/2}v_0$.
\item If $v(0)\le 0$, then $A_{i+1}\inv v$ is invalid.
\end{itemize}
(In Case~ii, it is not obvious that outputting $A_0^{-v_0(0)} v_0$ is better than simply returning the empty word to represent zero. 
However, the inductive construction of the algorithm requires that the output word retain a suffix $v_0$.)

\begin{Example}\label{Ex: Pinch}
Let $w= A_0A_2\inv A_1A_0^2A_2A_0$.  A quick direct calculation shows $w$ is valid and $w(0)=4$, but here is how our    \alg{Ackermann}  handles it.  
\begin{enumerate}
\renewcommand{\labelenumi}{\arabic{enumi}. }
\item First aim to eliminate the $A_2\inv$ (the subroutine \alg{Reduce}).  Look to the right of the $A_2\inv$ for the first subsequent letter (if any) of rank at least $2$, namely  the $A_2$.   
\item Try to  `cancel' the $A_2\inv$ with the $A_2$ ($\alg{Pinch}_2$)  ---
\begin{enumerate}
\item Reduce the rank of the subword $ A_1A_0^2$ between $A_2\inv$ and $A_2$ as follows ($\alg{CutRank}_2$).
\begin{enumerate}
\item Use the technique of Example~\ref{Ex: Positive} (\alg{Positive}) to check that the suffix $A_1A_0^2A_2A_0$  is valid and $A_1A_0^2A_2A_0(0) > 1$.  So, by by Lemma \ref{Lem: OneToZero Equiv}, we can legitimately substitute $A_0^2 A_2\inv A_1\inv$ for $A_2\inv$   to obtain   $$A_0 A_2\inv A_1\inv A_1A_0^2A_2A_0 \ \sim \ w.$$
\item Cancel the  $A_1\inv A_1$ (strictly speaking, this is done by calling $\alg{CutRank}_2$ on $A_2\inv A_1\inv A_1A_0^2A_2A_0$, and then $\alg{Pinch}_2$) to give $$ A_0^2A_2\inv A_0^2A_2A_0 \ \sim \ w.$$
\end{enumerate}
\item Next follow Step~\ref{the FinalPinch Step} from the outline above. 
Seek to replace the subword $A_2\inv A_0^2 A_2$ by an appropriate power of $A_0$ (by calling $\alg{FinalPinch}_2$ on the suffix $s:=A_2\inv A_0^2A_2A_0$) as follows. 
\begin{enumerate}
\item Check $A_0(0)\ne 0$ and $A_0^2 A_2 A_0(0) \ne 1$, so we can   substitute  $A_0 A_2\inv A_1\inv$ for $A_2\inv$  and $A_1A_2 A_0\inv$ for  $A_2$ in $s$ (as per Lemma \ref{Lem: OneToZero Equiv}) to get  $$A_0A_2\inv A_1\inv A_0^2 A_1 A_2 A_0\inv A_0 \  \sim \  s.$$ 
\item    
Convert the subword $A_1\inv A_0^2  A_1$ to a power of $A_0$ (by calling $\alg{Pinch}_1$ on $A_1\inv A_0^2 A_1 A_2 A_0\inv A_0$, which calls $\alg{BasePinch}_1$ since the subword between the $A_1\inv$ and the $A_1$ is a power of $A_0$).    It replaces  $A_1\inv A_0^2 A_1$ by $A_0$ (which is appropriate because $(2x+2)/2 = x+1$) to give $$s' \ := \ A_0 A_2\inv A_0 A_2 A_0\inv A_0 \  \sim \  s.$$
\item  The exponent sum of the $A_0$ between $A_2\inv$ and $A_2$ in $s'$ is $1$. 
(Were it  non-zero and less than half of $A_2A_0\inv A_0(0)=1$,  then  $A_2A_0\inv A_0(0)$ would be too far from another  integer  in the image of $A_2(n)$  for $s'$ to be valid.)  But, in this case,   we evaluate $A_2\inv A_0A_2 A_0\inv A_0(0)$ by   computing that it is $2$ directly from right to left, and then  evaluating $A_2\inv(2) = 1$ (by calling $\alg{Bounds}(2\ell(w))$).  So $A_2\inv  A_0A_2 A_0\inv A_0(0) = 1$, and we can conclude that $$s' \ \sim \ A_0^2 A_0\inv A_0.$$
(Preserving the suffix $A_0\inv A_0$  appears unnecessary here,  but it reflects the recursive design of the algorithm.)
\end{enumerate}
So  $$w'  \ :=  \ A_0^4   A_0\inv A_0 \ \sim \  w.$$  

\end{enumerate}
\item Now $\eta(w') = 0$.  So  evaluate $w'$ from right-to-left  in the manner of Example~\ref{Ex: Positive}  (\alg{Positive}) and declare that $w$ is valid and $w(0)>0$.
\end{enumerate}
\end{Example}

In our next example,  the input word has the form $A_r\inv u A_{r'} v$ with $\eta(u)=\eta(v)=0$ and $\Rank(u)<r < r'$. As there is no $A_r$ with which we can `cancel' the $A_r\inv$, we manufacture one by using Lemma~\ref{Lem: OneToZero Equiv} to create an $A_r$ to the left of the $A_{r'}$ and thereby reduce to a situation similar to the preceding example.
This   example also serves to explain how we resolve the special case $A_r\inv A_0^l A_r v$ which is crucial for avoiding explicit computation of large numbers.

\begin{Example}\label{Ex: Pinch0}
Set $w=A_2\inv A_0^{-2} A_3A_0^{100}$. 
\begin{enumerate}
\renewcommand{\labelenumi}{\arabic{enumi}. }
\item Identify the rightmost $A_i\inv$ with $i\ge 1$, namely the  $A_2\inv$. Scanning to the right of $A_2\inv$, the first $A_i$ we encounter with $i \geq 2$ is the $A_3$. (Send $w$ to \alg{Reduce}, which calls \alg{Whole}.)
\item Use techniques from Example~\ref{Ex: Positive} (\alg{Positive}) to check that $A_0^{100}(0)>0$.  So   we can substitute $A_2A_3A_0\inv$ for $A_3$, as per  Lemma~\ref{Lem: OneToZero Equiv}, to obtain   \[w_0  \ :=  \  A_2\inv A_0^{-2} A_2 A_3 A_0\inv A_0^{100} \ \sim \ w.\]
\item We check we can make substitutions as in Lemma~\ref{Lem: OneToZero Equiv} for $A_2\inv$ and $A_2$ 
to give  $$w_1 \ :=  \ (A_0A_2\inv A_1\inv ) \, A_0^{-2}  \, (A_1A_2 A_0\inv) \,  A_3A_0\inv A_0^{100} \ \sim \  w.$$
  (Run $\alg{CutRank}_2$ on $w_0$  which does nothing as $\Rank(u) < 1$, and then start running $\alg{FinalPinch}_2(w_0)$.)  
\item We now want to reduce the rank of the subword between the $A_2\inv$ and $A_2$ to zero ($\alg{Pinch}_2$), and so we ($\alg{BasePinch}_1$) process the suffix $$A_1\inv A_0^{-2} A_1A_2A_0\inv A_3 A_0\inv A_0^{100}$$   to replace  
$A_1\inv A_0^{-2} A_1$ by $A_0\inv$    
giving  $$w_2 \ := \ A_0A_2\inv A_0^{-1} A_2 A_0\inv A_3A_0\inv A_0^{100} \ \sim \ w$$ (the equivalence being because $(2x-2)/2=x-1$). 
\item  
Now the subword of $w_2$ between $A_2\inv$ and $A_2$ has rank $0$   (which causes $\alg{Pinch}_2$  to end and we return to $\alg{FinalPinch}_2$, which in turn invokes $\alg{BasePinch}_2$).   As $A_2$ is the function $\N \to \N$ mapping $n \mapsto 2^n$,  if $A_0^{z} A_2 A_0\inv A_3A_0\inv A_0^{100}(0)$ is in the domain of $A_2\inv$ for some $z\in\Z \ssm \{0\}$, then the   large gaps between powers of $2$ ensure that $2|z| \geq  A_2 A_0\inv A_3A_0\inv A_0^{100}(0)$. In the case of $w_2$, we have $z=-1$ and so 
we see that $w_2$  is invalid by checking that $ A_2 A_0\inv A_3A_0\inv A_0^{100}(0) >2$.    
We can do this  efficiently in the manner of Example~\ref{Ex: compare with constant} by noting that $A_3A_0\inv A_0^{100}(0)$ exceeds the threshold  $ \ell(A_2 A_0\inv A_3A_0\inv A_0^{100}) + 2 = 106$.  So we declare $w$ invalid.
\end{enumerate}
\end{Example} 

A major reason \alg{Ackermann} halts in polynomial time, is that as it manipulates  words, it does not  substantially increase their lengths.  One subroutine it employs, \alg{Bounds},  takes an integer as its input.  All   others input  a word $w$ and output  an equivalent word $w'$ and in every case but two,  $\ell(w') \leq  \ell(w)$.  The exceptions are the subroutines  $\alg{Whole}$  and $\alg{Reduce}$, where   $\ell(w')\leq  \ell(w) +2k$.  But they  are each called at most $\eta(w) \leq \ell(w)$ times when \alg{Ackermann} is run on input $w$, so they do not cause length to blow up.   The way this control on length is achieved is that while  length is increased  by making substitutions as per Lemma~\ref{Lem: OneToZero Equiv}, those increases are offset by a process of replacing a  suffix of the form $A_r\inv u A_r v$ (with $\eta(u)=\eta(v)=0$ and $\Rank(u)<r$) by an equivalent suffix of the form $A_0^l v$ with $\abs{l} \leq \ell(u)$.  

The technique of exploiting the large gaps between powers of $2$ to sidestep direct calculation applies to all words of the form $A_r\inv A_0^z A_r v$ where $r \geq 2$ and $z\neq 0$, after all the gaps in the range of $A_r$ grow even faster when $r >2$.  
In Lemma~\ref{Properties of A_k}~\eqref{A_i gap}, we showed that if $l\in\Z$ is non-zero and $A_r\inv A_0^lA_rv$ is valid, then $2|l|\ge A_r v(0)$. 
This condition can be efficiently checked if $\eta(v)=0$. If $2|l|\ge A_rv(0)$, direct computation of the value of $A_r\inv A_0^l A_rv(0)$ (using $\alg{Bounds}(2|l|)$) becomes efficient relative to $\ell(w)$ since $|l|\le \ell(w)$.

Our final example is a circumstance where we are unable to make substitutions because a hypothesis of Lemma \ref{Lem: OneToZero Equiv} fails.
\begin{Example}\label{OneToZero Example}
Let $w=A_3\inv A_0\inv A_3A_0$. Direct calculation shows that $w$ is valid and $w(0)=0$, but here is how our algorithm proceeds.
\renewcommand{\labelenumi}{\arabic{enumi}. }
\begin{enumerate}
\item As before, we identify the  $A_3\inv$, the subsequent $A_3$, and the subword  $A_0\inv$ that separates them. (Call $\alg{Pinch}_3$ on $A_3\inv u  A_3 v$ where $u= A_0\inv$ and $v=A_0$.) 
\item First we check that $A_0$ is valid and $A_0 (0) \geq 0$ and so is in the domain of $A_3$.  
Then we check that  $A_0\inv A_3A_0$ is valid (a necessary condition for validity of $w$) and $A_0\inv A_3A_0(0)\ge 0$ (a necessary condition to be in the domain of $A_3\inv$).  (In both cases we use $\alg{Positive}$.)
\item We notice that there are no  $A_1^{\pm 1}$ or $A_2^{\pm 1}$ between $A_3\inv$ and $A_3$ to remove.  ($\alg{Pinch}_3$ runs  $\alg{CutRank}_3(w)$, which does not change $w$.)    
\item We seek  to substitute $A_0 A_3\inv A_2\inv$ for $A_3\inv$  and $A_2A_3 A_0\inv$ for  $A_3$. ($\alg{Pinch}_3$ calls $\alg{FinalPinch}_3$.)  But,  by calculating  that $A_0\inv A_0\inv A_3 A_0(0)=0$  (which is done by calling $\alg{Positive}(A_0\inv A_0\inv A_3A_0)$), we discover that $A_0\inv A_3 A_0(0) = 1$, violating a hypothesis of Lemma \ref{Lem: OneToZero Equiv}.  

\item Invoke a subroutine ($\alg{OneToZero}$) for this special case. 
We calculate  the integer $m = v(0)$ by testing whether $A_0^{-m} v(0)=0$ starting with $m=1$ and incrementing $m$ by $1$ until we obtain a string equal to zero.   In this example $v=A_0$, and so $m=1$.   We  return $A_0^{-m}v = A_0\inv A_0$ where $A_0^{-m}v (0) = 0 = A_r\inv(1) = A_r\inv v(0)$. It would be simpler to return the empty word, but the recursive structure of \alg{Pinch} requires the output of an equivalent word whose suffix is $v$.
\item $\eta(A_0\inv A_0) = 0$, so the algorithm explicitly affirms validity,  finds the sign of $A_0\inv A_0(0)$, and returns $0$. (\alg{Positive}.)
\end{enumerate}
\end{Example}

\subsection{Our algorithm} \label{Ackermann specs}

We continue to have an integer $k \geq 0$ fixed and work with words on the alphabet $A_0^{\pm 1}, \ldots, A_k^{\pm 1}$.  The   polynomial time bounds we establish in this section all depend on $k$.

 \begin{figure}[ht]
    \psfrag{I}{\small{\textbf{input} $w$}}
    \psfrag{A}{\small{\alg{Ackermann}}}
        \psfrag{B}{\small{\alg{Bounds}}}
    \psfrag{R}{\small{\alg{Reduce}}}
    \psfrag{P}{\small{\alg{Positive}}}
        \psfrag{q}{\small{$\cdots$}}
    \psfrag{D}{\small{\parbox{28mm}{\textbf{declare} whether or not $w$ is valid and, if so, whether $w(0)<0$, $w(0)=0$, or $w(0) >0$}}}
        \psfrag{k}{\small{$\alg{Pinch}_k$}} \psfrag{a}{\small{$\alg{Pinch}_{k-1}$}} \psfrag{2}{\small{$\alg{Pinch}_{2}$}}   \psfrag{1}{\small{\parbox{3cm}{$\alg{Pinch}_{1}  \\ {} \qquad = \alg{BasePinch}$}}}
   \psfrag{l}{\small{$\alg{CutRank}_k$}} \psfrag{b}{\small{$\alg{CutRank}_{k-1}$}} \psfrag{c}{\small{$\alg{CutRank}_{2}$}}   \psfrag{d}{\small{$\alg{Pinch}_{1}$}}
   \psfrag{m}{\small{$\alg{FinalPinch}_k$}} \psfrag{g}{\small{$\alg{FinalPinch}_{k-1}$}} \psfrag{e}{\small{$\alg{FinalPinch}_{2}$}}   \psfrag{f}{\small{$\alg{FinalPinch}_{1}$}}
       \psfrag{y}{\tiny{\gray{$\eta(w) \neq 0$}}}
       \psfrag{z}{\tiny{\gray{$\eta(w) = 0$}}}
        \psfrag{x}{\tiny{\gray{$\eta(w)$ lowered}}}
      \centerline{\epsfig{file=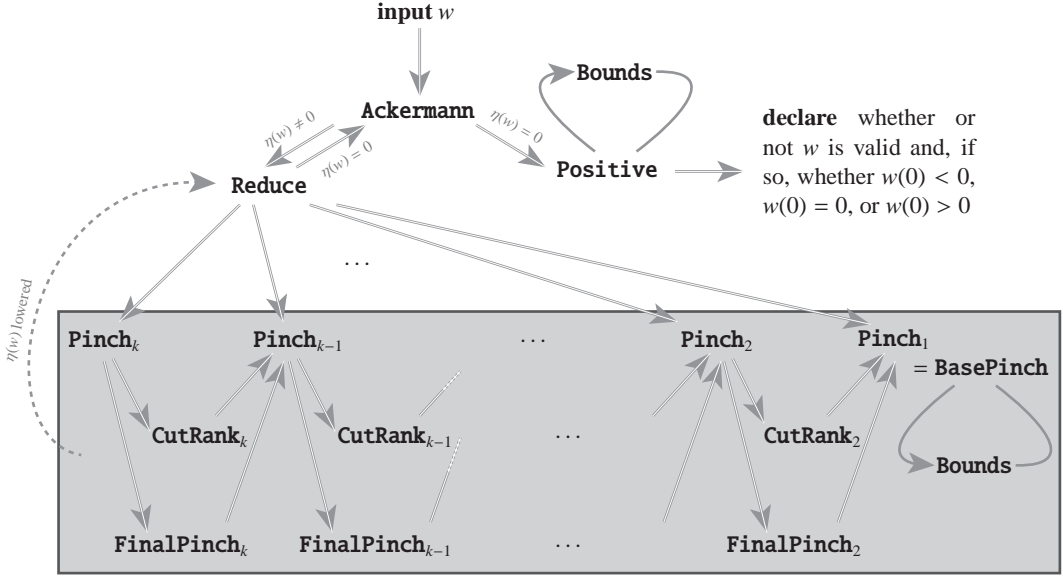}} \caption{An outline of the design of \alg{Ackermann}, indicating which routines call which other routines.  Any routine may   declare $w$ invalid and halt the algorithm.  From $\alg{Reduce}$, the algorithm progresses to $\alg{Pinch}_r$, where $r$ is the subscript of the rightmost of $A_1\inv, \ldots, A_k\inv$ to remain in $w$.  The progression through the  $\alg{Pinch}_i$,  $\alg{CutRank}_i$, and $\alg{FinalPinch}_i$ (shown boxed) is involved (and not apparent from the diagram) but ultimately decreases $\eta(w)$ by one.    A further routine  $\alg{OneToZero}$ (which handles certain special cases)   does not appear, but is called by a number of the routines shown.  $\alg{Positive}$ also serves as a routine, but only its role in providing the final step in the algorithm is indicated in the figure.} 
\label{flowchart}
\end{figure}
 
Our first  subroutine follows the procedure explained in Section~\ref{Ackermann Examples}, so we only sketch it   here.  

 \begin{algorithm}[h]
    \caption{ --- \alg{Bounds}. \newline 
    $\circ$ \  Input $\ell\in \N$  (expressed in binary). \newline
   $\circ$ \ Return a list of all the (at most $(\log_2\ell)^2$) triples of integers $(r,n,A_r(n))$ such that $r \geq 2$, $n \geq 3$, and $A_r(n) \leq \ell$.  \newline
  $\circ$ \ Halt  in  time $O(\ell)$. }
    \begin{algorithmic}[3]
     		\State list all values of $A_2(n)=2^n$ for which $2\le n \le \lfloor \log_2 \ell \rfloor$ \label{recall something}  
		\State recall (from Lemma~\ref{Properties of A_k}) that   $A_i(2)=4$ for all $i\ge 2$
		\State use the recursion $A_{i+1}(n+1) = A_i A_{i+1}(n)$ to calculate all $A_r(n)\le \ell$ for $r\ge 3$ and $n\ge 3$, halting when $A_r(3)> \ell$ 
    \end{algorithmic}
  \end{algorithm}

\begin{proof}[Correctness of $\alg{Bounds}$.]
$\alg{Bounds}$ generates its list of triples by first listing the at most  $\lfloor \log_2(\ell)\rfloor$ triples $(2,n,A_2(n))$  such that $n \geq 3$ and $A_2(n) = 2^n \leq \ell$, which it can do in time  $O((\log_2 \ell)^2)$   since $\ell$ is expressed in binary. It then reads  through
this list and uses the recurrence relation (and the fact that $A_3(2)=4$) to list all the $(3,n,A_3(n))$
for which $n \geq 3$ and $A_3(n) \leq \ell$.   It then uses those to list the $(4,n,A_4(n))$ similarly, and so on.    For all $r\ge 3$, $A_r(3) = A_{r-1}(4) \ge 2A_{r-1}(3)$, and so $A_r(3)\ge 2^r$.   So the triples $(r,n,A_r(n))$ outputted by \alg{Bounds} all have $r \leq  \lfloor \log_2 \ell \rfloor$.  As $r$ increases, there are fewer $n$ such that   $A_r(n) \leq \ell$.  So the complete list  \alg{Bounds} outputs comprises at most $(\log_2\ell)^2$  triples of binary numbers each recorded by a binary string of length at most $\log_2\ell$, and it is generated in time $O(\ell)$. 
\end{proof}
 
(In fact, $\alg{Bounds}$ halts in time polynomial in  $\log_2 \ell$, but we are content with the  $O(\ell)$ bound because other terms will dominate our cost-analyses of the routines that call $\alg{Bounds}$.)

\begin{remark}  $\alg{Bounds}$ does not give any $(r,n,A_r(n))$ for which  $A_r(n) \geq \ell$ but $r \leq 1$ or $n \leq 2$.  Nevertheless, such triples  require negligible computation to identify.  After all,  $A_r(0)=1$,  $A_r(1)=2$ and $A_r(2)=4$ for all $r \geq 1$ and $A_0(n) = n+1$ and  $A_1(n) = 2n$  for all $n \in \Z$. 
\end{remark}

 \begin{algorithm}[h]
    \caption{--- \alg{Positive}. \newline 
  $\circ$ \  Input  a word $w= x_nx_{n-1}\cdots x_1$   where  $x_1, \ldots, x_n \in \set{A_0^{\pm 1}, A_1, \ldots, A_k}$.  \newline  
  $\circ$ \  Return $\Invalid$ when $w$ is invalid and $\sgn(w(0))$ when $w$ is valid.   \newline 
  $\circ$ \  Halt  in  time  $O(\ell(w)^3)$.}
    \label{Alg: Positive}
    \begin{algorithmic}[3]
		\State \textbf{run} $\alg{Bounds}(n)$ 
		\State  evaluate   $x_1(0)$, then $x_2x_1(0)$, and so on until \\
		 \qquad \textbullet \ either $w(0)$ has been evaluated \\
		 \qquad \textbullet \ or some $x_i...x_1(0) > n$ (checked by consulting the output of $\alg{Bounds}(n)$) \\ 
		 \qquad \textbullet \ or some $x_i...x_1(0) <-n$   (that is, $x_i \neq A_0^{\pm 1}$ and $x_i...x_1(0) < 0$)    \\ 	
		 \qquad \textbullet \ or some $x_i...x_1$ is found to be invalid (that is, $x_i \neq A_0^{\pm 1}$ and $x_i...x_1(0) < 0$)    	
		 \State then, respectively, \textbf{return} \\
		 \qquad \textbullet  \ $\textbf{sgn}(w(0))$  \\
		 \qquad \textbullet  \  $\textbf{sgn}(w(0)) = +$ \\
		 \qquad \textbullet \ if   $x_{i+1}, \ldots, x_n  \notin \set{A_2,  \ldots, A_k}$,  then $\textbf{sgn}(w(0)) = -$, else  \textbf{invalid} \\
		 \qquad \textbullet \ \textbf{invalid}   	
    \end{algorithmic}
  \end{algorithm}

\begin{proof}[Correctness of  \alg{Positive}]
As $w$ is a word  on  $A_0^{\pm 1}, A_1, \ldots, A_k$ (that is, $\eta(w)=0$), 
   decreases in absolute value only occur in increments of $1$ as $w(0)$ is evaluated from right to left.   
The domains of $A_0$, $A_0^{-1}$ and $A_1$ are $\Z$,  and of $A_2, A_3, \ldots$ are $\N$,  so $w$ is invalid only  when some $A_i$ with $i \geq 1$ meets a negative input. 
If the threshold, $+n$, is exceeded, then $w$ must be valid and  $w(0)>0$, as subsequent  letter-by-letter evaluation could never reach a negative value.  
If  $x_i...x_1(0) <-n$ for some $i$ (which is easily tested as it can only first happen when $x_i$ is $A_0^{-1}$ or $A_1$), then $w$ is valid if and only if none of the subsequent letters are $A_2,  \ldots, A_k$; moreover, if $w$ is valid, then $w(0)<0$.   If $w$ is exhausted, then the algorithm has fully calculated $w(0)$ (and $|w(0)|<n$) and has confirmed $w$ as valid.

\alg{Positive}  calls  \alg{Bounds} once with input $n =\ell(w)$, which produces its list of at most $(\log_2n)^2$ triples in time $O(n)$.   
The thresholds employed in \alg{Positive} ensure that it performs arithmetic operations (adding one, doubling, comparing absolute values) with integers of absolute value at most $n$.  Each such operation takes   time $O(n^2)$,  so they and  the necessary  searches of the output of  \alg{Bounds}  take time $O(n^3)$.   
\end{proof}
 
 Our next subroutine is  the  $\Rank(u)=0$ case of $\alg{Pinch}_r$, to come.

 \begin{algorithm}[h]
\caption{ --- $\alg{BasePinch}$.   
 \newline $\circ$ \ Input a word  $w=A_r\inv u  A_r v$ with $r \geq 1$, $u=u(A_0)$, $v = v(A_0, \ldots, A_k)$ and $\eta(v) = 0$.
  \newline $\circ$ \ Either return  that $w$ is invalid,  or return  a valid word $w'=A_0^{l'}v \sim w$ such that $\ell(w')\le \ell(w)-2$. 
   \newline $\circ$ \ Halt  in time $O(\ell(w)^4)$.}
\label{Alg: Pinch_0}
\begin{algorithmic}[3]
	\State \textbf{set} $l := u(0)$   (so $A_0^l$ is $u$ with all $A_0^{\pm 1} A_0^{\mp 1}$ subwords removed and $A_r\inv A_0^l A_r v \sim w$) 
	\State{\textbf{if} $\alg{Positive}(A_rv)=\Invalid$, \textbf{halt} and \textbf{return} \textbf{invalid} }  
	\State  \textbf{if} $r\ge 2$ and $v(0)< 0$ (checked using $\alg{Positive}$), \textbf{halt} and \textbf{return invalid} \label{cleanup v check}
	\State{\textbf{if}  {$l=0$},   \textbf{halt} and  \textbf{return} $w': = v$} \label{l=0 line}	
	\State{\textbf{if}  {$r=1$},   \textbf{halt} and  \textbf{return} $w': = A_0^{l/2} v$  or \textbf{invalid} depending on whether $l$ is even or odd} \label{r=1 line}		

	\State
	\State we now have $l \neq 0$ and $r>1$

		\State \textbf{run} $\alg{Positive}(A_0^{l}A_rv)$   to determine if $A_0^l A_rv(0) \leq 0$ (so outside the domain of $A_r^{-1}$) \label{first test line}
		\State \qquad \textbf{if} so, \textbf{halt} and \textbf{return} \textbf{invalid} 
		\State \textbf{run} $\alg{Positive}(A_0^{-2|l|} A_rv)$ to determine whether $A_rv(0)> 2|l|$ \label{second test line}
		\State \qquad \textbf{if} so, \textbf{halt}  and \textbf{return}  \label{sign line} 
			\State
			\State{we now have that $0 \leq v(0) \leq \abs{l}$ and $0 < A_r v(0)  \leq  2|l| $ and $ A_r v(0)  + l  \leq  3|l|$ } \label{3l bound line}
				\State calculate   $v(0)$ 
				by running $\alg{Positive}(A_0^{-i}v)$ for $i=0, 1, \ldots, \abs{l}$     \label{0- line}     \label{Line: Pinch0 2Pos}
				\State \textbf{run} $\alg{Bounds}(3\abs{l})$ 
				\State search the output of  $\alg{Bounds}(3\abs{l})$ to find $A_r v(0)$  
				\State \textbf{set}  $m := A_rv(0)+l$   
				\State search the output of $\alg{Bounds}(3\abs{l})$ for $c$ with $A_r(c) =m$  (so $c= A_r\inv A_0^l  A_r v(0) = w(0)$)  \label{the search line}
				\State \qquad \textbf{if} such a $c$ exists, \textbf{halt} and \textbf{return} $w' := A_0^{c- v(0)}v$ \label{readback line}
				\State \qquad \textbf{else}  \textbf{halt} and \textbf{return} \textbf{invalid}
\end{algorithmic}
\end{algorithm}

\begin{proof}[Correctness of  $\alg{BasePinch}$]
The idea is that when $w$ is valid, either $l=0$ or the sparseness of the image of $A_r$ implies that $l$ is large enough that $w(0)$ can be calculated efficiently.  Here is why the algorithm runs as claimed.

\begin{itemize}
\item[\ref{cleanup v check}:]  If $v(0)<0$, then $w$ is invalid.
\item[\ref{l=0 line}:]  If $r \geq 2$, then  $A_r\inv A_r v\sim v$ by Lemma~\ref{Lem: OneToZero Equiv}.
\item[\ref{r=1 line}:]  Since $A_1$ is the function $n \mapsto 2n$, the parity of $A_0^l  A_r v(0)$ is the parity of $l$ when $r=1$, and determines the validity of $w$.  

\item[\ref{first test line}, \ref{second test line}:]  We know $A_0^{l}A_rv$ and $A_0^{-2|l|}A_rv$ are valid at these points because $A_rv$ is valid.

\item[\ref{sign line}:]   Let $q=v(0)$.  
	For all $p\ne q$ we have  $|A_r(q)-A_r(p)|\ge \frac12A_r(m)$ by Lemma~\ref{Properties of A_k}~\eqref{A_i gap}, and so $|A_r(q)-A_r(p)| > |l|$. If 
	$A_r\inv A_0^l A_r v$ is valid, then there exists $p\in \N$  such that
	$A_r(p) = A_0^lA_rv(0)= l + A_r(q)$, but then $|A_r(p)-A_r(q)|=|l|$ for some $p \neq q$ (since $l \neq 0$), contradicting  $|A_r(q)-A_r(p)|>l$. 
	Thus $w$ is invalid. 

\item[\ref{3l bound line}]   The reason  $0 < A_r v(0)$ is that $r >1$ and so   $\Img A_r$ contains only positive integers.  And   $A_r v(0) \leq 2 \abs{l}$ because of lines \ref{second test line} and \ref{sign line}.   It follows that $v(0) \leq \abs{l}$ because $2v(0) = A_1 v(0) \leq A_r v(0) \leq 2 \abs{l}$.  And $v(0) \geq 0$ since $v(0)$ is in the domain of $A_r$, which is $\N$ when $r >1$.     We have $A_0^lA_r v(0) \leq 3|l|$ here because $A_r v(0)  \leq 2\abs{l}$ and  so  $A_0^l A_r v(0)  \leq l +  2\abs{l}$.  
 
\item[\ref{the search line}:] 	If $m =  A_rv(0)  + l =  A_0^l A_rv(0)$ is in the domain of $A_r\inv$,  then $m >0$.  And, from line~\ref{3l bound line}, we know $m \leq 3 \abs{l}$, so this will find $c$ if it exists.  If no such $c$ exists, $w$ is invalid.

\item[\ref{readback line}:]   $A_0^{c-v(0)}v(0)  =   c   =  A_r\inv (l+A_rv(0))   =   A_r\inv A_0^l A_rv(0)$.
	
\end{itemize}

 We must  show that $\ell(w')\le \ell(w)-2$.  In the cases of lines~\ref{l=0 line} and \ref{r=1 line}, this is immediate, so suppose $r\ge 2$.  
As for line~\ref{readback line}, we will show that  $|c-v(0)|\le |l|$, from which the result will immediately follow.   

First suppose $l\ge 0$. By Lemma~\ref{Properties of A_k} and the fact that $v(0) \geq 0$, we have $A_r(v(0) + l)\ge A_r(v(0)) + l$.  
 So $v(0)+l\ge A_r\inv (A_r v(0) + l)=c$.
So $c-v(0)  \leq l = |l|$.  And $0 \leq c-v(0)$ because $A_r(c) = A_r(v(0))+l \ge A_r(v(0))$. So $|c-v(0)|\le |l|$, as required. 
   
Suppose, on the other hand, $l <0$. Then $$c \ = \ A_r\inv A_o^l A_rv(0) \ \leq \ A_r\inv A_r v(0) \ = \ v(0)$$ and so $\abs{c - v(0)}   =  v(0) -c$. But then  $\abs{c - v(0)} \leq v(0)$ because $v(0),c \geq 0$.  
So if $v(0) + l \leq 0$, then $\abs{c - v(0)} \leq - l = \abs{l}$, as required.   Suppose instead that $v(0) + l > 0$.
We have  that $A_r(v(0)+l) \le A_r(v(0)) + l$  because  $A_r(p-m) \leq A_r(p) -m$ by Lemma \ref{Properties of A_k}~\eqref{A_k9} for all $p \ge m \geq 0$.   So $v(0)+l \leq A_r\inv (A_r(v(0)) + l) = c$.  So $l \le c-v(0)$.  And $c-v(0) < 0$ because $  A_r(c) = A_rv(0) +l  < A_rv(0)$. So $|c-v(0)| \le |l|$, again as required.

Next we explain why the integer calculations performed by the algorithm involve integers of absolute value at most $3 \ell(w)$.  
	The algorithm calls  \alg{Positive} on words of length at most $3\ell(w)$, and so (by the properties of \alg{Positive}  established), each time it is called, $\alg{Positive}$  calculates with integers no larger than $3\ell(w)$.   On input $3\abs{l} \leq 3 \ell(w)$, $\alg{Bounds}$ calculates with integers    of absolute value at most $3 \ell(w)$.   
	The only remaining integer manipulations concern $m, l, 2\abs{l}, A_r v(0)$, all of which have absolute value at most $3 \ell(w)$.

Finally, that  $\alg{BasePinch}$ halts in time $O(\ell(w)^4)$ is straightforward given the previously established cubic and linear halting times for  $\alg{Positive}$ and  \alg{Bounds}, respectively, and the following facts. 
It may add a pair of  positive binary  numbers each at most $2\ell(w)$, may  determine the parity of a number of absolute value at most $\ell(w)$, and may  halve  an even positive  number less than $\ell(w)$.
It calls $\alg{Positive}$ at most $\abs{l} +3 \leq \ell(w)+3$ times, each time on input a word of length at most $2\ell(w)$. It calls  \alg{Bounds}  at most once---in that event the input to \alg{Bounds} is a non-negative integer that is at most $3\ell(w)$ and the output of \alg{Bounds} is searched at most twice and has  size  $O((\log_2 \ell(w))^2)$. 
\end{proof}

\begin{algorithm}[H]
\caption{ --- $\alg{OneToZero}$. 
 \newline $\circ$ \ Input a valid word $w=A_r\inv uA_r v$ with $\eta(u)=\eta(v)=0$, $u\ne \epsilon$, $uA_r v(0) = 1$ and $r \geq 2$.
 \newline $\circ$ \ Return a  word $A_0^{-v(0)}v \sim w$ of length at most $\ell(w)-2$. 
  \newline $\circ$ \ Halt  in time $O(\ell(w)^4)$. }
\label{Alg: OneToZero}
\begin{algorithmic}[3]
\State \textbf{run} $\alg{Positive}(A_0^{-m}v)$ for $m=0, 1, ...$ \textbf{until} it declares that $A_0^{-m}v=0$  \label{the until}
\State \textbf{halt} and  \textbf{output} $A_0^{-m} v$  \label{onetozero endline}
\end{algorithmic}
\end{algorithm}

\begin{proof}[Correctness of  $\alg{OneToZero}$]
{ \quad} \\ \vspace*{-6mm} 
\begin{itemize}
\item[\ref{the until}:] As $w$ is valid, $v(0)$ is in the domain of $A_r$, which is $\N$ as $r \geq 2$.  So    $m = v(0)$ will eventually be found.  
\item[\ref{onetozero endline}:] $w(0) = A_r\inv (1)  =0$ and so $A_0^{-m} v \sim w$ as required, since $A_0^{-m} v(0)=0$.
\end{itemize}
  Since $\eta(u)=0$, the only letter $u$ may contain which decreases the value in the course of evaluating $uA_rv(0)$  is $A_0\inv$.  So, as $uA_rv(0)=1$ and $A_rv(0)\ge v(0)+1$, there must be at least $v(0)$ letters  $A_0\inv$ in $u$.   So 
$\ell(u) \geq  v(0)$. So  $\ell(A_0^{-v(0)}v)  \le \ell(w)-2$, as required.  

$\alg{OneToZero}$ calls $\alg{Positive}$ $m=v(0) \leq \ell(u) \leq \ell(w)$  times, each time on input of length at most $\ell(w)$. So, by the established properties of $\alg{Positive}$,  it halts in time $O(\ell(w)^4)$.   
\end{proof}

The input $w$ to $\alg{OneToZero}$ necessarily has $w(0)=0$, so it would seem it should just  output the empty word rather than   $A_0^{-v(0)}v$.  However, $\alg{OneToZero}$ is used by  $\alg{Pinch}_r$, which we will describe next and whose inductive construction requires the suffix $v$.

$\alg{Pinch}_r$ for $r \geq 1$ is a family of subroutines which we will construct alongside further families  $\alg{CutRank}_r$ and $\alg{FinalPinch}_r$ for $r \geq 2$. $\alg{Pinch}_{r-1}$ is a subroutine of  $\alg{CutRank}_r$  and of  $\alg{FinalPinch}_r$.  $\alg{CutRank}_r$ and $\alg{FinalPinch}_r$ are subroutines of  $\alg{Pinch}_r$.  
It may appear that we could discard $\alg{CutRank}_r$ and use  $\alg{FinalPinch}_r$ instead, by expanding $\alg{FinalPinch}_r$  to allow inputs with $\rank(u) = r-1$ and expanding $\alg{Pinch}_r$ to allow inputs where $\rank(u)=r$.  But this  would cause problems with maintaining the suffix $v$.

 		\begin{algorithm}[h]
\caption{ --- $\alg{Pinch}_r$ for $r \geq 1$.  
 \newline $\circ$ \ Input a word $w=A_r\inv u A_r v$ with $\eta(u)=\eta(v)=0$ and $\Rank(u)\le r-1$.
 \newline $\circ$ \ Either return that $w$ is invalid,  or return a valid word $w'=A_0^{l'}v \sim w$ such that $\ell(w')\le \ell(w)-2$. 
  \newline $\circ$ \ Halt in $O(\ell(w)^{4+(r-1)})$  time.}
\label{Alg: Pinch}

\begin{algorithmic}[3]
  \State \textbf{if} $r=1$  \textbf{run} $\alg{BasePinch}(w)$ and then \textbf{halt}  \label{r=1 case}
	\State \textbf{run} $\alg{Positive}(v)$ to determine whether $v$ is invalid  or $v(0)<0$
	\State  \qquad \textbf{if} so \textbf{halt} and  \textbf{return} \invalid  \label{negative halt}
	\State \textbf{run} $\alg{Positive}(uA_rv)$  to determine whether $uA_rv$ is valid or $uA_rv(0) \leq 0$    \label{pinch validity check}
	\State  \qquad \textbf{if} so \textbf{halt} and  \textbf{return} \invalid  \label{next invalid check} 
	\State \textbf{run} $\alg{CutRank}_r(w)$ 
	\State  \qquad it either declares $w$ invalid, in which case \textbf{halt} and  \textbf{return} \invalid 
	\State \qquad  or it returns a word $w' = A_0^i A_r\inv u' A_r v$  such that 
	\State \qquad   $w' \sim w$,  $\ell(w') \leq \ell(w)$, $\eta(u')=0$, $\green{u'\ne \epsilon}$ and $\Rank(u')< r-1$ 
	\State \textbf{run} $\alg{FinalPinch}_r(A_r\inv u' A_r v)$ \label{call fp}
	\State  \qquad \textbf{if} it declares $A_r\inv u' A_r v$ invalid,   \textbf{halt} and  \textbf{return} \invalid \label{final pinch says invalid}
	\State \qquad \textbf{else} it outputs $A_0^l v$ for some $l$, in which case \textbf{set} $w'' := A_0^{i+l}v$  \label{final pinch equiv}
	\State \textbf{run} $\alg{Positive}(w'')$  \label{run positive}
	\State  \qquad \textbf{if} it declares $w''$ invalid,   \textbf{halt} and  \textbf{return} \invalid 
	\State \qquad \textbf{else} \textbf{return} $w''$  \label{final output validity check} 
 \end{algorithmic}
\end{algorithm}

\begin{algorithm}[h]
\caption{--- $\alg{CutRank}_r$ for $r \geq 2$.   \newline
$\circ$ \ Input a word $w=A_r\inv u A_r v$ with $\eta(u)=\eta(v)=0$ and $\Rank(u) \leq r-1$. \newline
 $\circ$ \  Either declare $w$ invalid,  or return $w'=A_0^lv$ where $\ell(w')\le \ell(w)-2$, or return $w' = A_0^iA_r\inv u' A_rv \sim w$ where $\Rank(u') \leq r-2$, $\eta(u') = 0$, and $\ell(w') \leq \ell(w)$.   \newline
$\circ$ \  Halt in time $O(\ell(w)^{4+(r-1)})$.}  
\label{Alg: CutRank}
\begin{algorithmic}[3]
\State  \textbf{set} $i=0$ and re-express $w$ as $A_0^i A_r\inv u A_r v$
\State \textbf{if} $v(0)<0$ (checked using $\alg{Positive}$), \textbf{halt} and \textbf{return invalid}\label{empty word invalid check}
\State  \textbf{if} {$u$ is the empty word}, \textbf{halt} and \textbf{return} $v$   \label{unnecessary empty word line}
\While{$\Rank(u) = r-1$}  \label{CutRank while}
		\State{\textbf{run} $\alg{Positive}(A_0\inv u A_r v)$ to test whether $uA_rv(0) = 1$} 
			\State \qquad \textbf{if} so \textbf{halt} and \textbf{return} the output $w'=A_0^lv$ of $\alg{OneToZero}(w)$  \label{switch to OneToZero}
		\State \textbf{run} $\alg{Positive}( u A_r v)$ to test whether  $uA_rv(0) \le 0$ 
		\State \qquad \textbf{if} so, \textbf{halt} and \textbf{return} $\Invalid$ \label{CutRank fault}
		\State \textbf{express} $u$ as $u' A_{r-1} u''$ where $\Rank(u') <r-1$ (i.e.\ locate the  leftmost $A_{r-1}$ in $u$) 
		\State \textbf{increment} $i$ by $1$
		\State \textbf{set} $w :=  A^{i}_0A_r\inv A_{r-1}\inv u' A_{r-1} u'' A_r v$ (i.e.\ substitute $A_0A_r\inv A_{r-1}$ for $A_r\inv$ in $w$)   \label{CutRank substitution}
		\State \textbf{run} $\alg{Pinch}_{r-1}(A_{r-1}\inv u' A_{r-1} u'' A_r v)$
		\State \qquad \textbf{if} it returns invalid \textbf{halt},  \textbf{return} $\Invalid$
		\State \qquad \textbf{else}   let $w_0:=A_0^s u'' A_r v$ be the (valid) word returned  
		\State \qquad   \textbf{set}   $w := A^{i}_0A_r\inv w_0$ \label{new w}
		\State \qquad   \textbf{set} $u := A_0^s u''$ so that $w = A_0^i A_r^{-1} u A_r v$  \label{CutRank induction}
 \EndWhile \label{CutRank endwhile}
\State \textbf{return} $w$ \label{return w}
\end{algorithmic}
\end{algorithm}

\begin{algorithm}[h]
\caption{ --- $\alg{FinalPinch}_r$  for $r \geq 2$.    \newline
$\circ$ \ Input a word $w = A_r\inv u A_rv$ with $\eta(u)=\eta(v) = 0$, $u\ne\epsilon$ and $\rank(u)<r-1$.  \newline
$\circ$ \ Either declare $w$ invalid or return a word $A_0^l v \sim w$ of length at most $\ell(w)-2$. \newline
$\circ$ \ Halt in $O(\ell(w)^{4+(r-2)})$ time.}    
\label{Alg: FinalPinch}

\begin{algorithmic}[3]
\State \textbf{run} $\alg{Positive}(A_0\inv u A_r v)=0$ to decide among the following cases  
\State \qquad \textbf{if}  $A_0\inv u A_r v$ is invalid or $u A_r v(0) <1$,  \textbf{halt} and \textbf{return} \textbf{invalid} \label{<1 case}
\State \qquad \textbf{if}  $u A_r v(0) =1$,  \textbf{halt} and \textbf{return} $\alg{OneToZero}_r(w)$ \label{finalpinch switch to onetozero}
\State we now have that $u A_r v$ is valid and $u A_r v(0) >1$ \label{fghj}
\State
\State \textbf{run} $\alg{Positive}(v)$ to determine whether  $v(0)<0$, $v(0)=0$, or $v(0)>0$
\State	
\State  \textbf{if} $v(0)<0$,   \textbf{halt} and \textbf{return} \textbf{invalid}  \label{v<0 case}  \label{input neg case} 
\State	
	\State{\textbf{if} $v(0)=0$}  \label{input zero case} 
		\State \qquad  \textbf{if}  $r=2$,  \textbf{run} $\alg{BasePinch}(A_r\inv u A_r v)$  \label{FinalPinch r=2 Case}
	\State \qquad  \qquad \textbf{if}  it returns \textbf{invalid}, \textbf{halt} and do likewise
	\State \qquad  \qquad  \textbf{else} \textbf{halt} and return its result $A_0^{l'}v$, which will satisfy $\ell(A_0^{l'} v) \le \ell(w)-2$ \label{output of basepinch ref}
	\State \qquad  \textbf{if} $r>2$,  \textbf{run} $\alg{Pinch}_{r-1}(A_{r-1}\inv u A_{r-1} v)$
	\State \qquad  \qquad  \textbf{if}  it returns \textbf{invalid}, \textbf{halt} and do likewise 
	\State \qquad  \qquad  \textbf{else}  it  returns  $A_0^l v$ for some  $\abs{l} \leq \ell(u)$  \label{invalid returned}
	\State{\qquad  \textbf{if}  $l \leq 0$,  \textbf{halt} and \textbf{return} $\Invalid$ }\label{outside domain}
	\State \qquad  \textbf{run}  $\alg{BasePinch}(A_r\inv A_0^{l-1} A_r v)$
	\State \qquad  \qquad \textbf{if}   it returns \textbf{invalid}, \textbf{halt} and do likewise \label{finalpinch zero pinch0}
	\State \qquad  \qquad  \textbf{else}  it returns $A_0^{l'}v$ for some   $\abs{l'} \leq \abs{l-1} = l -1 $, \label{abs un}
	\State \qquad  \qquad \qquad  in which case \textbf{halt} and \textbf{return} $A_0^{l'+1}v$  \label{next one}
\State	
\State{ \textbf{if} $v(0) > 0$}  \label{not zero case} 
	\State{\qquad \textbf{run} $\alg{Pinch}_{r-1}(A_{r-1}\inv u A_{r-1}A_r A_0\inv v)$ } \label{finalpinch induction}
	\State \qquad \qquad  \textbf{if} it returns \textbf{invalid}, \textbf{halt} and do likewise \label{invalid returned3}

	\State \qquad \qquad \textbf{else} it  returns  $A_0^lA_r A_0\inv v$ for some $\abs{l} \leq \ell(u)$ \label{return something}
	\State \qquad \textbf{run} $\alg{BasePinch}(A_r\inv  A_0^{l} A_r A_0\inv v)$ \label{finalpinch pinch0}
	\State \qquad  \qquad  \textbf{if}  it returns \textbf{invalid}, \textbf{halt} and do likewise \label{invalid returned4} 
	\State \qquad  \qquad \textbf{else} it  returns  $A_0^{l''} A_0\inv v$ for some $\abs{l''} \leq \abs{l}$,
	\State \qquad  \qquad \qquad in which case    \textbf{halt} and \textbf{return} $A_0^{l''}  v$ \label{final return}  
 \end{algorithmic}
\end{algorithm}

\emph{Correctness of  $\alg{Pinch}_{r-1}$   implies the correctness of  $\alg{CutRank}_r$ for all $r \geq 2$.}
The idea of $\alg{CutRank}_r$ is that each pass around the while loop eliminates one $A_{r-1}$ from $u$. 
 So in the output,  $\Rank(u)  < r-1$.

\begin{itemize}
\item[\ref{empty word invalid check}:] If $r\ge 2$, then the domain of $A_r$ is $\N$, and so $w$ is invalid when $v(0) < 0$. 
\item[\ref{unnecessary empty word line}:]  Since $v(0)\ge 0$ now, Lemma \ref{Lem: OneToZero Equiv} applies.
\item[\ref{switch to OneToZero}:] $\ell(w') \leq \ell(w) -2$ by the specifications of  $\alg{OneToZero}$.
\item[\ref{CutRank fault}:]  If $uA_rv(0) \leq 0$, it is outside the domain of $A_r\inv$ (as $r \geq 2$),  so the algorithm's input is invalid.

\item[\ref{CutRank substitution}:] 
 Substituting gives an equivalent word here by Lemma~\ref{Lem: OneToZero Equiv},  since $uA_rv(0)\ge 1$. At this point, $\ell(w)$ is at most $2$ more than its initial length.  

\item[\ref{CutRank induction}:]  Now $w$  is no longer than it was at the start of the \textbf{while} loop because  $\alg{Pinch}_{r-1}$ (assuming it does not halt) trims at least $2$ letters, offsetting the gain at line~\ref{CutRank substitution}.
  The word $w$ here at the end of the \textbf{while} loop is equivalent to the $w$ at the start because of our remark on line~\ref{CutRank substitution} and because we are replacing a suffix $A_{r-1}\inv u' A_{r-1} u'' A_r v$ by an equivalent word produced by $\alg{Pinch}_{r-1}$.

\item[\ref{return w}:]  It follows from our remarks on lines~\ref{CutRank substitution} and \ref{CutRank induction} that $\ell(w)$ here   is at most the length of the $w$ originally inputted.  
\end{itemize}

The while loop is traversed at most $\ell(w)$ times.  Each time,   \alg{Positive} (twice), $\alg{OneToZero}$ and $\alg{Pinch}_{r-1}$  may be called, and by the remarks above, their inputs are always of length at most $\ell(w)$.  So, as each of these subroutines halt in time $O(\ell(w)^{4+(r-2)})$, $\alg{CutRank}_r$  halts in $O(\ell(w)^{4+(r-1)})$ time.
\qed

\emph{Correctness of $\alg{Pinch}_{r-1}$  implies correctness of  $\alg{FinalPinch}_r$ for $r \geq 2$.} { \quad} \\ \vspace*{-6mm} 
\begin{itemize}

\item[\ref{<1 case}:] If $u A_r v(0) <1$, then it is outside the domain of $A_r^{-1}$.  

\item[\ref{fghj}:] $uA_rv$ is valid if and only if $A_0\inv uA_rv$ is valid.

\item[\ref{v<0 case}:] In this case $v(0)$ is outside the domain of $A_r$.

\item[\ref{FinalPinch r=2 Case}:] If $r=2$, the rank of $u$ is zero, so   $\alg{BasePinch}$ applies.
\item[\ref{output of basepinch ref}:]   $\ell(A_0^{l'} v) \le \ell(w)-2$ by properties of $\alg{BasePinch}$. 

\item[\ref{invalid returned}:]      
$w \ \sim \ A_0 A_r\inv A_{r-1}\inv u A_{r-1} v$ when $r>2$ and $v(0)=0$,
because  $A_rv \sim A_{r-1}v$ and we can substitute $A_0 A_r\inv A_{r-1}\inv$ for $A_r\inv$ as per Lemma~\ref{Lem: OneToZero Equiv}, given that $uA_rv(0)>1$.   So if   $A_{r-1}\inv u A_{r-1} v$ is invalid, then so is $w$.  And if $\alg{Pinch}_{r-1}$ gives us that $A_{r-1}\inv u A_{r-1} v \sim A_0^l v$, then $w \sim  A_0 A_r\inv  A_0^l v$.

\item[\ref{outside domain}:]     If $l \leq 0$, then $w$ is invalid because   $A_0^{l}v(0)\le 0$ and lies outside of the domain of $A_r\inv$ (since $r \geq 2$).  

\item[\ref{finalpinch zero pinch0}:]   Next, working from $w \sim A_0A_r\inv A_0^{l}  v$ established in our comment above on line~\ref{invalid returned}, we get that  $w \sim A_0A_r\inv A_0^{l-1} A_r v$ because $A_0^{-1} A_rv \sim v$, given that $r \geq 2$ and $v(0)=0$.   So, if $\alg{BasePinch}$ tells us that $A_r\inv A_0^{l-1} A_r v$ is invalid, then so is $w$.

\item[\ref{abs un}:] $\abs{l-1} = l -1$ here  because $l >0$ here.  

\item[\ref{next one}:]    Similarly, if $A_r\inv A_0^{l-1} A_r v \sim A_0^{l'} v$, then $w \sim A_0^{l'+1} v$.  
Now, $\abs{l'+1} \leq \abs{l'}+1 \leq l$ by line~\ref{abs un}, and $l \leq \ell(u)$ in the case $r>2$ of line~\ref{invalid returned}.   So  $\ell(A_0^{l' +1}v ) \leq \ell(w) - 2$, as required.

\item[\ref{not zero case}:] $w \sim  A_0 A_r\inv A_{r-1}\inv u A_{r-1}A_rA_0\inv v$ because Lemma~\ref{Lem: OneToZero Equiv} tells  us that substituting $A_{r-1}A_rA_0\inv$ for $A_r$ and  $A_0 A_r\inv A_{r-1}\inv$ for $A_r\inv$ in $w$ gives an equivalent word as $v(0) >0$ and $u A_{r-1} v(0) >1$.  This word is longer than $w$ by $2$. 

\item[\ref{invalid returned3}:] So, if the suffix $A_{r-1}\inv u A_{r-1}A_rA_0\inv v$  is invalid, then so is $w$.  

\item[\ref{return something}:] Similarly, if  the suffix $A_{r-1}\inv u A_{r-1}A_rA_0\inv v \sim   A_0^{l} A_r A_0\inv v$, then $w \sim  A_0 A_r\inv  A_0^{l} A_r A_0\inv v$.

\item[\ref{invalid returned4}:] If the suffix $A_r\inv  A_0^{l} A_r A_0\inv v$  is invalid, then so is $w$.

\item[\ref{final return}:] If the suffix $A_r\inv  A_0^{l} A_r A_0\inv v \sim  A_0^{l''} A_0\inv v$, then $w \sim A_0 A_0^{l''} A_0\inv v \sim A_0^{l''} v$ and has length at most $\ell(w)-2$ since $\abs{l''} \leq \abs{l}$ and (from line~\ref{return something})  $\abs{l} \leq \ell(u)$ (or to put it another way, we have taken $A_0 A_r\inv A_{r-1}\inv u A_{0} A_{1} v$ (see the comment on line~\ref{not zero case}) which is four letters longer than $w$, and  $\alg{Pinch}_{r-1}$ and $\alg{BasePinch}$ have each shortened it by two).  
\end{itemize}

\alg{FinalPinch}$_r$   halts in $O(\ell(w)^{4+(r-2)})$ time  because it makes at most four calls on subroutines ($\alg{Positive}$, $\alg{OneToZero}$, $\alg{Pinch}_{r-1}$ or $\alg{BasePinch}$) and, each time, the subroutine has input of length at most $\ell(w)+2$  and  halts in   $O(\ell(w)^{4+(r-2)})$ time.
\qed

\emph{Correctness of  $\alg{CutRank}_r$ and $\alg{FinalPinch}_r$ implies correctness of  $\alg{Pinch}_r$ for $r \geq 2$.}
\begin{itemize}

\item[\ref{negative halt}:] If $v$ is invalid, then so is $w$.  If $v(0)<0$, then   $v(0)$ is outside the domain of $A_r$ (as $r\geq 2$) and so $w$ is invalid.

\item[\ref{next invalid check}:]   If $u A_r v$ is invalid, then so is $w$.  If $uA_r v(0) \leq 0$, then   $v(0)$ is outside the domain of $A_r\inv$ (as $r\geq 2$)  and so $w$ is invalid.

\item[\ref{call fp}:] $\ell(A_r^{-1} u' A_r v) \leq \ell(w') \leq \ell(w)$, the second inequality  being by an established property of $\alg{CutRank}_r$.

\item[\ref{final pinch says invalid}:]  If the suffix $A_r\inv u' A_r v$ of $w'$ is invalid, then so is $w'$, and hence so is $w$.  

\item[\ref{final pinch equiv}:] $w'' \sim w$ because it is obtained by replacing the suffix $A_r\inv u' A_r v$ of $w'$ by an equivalent word.

\item[\ref{run positive}:] $\eta(w'') =0$, so we can use $\alg{Positive}$ to determine   validity of $w''$. Also, $\ell(w'') \leq i + \ell(A_0^lv) \leq i + \ell(A_r^{-1} u' A_r v) -2 = \ell(w') -2 < \ell(w)$, the second and final inequalities follow from established properties of $\alg{FinalPinch}_r$ and $\alg{CutRank}_r$, respectively. 

\end{itemize}
That $\alg{Pinch}_r$ runs in $O(\ell(w)^{4+(k-1)})$ time follows directly from the time bounds for the subroutines $\alg{Positive}$, $\alg{CutRank}_r$, $\alg{BasePinch}$ and $\alg{FinalPinch}_{r}$ as it calls these   at most six times in total and on each occasion, the input has length at most $\ell(w)$---see the comments above on lines~\ref{call fp} and \ref{run positive}. \qed

\emph{Correctness of  $\alg{Pinch}_r$ for $r \geq 1$ and of  $\alg{CutRank}_r$ and $\alg{FinalPinch}_r$  for $r \geq 2$.} For $r=1$, the correctness of $\alg{Pinch}_1$ follows from that of $\alg{BasePinch}$.   As explained above, for $r \geq 2$,  the correctness of $\alg{CutRank}_r$ and $\alg{FinalPinch}_r$ implies that of $\alg{Pinch}_r$, and the correctness of $\alg{Pinch}_{r-1}$ implies that of $\alg{CutRank}_r$ and $\alg{FinalPinch}_r$.  So, by induction on $r$,  $\alg{Pinch}_r$ is correct for all $r \geq 1$. 
\qed

 \begin{algorithm}[h]
\caption{ --- \alg{Reduce}.    \newline $\circ$ \  Input  a word $w$ with $\eta(w)>0$.  \newline
$\circ$ \ Either return that $w$ is invalid, or return a  word $w' \sim w$ with $\ell(w') \leq  \ell(w) + 2k$ and $\eta(w') = \eta(w) -1$.   \newline 
$\circ$ \ Halt in $O(\ell(w)^{4+(k-1)})$ time. }

\begin{algorithmic}[3]
	\State express $w$ as $w_1 A_r\inv w_2$ where $r \geq 1$ and $\eta(w_2) =0$ 
	\State (i.e.\ locate  rightmost   $A_1\inv,A_2\inv, \ldots, A_k\inv$ in $w$) 
	\State
	\State \textbf{if} $\Rank(w_2)< r$ and $r\ge 2$,  \textbf{run}  $\alg{Pinch}_r(A_r\inv w_2 A_0\inv A_r)$ \label{unchanged eval}
 		\State \qquad  \textbf{if} it declares $A_r\inv w_2 A_0\inv A_r$ invalid, \textbf{halt} and \textbf{return} $\Invalid$
		\State \qquad   \textbf{else} it returns  $A_0^l$ for some $\abs{l} \leq \ell(w_2)+1$, in which case \textbf{return} $w':= w_1A_0^l$ \label{equiv by pinch}
		\State 
	\State \textbf{if} $\rank(w_2)=0$ and $r=1$, \textbf{run} $\alg{Pinch}_1(A_1\inv w_2 A_1)$   \label{Reduce case r=1}
	\State \qquad  \textbf{if} it declares $A_1\inv w_2 A_1$ invalid, \textbf{halt} and \textbf{return} $\Invalid$ 
		\State \qquad   \textbf{else} it returns  $A_0^l$ for some $\abs{l} \leq \ell(w_2)$, in which case \textbf{return} $w':= w_1A_0^l$  \label{equiv by pinch base}
	\State
	\State \textbf{if} $\Rank(w_2) \geq r$ 
	        	\State \qquad express $w_2$ as $w_3 A_s w_4$ where $r\leq s$ and $\Rank(w_3) < r$ \label{w3andw4}
							\State \qquad \textbf{run} $\alg{Positive}(w_4)$ to decide among the following cases
								\State \qquad \qquad if $r=s=1$, set $w'' = \alg{Pinch}_1(A_r\inv w_3 A_s w_4)$ \label{r=s=1 case}
	\State \qquad \qquad \textbf{else if}  $w_4$ is invalid  or $v(0)<0$, \textbf{halt} and  \textbf{return} \invalid  \label{v<0 invalid}
	\State  \qquad \qquad \textbf{else if}  $w_4(0) =0$, $r=1$ and $s>r$, set $w'' = \alg{Pinch}_{r}(A_{r}\inv u A_0 A_{r} v)$ \label{whole zero special}
	\State \qquad \qquad \textbf{else if}  $w_4(0) =0$ and $r>1$, set $w''=\alg{Pinch}_{r}(A_{r}\inv w_3 A_{r} w_4)$ \label{whole zero} 
	\State \qquad \qquad \textbf{else}  $ w_4(0) >0$, so set $w'' = \alg{Pinch}_{r}(A_{r}\inv w_3 A_{r}  \   A_{r+1}A_0\inv  \   A_{r+2} A_0\inv   \   \cdots   A_{s}A_0\inv \  w_4)$  \label{whole sub}
 		\State \qquad  \textbf{if} $w''=\Invalid$, \textbf{halt} and \textbf{return} $\Invalid$ \label{another invalid}
		\State \qquad \textbf{else} \textbf{return} $w': = w_1 w''$ \label{happy returns}
\end{algorithmic}
\end{algorithm}

\begin{proof}[Correctness of  $\alg{Reduce}$]
  The idea is to eliminate the rightmost $A_r\inv$ with $1 \le r\le k$ from $w$ by either using $\alg{Pinch}_r$ directly on a suffix of $w$ or by manipulating $w$ into an equivalent word with a suffix that can be input into $\alg{Pinch}_r$.   

\begin{itemize}
\item[\ref{unchanged eval}:]  $A_0\inv A_r(0)=0$ (since $r \geq 2$), so $w_2A_0\inv A_r \sim w_2$.

\item[\ref{equiv by pinch}:]  $A_0^l \sim A_r\inv w_2 A_0\inv A_r$ and so $w' \sim w$.  Evidently, $\eta(w')=\eta(w)-1$.  And $\ell(w') = \ell(w_1) + \abs{l}  \leq  \ell(w_1) +\ell(w_2)+1 =  \ell(w)  \leq \ell(w) + 2k$, as required. 

\item[\ref{Reduce case r=1}:]  $A_1(0)=0$, so $w_2A_1\sim w_2$.   

\item[\ref{equiv by pinch base}:]  $A_0^l \sim A_1\inv w_2 A_1 \sim A_1\inv w_2$ and so $w'\sim w$, as required.  Also, evidently,  $\eta(w')=\eta(w)-1$, and $\ell(w') \leq  \ell(w) + 2k$, as required. 
\item[\ref{w3andw4}:] Moreover, $\eta(w_3)=\eta(w_4)=0$ because $\eta(w_2)=0$, as will be required in line~\ref{r=s=1 case}.  
\item[\ref{r=s=1 case}:] The length of $w''$ is at most $\ell(w)-\ell(w_1) -2$ by properties of $\alg{Pinch}_r$. 
\item[\ref{v<0 invalid}:] If $w_3(0) < 0$, then $w$ is invalid because $s\ge 2$
\item[\ref{whole zero special}:] In this case  $A_{r}\inv w_3 A_0A_{r} w_4 \sim A_{r}\inv w_3 A_{s} w_4$  since $A_0A_{r}(0) = A_{s}(0)$.   As required, if $w'' \ne \Invalid$, it has length at most $\ell(A_{r}\inv u A_0 A_{r} v) = \ell(w)-\ell(w_1) +1  < \ell(w)-\ell(w_1) + 2k$ and contains no $A_1\inv, \ldots, A_k\inv$ by the properties established for $\alg{Pinch}_{r}$.
\item[\ref{whole zero}:]  Similarly, in this case   $A_{r}\inv w_3 A_{r} w_4 \sim A_{r}\inv w_3 A_{s} w_4$  since $A_{r}(0) = A_{s}(0)$, and the output has the required properties. 

\item[\ref{whole sub}:]If $w_4(0)> 0$, then $A_{r}\inv w_3 A_{s} w_4$   and  $A_{r}\inv w_3 A_{s-1}A_{s}A_0\inv w_4$  are equivalent by Lemma~\ref{Lem: OneToZero Equiv}.   As $v(0)-1\ge 0$, and so is in the domain of $A_s$, the word   $A_{s}A_0\inv v$ is valid. And, as $A_{s}A_0\inv v(0) = A_{s}(v(0)-1) >0$, we may replace the $A_{s-1}$ by $A_{s-2}A_{s-1}A_{0}\inv$ to get another equivalent word.  Indeed, we may repeat this process $s-r \le k$ times, to yield  an equivalent word 
\[ A_{r}\inv w_3 A_{r} \  A_{r+1}A_0\inv  \  A_{r+2} A_0\inv  \   \cdots   A_{s}A_0\inv  \ w_4\]
of length $\ell(w)-\ell(w_1) + 2(s-r)$.  Applying  $\alg{Pinch}_r$ then returns (if valid) an equivalent word 
\[w''  \ = \  A_0^l    \ A_{r+1}A_0\inv \   A_{r+2} A_0\inv   \  \cdots   A_{s}A_0\inv \  w_4\]
whose length is at most $\ell(w) -\ell(w_1) + 2(s-r)-2$. 

\item[\ref{another invalid}:] If the suffix $A_r\inv w_3 A_s w_4$ of $w$ is invalid, then $w$ is invalid. 

\item[\ref{happy returns}:] By the above $\ell(w'') \le \ell(w)-\ell(w_1)+2(s-r)$, we have that  $w'' \sim A_r\inv w_3 A_s w_4$,  $\eta(w'')=0$ and $\ell(w'') \leq \ell(A_r\inv w_3 A_s w_4) + 2r = 1 + \ell(w_2) +2r$.  It follows that $w \sim w_1w''$ and $\ell(w_1w'') = \ell(w_1) + \ell(w'') \leq \ell(w_1) +   1+ \ell(w_2) + 2r \leq \ell(w) +2k$, as required.  Also, again evidently, $\eta(w')=\eta(w)-1$.
\end{itemize}

\alg{Reduce} halts in $O(\ell(w)^{4+(k-1)})$ time since $\alg{Pinch}_{r}$ and $\alg{Positive}$ do and they are each called at most once  and only on words of length at most $\ell(w)+2k$, and otherwise $\alg{Reduce}$ scans $w$   and compares non-negative integers that are at most   $k$. 
\end{proof}

 \begin{proof}[Proof of Theorem~\ref{Ackermann}]
Here is our algorithm \alg{Ackermann} satisfying the requirements of Theorem~\ref{Ackermann}: it declares, in polynomial time in $\ell(w)$, whether or not a word $w(A_0, \ldots, A_k)$ is valid, and if so, it  gives $\sgn(w)$. 

   \begin{algorithm}[h]
    \caption{--- \alg{Ackermann}.   \newline
    $\circ$ \ Input a word $w$.  \newline
    $\circ$ \ Return  whether $w$ is valid and if it is, return  $\sgn(w(0))$.  \newline 
    $\circ$ \ Halt  in $O(\ell(w)^{4+k})$  time. }
    \label{Alg: Ackermann solution}
    \begin{algorithmic}
    \State{\textbf{if} $\eta(w)>0$, \textbf{run} \alg{Reduce} successively until 
    \State   \qquad it either returns that $w$ is $\Invalid$, 
    \State   \qquad or it returns some $w' \sim w$ with $\eta(w')=0$} 
      \State{\textbf{run} $\alg{Positive}(w')$}
    \end{algorithmic}
  \end{algorithm}

After at most $\eta(w)\le \ell(w)$ iterations of $\alg{Reduce}$, we have a word $w'$ with $\eta(w')=0$ such that $w'(0)=w(0)$. We then apply  \alg{Positive} to $w'$ to obtain the result.

The correctness of \alg{Ackermann} is immediate from the correctness of  \alg{Reduce} and $\alg{Positive}$.

\alg{Reduce} is called at most $\ell(w)$ times as it decreases $\eta(w)$ by one each time.  Each time it is run, it adds at most $2k$ to the length of the word. 
So the lengths of the words inputted into \alg{Reduce} or $\alg{Positive}$ are at most $\ell(w)+2k\ell(w)$.   
So, as $\alg{Reduce}$ and $\alg{Positive}$ run in $O(\ell(w)^{4 + (k-1)})$ time in the lengths of their inputs,  \alg{Ackermann} halts in $O(\ell(w)^{4+k})$ time. 
\end{proof}

\section{Efficient calculation with $\psi$-compressed integers}  \label{phi function preliminaries}

\subsection{$\psi$-functions and $\psi$-words}

Similarly to Ackermann functions in Section~\ref{Ackermann function preliminaries}, we define \emph{$\psi$-functions} by   
\begin{align*} 	
  \psi_1 & : \Z \to \Z  \qquad n    \mapsto  n-1    \\
\psi_2 & : \Z \to \Z  \qquad  n   \mapsto 2n - 1      \\ 
 \psi_i & : -\N \to -\N   & \forall i \geq 3 \\ 
& \qquad \psi_i(0)   \   :=  \ -1 &  \forall i \geq 1 \\  
& \qquad \psi_{i+1}(n)   \   := \   \psi_{i}\psi_{i+1}(n+1) -1  &  \forall n \in -\N, \forall i \geq 2. \\
\end{align*}

Having entered  the $i=1$ row and $n=0$ column as per the definition,   a table of values of  $\psi_i(n)$ can be completed by determining each row from right-to-left from the preceding one using the recurrence relation:   
$$\begin{array}{c c c c c c c c | l }
\cdots & n & \cdots & -4 & -3 & -2 & -1 & 0 &    \\ \hline
\cdots & n-1 & \cdots & -5 & -4 & -3 & -2 & -1 & \psi_1 \\
\cdots & 2n-1 & \cdots & -9 & -7 & -5 & -3 & -1 &\psi_ 2 \\ \vspace{-2mm}
\cdots & 2- 3 \cdot 2^{-n}  & \cdots & -46 & -22 & -10 & -4 & -1 & \psi_3 \\ \vspace{-2mm}
& \vdots  &  & \vdots & 1-3 \cdot 2^{95} & -95 & -5 & -1 & \psi_4 \\ \vspace{-2mm}
& &  &  & \vdots & \vdots & \vdots & \vdots & \vdots  \\ \vspace{-2mm}
& &  & & & & -i-1 & -1 & \psi_i \\ \vspace{-2mm}
& &  & & & & \vdots & \vdots & \vdots \\
\end{array} $$

 The following proposition explains why we defined $\psi$-functions with the given domains.   It details the key property of $\psi$-functions, which is that they govern whether and how  a power of $t$ \emph{pushes past} an  $a_i$ on its right, to leave an element of $H_k$ times a new power of $t$ without changing the element of  $G_k$ represented.

\begin{prop}\label{Prop: Hydra Game}
Suppose $r,\,i$ and $k$ are integers such that $1 \leq i \leq k$.  Then $t^r a_i \in H_k t^s$  in $G_k$  if and only if  $r$ is in the domain of $\psi_i$ and $s = \psi_i(r)$. 
\end{prop}

\begin{proof}
First we prove the `if' direction by inducting on pairs $(i,r)$, ordered lexicographically.  We start with the cases $i=1$ and $i=2$.
As $a_1t\in H_k$ and $t^{-1} a_1 t = a_1$,  
\[t^r a_1 \ = \ a_1 t \  t^{r-1} \in H_kt^{r-1} \ = \ H_kt^{\psi_1(r)} \] for all $r \in \Z$.
And as, $a_2t \in H_k$ and $t^{-1} a_2 t =a_2 a_1$ also,
$$t^r a_2 \ = \ t^r a_2 t^{-r} t^r \ = \ a_2 a_1^{-r} t^r \ = \ a_2t \ (a_1t)^{-r} \ t^{2r-1}$$ 
for all $r \in \Z$.  Next  the case where $r=0$ and $1 \leq i \leq k$:  $$t^r a_i  \ =  \ a_i  \  = \  a_i t \  t\inv  \ \in \  H_kt\inv  \ =  \  H_kt^{\psi_i(0)},$$   since $a_it\in H_k$ and $\psi_i(0)=-1$.  Finally,   induction gives us that  
\[t^r a_i \ = \ t^{r+1} a_i a_{i-1}t\inv \ \in  \ H_k t^{\psi_i(r+1)} a_{i-1}t^{-1} \ = \ H_k t^{\psi_{i-1}\psi_i(r+1)-1} \ = \  H_k t^{\psi_i(r)}\]
for all $i \geq 2$ and $r \le 0$, as required.   

For the `only if' direction suppose  $t^r a_i \in H_k t^s$  for some $s \in \Z$.  Then $$t^r a_i t^{-r} \ = \   \theta^{-r}(a_i)  \ \in \  H_k t^{s-r}$$  for some $s \in \Z$.     Lemma~7.3  in \cite{DR} tells us that in the cases $i=1,2$ this occurs  when $r \in \Z$,   and in the cases $i \geq 3$ it occurs  when $r \in - \N$.   In other words, it occurs when $r$ is in the domain of $\psi_i$.   Now,  given that $r$ is in the domain of $\psi_i$, we have that $t^r a_i \in H_k t^{\psi_i(r)}$ from the calculations earlier in our proof, and so $H_k t^{\psi_i(r)} = H_k t^s$, but this implies that $s = \psi_i(r)$ by Lemma 6.1 in \cite{DR}.  
\end{proof}

For example, painful calculation can show that
\[t^{-2} a_3 a_1\ = \ (a_3 t)(a_2 t)(a_1 t) (a_2 t) (a_1 t)^5 t^{-11} \in H_3t^{-11},\]
but  Proposition~\ref{Prop: Hydra Game} immediately gives:
\[t^{-2}a_3a_1 \in H_3 t^{\psi_1\psi_3(-2)} \ = \ H_3 t^{-11}.\]

The following criterion for whether and how  a power of $t$ pushes past an  $a\inv_i$ on its right, to leave an element of $H_k$ times a new power of $t$ can   be derived from Proposition~\ref{Prop: Hydra Game}.   

\begin{cor} \label{pushing cor}
Suppose $i$ and $k$ are integers such that  $1 \leq i \leq k$.  Then $t^s a_i\inv \in H_k t^r$  in  $G_k$  if and only if  $r$ is in the domain of $\psi_i$  and $s = \psi_i(r)$. 
\end{cor}

\begin{proof}
$t^s a_i\inv \in H_k t^r$ if and only if $t^r a_i \in H_k t^s$.
\end{proof}

The connection between $\psi$-functions and hydra groups is also apparent in that they relate to the functions $\phi_i$ of \cite{DR} by the identity $\psi_i(n)   = n-\phi_i(-n)$  for all $n \in -\N$ and all $i \geq 1$.  We will not use this fact  here, so we omit a proof, except to say that the recurrence $\phi_{i+1}(n)  =   \phi_{i+1}(n-1) + \phi_{i}(\phi_{i+1}(n-1)+n-1)$ for all $i \ge 1$  and $n\ge 1$ of Lemma 3.1 in \cite{DR} translates to the defining recurrence of $\psi$-functions.

\begin{lemma} \label{Properties of bphi}
\begin{align}
  \psi_2(n) \ & = \  2n-1 & &   \forall n \le 0,  \label{psi2 case} \\
  \psi_3(n) \ & = \  2-3\cdot 2^{-n} & &  \forall n \le 0, \label{psi3 case}  \\
    \psi_i(-1) \ &= \ -i-1  && \forall i \geq 1,  \label{bphi of 1} \\
    \psi_i(n) \ &\geq \ \psi_{i+1}(n)  && \forall i \geq 1, n \le 0,  \label{bphi7} \\
    \psi_i(n) \ &> \ \psi_i(n-1)  && \forall i\ge  1,n \leq 0, \label{bphi8} \\
    n \ & > \ \psi_i(n)  && \forall i\ge  1,n \le 0, \label{bphi9} \\
		\psi_i(m) +\psi_i(n) \ &\ge \ \psi_i(m+n)  && \forall n,m\le -2,\, i\ge 2, \label{bphi12}  \\
	 |\psi_i(m)-\psi_i(n)| \ & \ge \  \frac12 |\psi_i(n)| & & 	\forall i \geq 3, m \neq n.   \label{Spacing for bphi} 
  \end{align}
\end{lemma} 

\begin{proof}
  (\ref{psi2 case}--\ref{bphi12})  are evident from the manner in which the table of values of $\psi_i(n)$ above is constructed. Formal induction proofs could be given  as  for Lemma~\ref{Properties of A_k}.  

For \eqref{Spacing for bphi}, when $m > n$ (so that $|n|>|m|$),   
\begin{align*} |\psi_3(m)-\psi_3(n)| & \ = \  |3\cdot 2^{-m} -3\cdot 2^{-n}|  \ \ge \ |3\cdot 2^{-n} - 3\cdot 2^{-n-1}| \ = \ \frac12 \cdot 3\cdot 2^{-n}  \\
  & \ \ge \ \frac12 \cdot 3\cdot 2^{-n} -1  \ = \ \frac12 (3\cdot 2^{-n}-2) \ = \ \frac12 \ |\psi_3(n)|, 
\end{align*}
and when  $m < n$ (so that $|n|<|m|$), by the preceding 
\[|\psi_3(m)-\psi_3(n)| \ = \ |\psi_3(n)-\psi_3(m)|\ \ge \ \frac12 |\psi_3(m)|\ \ge \ \frac12 \ |\psi_3(n)|,\]
using \eqref{bphi8} for the last inequality.  So the result holds for $i=3$.   That it also holds for all $i > 3$ then follows.   We omit the details.  
\end{proof}

   By    \eqref{bphi8},  $\psi$-functions are injective and so have  inverses $\psi_i\inv$ defined on the images of $\psi_i$:
 $$\begin{array}{rll} 
 \psi_1\inv \!\!\! & : \Z \to \Z \qquad & n \mapsto n+1,  \\
 \psi_2\inv \!\!\! & :  2\Z+1 \to \Z \qquad  & n \mapsto (n+1)/2, \\ 
 \psi_i\inv \!\!\! & :  \textup{Img} \psi_i \to - \mathbb{N}  \qquad  & n \mapsto \psi_i\inv(n).
 \end{array}$$
 
    So, like Ackermann functions, they can  specify integers.
A \emph{$\psi$-word} is a word  $f = f_n f_{n-1}\cdots f_1$ where each $f_i\in \{\psi_1^{\pm 1},\psi_2^{\pm 1},\ldots\}$.   We let    $$\eta(f)  \ := \ \#\{ i \mid \,  1\le i\le n,  f_i = \psi_j\inv \textup{ for some }   j \ge 2\}.$$ If   $f_{j-1}\cdots f_1(0)$ is in the domain of $f_j$ for all $2\le j\le n$, then $f$ is   \emph{valid} and \emph{represents} the integer $f(0)$.  
  When  $f$ is non-empty,  $\Rank(f)$ denotes the highest $i$ such that $\psi_i^{\pm 1}$ is a letter of $f$.
  We define an equivalence relation $\sim$ on words  as in Section~\ref{Ackermann function preliminaries}.

Proposition~\ref{Prop: Hydra Game} and Corollary~\ref{pushing cor} combine to tell us, for example, that:
\[t^{-3}a_2\inv a_1 \ \in \ H_2 t^{\psi_1\psi_2\inv(-3)} \]
if $-3 \in \Img \psi_2$ and  $\psi_2\inv(-3)$ is in the domain of $\psi_1$---in other words, if $\psi_1\psi_2\inv \psi_1^3$ is valid.  In fact these provisos are met: $\psi_2\inv(-3) = -1$ and $\psi_1(-1)=-2$, so $t^{-3}a_2\inv a_1 \ \in \ H_2 t^2$.  And, given that $H_k t^r = H_k$ if and only if $r=0$ by Lemma~6.1  in \cite{DR}, determining whether $t^{-3}a_2\inv a_1  \in   H_2$ amounts to determining whether $\psi_1\psi_2\inv \psi_1^3(0)=0$.  (In fact it equals $2$, as we just saw, so $t^{-3}a_2\inv a_1 \notin H_2$.)  
This suggests that efficiently testing validity of $\psi$-words and when valid, determining whether a $\psi$-word represents zero, will be a step towards a polynomial time algorithm solving the membership problem for $H_k$ in $G_k$.  (Had $\psi_1\psi_2\inv \psi_1^3$  been invalid,  we could not have immediately concluded that that $t^{-3}a_2\inv a_1   \notin   H_2$ or indeed that  $t^{-3}a_2\inv a_1   \notin  \bigcup_{r \in \Z}  H_2 t^r$. We will address this delicate issue in  Section~\ref{Sec: Piece Criterion}.)  
So we will work towards proving this analogue to Theorem~\ref{Ackermann}:
\begin{prop}\label{Thm: Psi}
There exists an algorithm $\alg{Psi}$ that takes as input a $\psi$-word $f = f(\psi_1, \ldots, \psi_k)$ and determines in time $O(\ell(f)^{4+k})$   whether or not $f$ is valid and if so,  whether $f(0)$ is positive, negative or zero. 
\end{prop}

 Expressing the recursion relation in terms of  $\psi$-words will  be key. So, analogously to Lemma~\ref{Lem: OneToZero Equiv}, we have:
\begin{lemma}\label{Lem: OneToZero Equiv bphi}
Suppose $u, v $ are $\psi$-words. The following equivalences hold if $v$ is invalid or if $v$ is valid and satisfies the further  conditions indicated:
$$\begin{array}{rll}
u \psi_{i+1} v \!\!\!  & \sim \ u \psi_1 \psi_i \psi_{i+1} \psi_1\inv v   & v(0)<0 \text{ and } i\ge 2 \\
u \psi_{i+1}^{-1} v \!\!\! & \sim \ u \psi_1 \psi_{i+1}^{-1} \psi_{i}^{-1} \psi_1\inv v   \qquad \qquad & v(0)<-1 \text{ and } i\ge 1 \\
u \psi_i\inv \psi_i v \!\!\! & \sim uv  & v(0)\ge 0 \text{ and } i\ge 1.
\end{array}$$
\end{lemma}

\subsection{An example} \label{psi word example}

 Let $$f \ = \  \psi_3\inv \psi_2\inv \psi_1^2\psi_2^2\psi_3  (\psi_2\psi_3)^2 \psi_1 \psi_1\inv.$$  Here is how $\alg{Psi}$  checks its validity and determines the sign of $f(0)$.
\begin{enumerate}
\renewcommand{\labelenumi}{\arabic{enumi}. }
\item First we locate  the rightmost $\psi_i\inv$ in $f$ with $i \geq 2$, namely the  $\psi_2\inv$, and look to `cancel' it with the first $\psi_2$ to its right.   In short, this  is possible because $$((2x-1)-2-1)/2 \ = \  x-1,$$ allowing us to replace $\psi_2\inv \psi_1^2 \psi_2$ with $\psi_1$ to give $$\psi_3\inv \psi_1 \psi_2 \psi_3  (\psi_2\psi_3)^2 \psi_1 \psi_1\inv \ \sim \ f.$$ 
\item Next we identify the   new rightmost $\psi_i\inv$  with $i \geq 2$, namely the $\psi_3\inv$  and we look to `cancel' it with the $\psi_3$ to its right.  To this end we first  reduce the rank of the subword between the $\psi_3\inv$ and $\psi_3$ (like $\alg{CutRank}$). We check  by direct calculation that $$\psi_1\psi_2\psi_3 (\psi_2\psi_3)^2 \psi_1 \psi_1\inv(0) \ < \ -1$$ (like   $\alg{Positive}$), so   the substitution $\psi_1\psi_3\inv \psi_2\inv \psi_1\inv$ for $\psi_3\inv$ is legitimate by Lemma~\ref{Lem: OneToZero Equiv bphi} and   
\[ \psi_1\psi_3\inv \psi_2\inv \psi_1\inv \psi_1\psi_2\psi_3 (\psi_2\psi_3)^2 \psi_1 \psi_1\inv \ \sim \ f.\]
By Lemma~\ref{Lem: OneToZero Equiv bphi}, cancelation of the $\psi_1^{- 1}$ with $\psi_1$, $\psi_2^{- 1}$ with $\psi_2$, and then $\psi_3^{- 1}$ with $\psi_3$  then gives
\[\psi_1(\psi_2\psi_3)^2 \psi_1 \psi_1\inv \ \sim \ f. \]
\item 
This contains no $\psi_2\inv, \ldots, \psi_k\inv$ and direct evaluation from right to left (like   \alg{Positive})  tells us that $\psi_1(\psi_2\psi_3)^2 \psi_1 \psi_1\inv$ is valid and represents a negative integer. 
\end{enumerate}
 
\subsection{Our algorithm in detail} \label{Psi in detail}
Fix an integer $k \geq 1$.  

Subroutines of $\alg{Psi}$ correspond to  subroutines of $\alg{Ackermann}$.  We first have an   analogue of $\alg{Bounds}$,  to calculate relatively small evaluations of the $\psi_k$.
 \begin{algorithm}[h]
    \caption{ --- \alg{BoundsII}. \newline $\circ$ \ Input $\ell\in\N$. \newline $\circ$ Return a list of all the (at most $(\log_2\ell)^2$) triples of integers $(r,n,\psi_r(n))$ such that $r \geq 3$, $n \le -2$, and $|\psi_r(n)| \le \ell$.
		\newline $\circ$ Halt in  time $O(\ell)$.}
  \end{algorithm}

With these minor changes, it works exactly like $\alg{Bounds}$: replace $A_i$ by $\psi_{i+1}$, calculate values of $\psi_r(n)$ for $n\le -2$, and use the recursive relation for $\psi$-functions.
The correctness argument for $\alg{BoundsII}$ is virtually identical to that for $\alg{Bounds}$.

Similarly to $\alg{Ackermann}$, $\alg{Psi}$ works right-to-left through a $\psi$-word eliminating letters $\psi_r\inv$ for $r\ge 2$, which like (the $A_r\inv$ for $r\ge 1$) greatly decrease absolute value when  evaluating the integer represented by a valid $\psi$-word.    Once all have been eliminated, giving a $\psi$-word $f$ with $\eta(f)=0$,    a subroutine \alg{PositiveII} determines the validity of $f$.   

 \begin{algorithm}[h]
    \caption{ --- \alg{PositiveII}. \newline $\circ$ Input a $\psi$-word $f$ with $\eta(f)=0$. \newline $\circ$ Either return that $f$ is $\invalid$, or that $f$ is valid and declare whether $f(0)>0$, $f(0)=0$, or $f(0)<0$. \newline $\circ$ Halt in  time $O(\ell(f)^3)$.}
	  \end{algorithm}
	  
$\alg{PositiveII}$ can be constructed analogously to $\alg{Positive}$ with the following changes:
\begin{enumerate}
\item The role of $\psi_i$ corresponds to the role of $A_{i-1}$.
\item Unlike Ackermann functions, $\psi_i:-\N\to -\N$, so appropriate signs and inequalities need to be altered. 
\item We still evaluate letter-by-letter. However, in place of  using $\alg{Bounds}$ to check whether an evaluation by $A_i$ is above some (positive) threshold, we use $\alg{BoundsII}$ to check that $\psi_k$ evaluated on a negative number is below some (negative) threshold.
\item Similarly, the case where a partial letter-by-letter evaluation is negative should be replaced by a case where the partial letter-by-letter evaluation is positive. 
\end{enumerate}
Then \alg{PositiveII} can be justified  similarly to  $\alg{Positive}$.

Next \alg{BasePinchII} processes words of the form $\psi_k^{-1}   \psi_1^l  \psi_k  v$. We make one major change: we have a stricter bound that \alg{BasePinch} on the length of the returned word $f'$. 
The substitution suggested by Lemma \ref{Lem: OneToZero Equiv bphi} requires a substitution of $4$ letters for $1$ rather than the $3$ for $1$ substitution suggested by Lemma \ref{Lem: OneToZero Equiv} for the Ackermann case. 
Here and in $\alg{PinchII}$, stricter bounds on the length of the output compensate for the longer substitution and thus prevent the length of words processed by recursive calls to $\alg{PinchII}$ from growing too large.  

 \begin{algorithm}[H]
    \caption{ --- \alg{BasePinchII}.
		\newline $\circ$ Input  a word  $f=\psi_r\inv u \psi_r v$ with  $k \geq 1$, $\rank(u) \le 1$,   $v$ a $\psi$-word, and $\eta(v) = 0$.  
\newline $\circ$ Either return $\Invalid$ when $f$ is invalid or return a word $f'=\psi_1^{l'}v \sim f$ such that $\ell(f')\le \ell(f)-2$ if $u$ is empty, $\ell(f')\le \ell(f)-4$ if $r>2$, and otherwise, $\ell(f')\le \ell(f)-3$. 
\newline $\circ$ Halt in  time $O(\ell(f)^4)$.}
		  \end{algorithm}

Construct $\alg{BasePinchII}$ like $\alg{BasePinch}$ with the following changes:
\begin{enumerate}
\item Replace all called subroutines by their $\psi$-versions.
\item $\psi_{i+1}$ replaces $A_i$ for all $i\ge 0$.
\item Signs and inequalities are adjusted to reflect that $\psi_{i+1}:-\N\to -\N$ and that $\psi_1(n)=n-1$ (in contrast to  $A_0(n)=n+1$). 
\item For the case $r=2$, whenever $\psi_2v(0)$ is valid, it is odd (since $\psi_2(n)=2n-1$) and hence the parity of $l$ determines the parity of $u\psi_2v(0)$.
 For validity, we need $u\psi_1v(0)$ to be odd, and this is sufficient since $\psi_2\inv(n) = (n+1)/2.$
When $l$ is even, return the equivalent word $f' := \psi_1^{l/2}v$. Otherwise $f$ is invalid. 
The restrictions on the length of $l$ follow directly from the fact that $|l/2|\le |l|-1$ if $l=0$. Henceforth, assume that $r\ge 3$. 
\item The inequality 
\[ |\psi_r(m)-\psi_r(p)| \ \ge \  \frac12|\psi_r(m)|\]
 which holds for all $r\ge 3$ and $m\ne p$ takes the place of the analogous inequality for Ackermann functions:
\[|A_r(p)-A_r(n)| \ \ge \   \frac12 A_r(n)\] 
which holds for all $r\ge 2$ and $m\ne p$. 
Following similar arguments for $\alg{BasePinch}$, we instead need $0\ge \psi_r v(0)\ge -2|l|$ to account for the fact that the $\psi_i$ are functions $-\N\to -\N$. 
\item If the algorithm outputs $f' \sim f$ with  $f'(0) = c\in\Z$, then  $f'=\psi_1^{v(0)-c}v$. 
\end{enumerate}

\begin{proof}[Correctness of $\alg{BasePinchII}$]
The   argument is essentially the same as that for $\alg{BasePinch}$ except that we need to verify the stronger assertions on $\ell(f')$.
If $l=0$, the algorithm eliminates $\psi_r\inv$ and $\psi_r$, reducing   length by $2$.

	For the case $l\ne 0$, consider the following: we claim that $$|\psi_r(n)-\psi_r(n-1)| \ \ge \  |\psi_3(0)-\psi_3(-1)| \ =  \ 3.$$
	Explicitly, for $r=3$, we have:
	\[|\psi_r(n) - \psi_r(n-1)| \ = \ 3\cdot 2^{-n} - 3 \cdot 2^{-(n-1)} \ = \ 3\cdot 2^{-n}\ge 3 \cdot 2^0 \ = \ 3\]
	because $n\le 0$.
	For $r>3$, assume the result holds for all ranks less than $r$. We have:
	\begin{align*}|\psi_r(n)-\psi_r(n-1)| & \ = \ |\psi_{r-1}(\psi_r(n)) - \psi_{r-1}\psi_r(n-1)| \\ & \ \ge \ |\psi_{r-1}\psi_r(n) - \psi_{r-1}(\psi_r(n)-1)| \ \ge \  |\psi_3(0)-\psi_3(-1)| \end{align*}
	where the final two inequalities follow from the fact that $\psi_{r-1}$ is non-decreasing and the inductive hypothesis, respectively.
	
	By extending this argument inductively and using  that $\psi_{r}$ is non-decreasing:
	\[|\psi_r(n) -\psi_r(n+m)| \ \ge \ 3m.\]
	So, for $r>3$ and $l\ne 0$ where $f' = \psi_0^{c-v(0)}v(0)$, we have that $\psi_r(c) - \psi_r(v(0)) = l$ implies that $|c-v(0)|\le \frac13 |l|$.
	In particular, if $l\ne 0$, then $|l|\ge 3$. 
	Therefore, 
	\[\ell(f')  \ = \  |c-v(0)| +\ell(v) \  \le \  \frac13 |l| +\ell(v)  \ \le  \  |l| - 2 +\ell(v) \  =   \  \ell(f)-4\]
	since $|l|-2 \ge \frac13|l|$ if $|l|\ge 3$. 
	Thus we have verified the assertions concerning  $\ell(f')$. 
\end{proof}

 $\alg{OneToZeroII}$ is essentially the same as  \alg{OneToZero} with  $A_0$ replaced by $\psi_1$.

\begin{algorithm}[H]\label{Alg: OneToZeroII}
\caption{ --- $\alg{OneToZeroII}$. 
 \newline $\circ$ \ Input a valid word word of the form $f=\psi_r\inv u\psi_r v$ with $r\ge 3$, $u$ not the empty word, and $\eta(u)=\eta(v)=0$ such that $u\psi_rv(0)=-1$.
 \newline $\circ$ \ Return an equivalent word of the form $f'=\psi_1^{v(0)}v$ with $\ell(f')\le \ell(f)-3$. 
  \newline $\circ$ \ Halt  in time $O(\ell(f)^4)$. }
\end{algorithm}

\begin{proof}[Proof that $\ell(f') \leq \ell(f) -3$ in $\alg{OneToZeroII}$.]  
Now $v(0) \leq 0$ since $v(0)$ is in the domain of $\psi_r$ and $r \geq 3$.  Consider first the case $v(0) \leq -1$.
First observe that $\psi_r (x) \leq x -3$ when $x \leq -1$ and $r \geq 3$.    
Since $\eta(u)=0$, $\psi_1\inv$ is the only letter it can contain which decreases the absolute value as $f(0)$ is evaluated.   So, given that $u\psi_r v(0)=-1$, $u$ must contain $\psi_1\inv$ at least $\abs{v(0)  - 3}-1 = \abs{v(0)} +2$ times.    So $\ell(u) \geq  \abs{v(0)} +2$ and therefore
$$\ell(f) - \ell(f')  \ = \  2 + \ell(u) - \abs{v(0)}  \ \geq \  4,$$ and so   $\ell(f') < \ell(f) -3$
as required.  

If $v(0)=0$, \alg{OneToZeroII} returns $f'=v$. Since $u$ is not the empty word, $\ell(f') \le \ell(f)-3$ as required. 
  \end{proof}


$\alg{PinchII}_r$ is an analogue to \alg{Pinch}$_r$. As in the previous situation, the proof is by induction and uses \alg{BasePinchII} as its base case.  As in \alg{BasePinchII}, there are now stronger restrictions on the length of a returned equivalent word.

\begin{algorithm}[h]
\caption{ --- $\alg{PinchII}_r$ for $r \geq 2$.  
 \newline $\circ$ \ Input  a word  $f=\psi_r\inv u \psi_r v$ with  $r \geq 2$, $\rank(u) \leq r-1$,   $v$ a $\psi$-word, and $\eta(v) = 0$.
 \newline $\circ$ \ Either return that $f$ is invalid, or return a word $f'=\psi_1^{l'}v$ equivalent to $f$ such that $\ell(f')\le \ell(f)-2$ if $u$ is empty, $\ell(f')\le \ell(f)-4$ if $r>2$ and $\rank(u)=1$, and otherwise, $\ell(f')\le \ell(f)-3$. 
  \newline $\circ$ \ Halt in $O(\ell(f)^{4+(k-1)})$  time.}
\label{Alg: PinchII}
\end{algorithm}

The construction of \alg{PinchII}$_r$ is the same as $\alg{Pinch}_r$ except that:
  \begin{enumerate}
\item    We replace $A_r$ by $\psi_{r+1}$ for $r\ge 0$. 
\item We replace all called subroutines by their $\psi$-word versions.
\item  In line \ref{pinch validity check}, when $\alg{PositiveII}$ checks the value of $u\psi_r v$, declare the word invalid if the result was invalid, positive or $0$.  Otherwise, run $\alg{CutRankII}_r$(w) followed by \alg{FinalPinchII}$_r$ when the result of $\alg{CutRankII}_r$ is not invalid. 
\end{enumerate}

Before discussing the correctness of \alg{PinchII}$_r$, we construct and analyze its subroutines $\alg{CutRankII}_r$ and $\alg{FinalPinchII}_r$.

\begin{algorithm}[H]
\caption{ --- $\alg{CutRankII}_r$ for $r \geq 2$.
 \newline $\circ$ \ Input a $\psi$-word of the form $f:=\psi_r\inv u\psi_r v$ with $\eta(u)=\eta(v)=0$ and $\Rank(u)\le r-1$.
 \newline $\circ$ \ Either declare that $f$ is invalid, or \textbf{halt} and  return $f':= \psi_1^lv\sim f$, or return $f' := \psi_r\inv u' \psi_r v \sim f$ where $\rank(u') \le r-2$.  In all cases $\ell(f') \leq \ell(f)$ and if $f':=\psi_1^lv$, then $\ell(f')\le \ell(f)-3$.   
  \newline $\circ$ \ Halt in $O(\ell(f)^{4+(k-1)})$  time.}
\end{algorithm}

The construction of \alg{CutRankII}$_r$ is the same as $\alg{CutRank}_r$ except that:  
\begin{enumerate}
\item We replace $A_r$ by $\psi_{r+1}$ for $r>0$, $A_0$ by $\psi_1\inv$. We replace all called subroutines by their $\psi$-word versions.
 \item In line  \ref{switch to OneToZero},  check whether $u\psi_rv(0) =-1$. If so, run and return the result of \alg{OneToZeroII}(w).
 \item In line \ref{CutRank substitution}, instead of the substitution $A_r = A_{r-1}A_rA_0\inv$ which encodes the defining recursion relation for Ackermann functions, use Lemma~\ref{Lem: OneToZero Equiv bphi} and make the substitution $\psi_r\inv = \psi_1\psi_{r}\inv\psi_{r-1}\inv \psi_1\inv$ to convert $w$ to $\psi_1\psi_{r}\inv\psi_{r-1}\inv \psi_1\inv u' \psi_{r-1} u'' \psi_r v$ where $\eta(u)=\eta(u')=\eta(u'')=0$ and $u'$ has rank strictly less than $r-1$. 
\end{enumerate}

\begin{proof}[Correctness of $\alg{CutRankII}_r$ assuming correctness of $\alg{PinchII}_{r-1}$.]
In the case $\alg{OneToZeroII}$ is used, all claims follow from the specifications of that algorithm.

We show  $\ell(f') \leq \ell(f)$.
The only changes from $\alg{CutRank}_r$ occur in the \textbf{while} loop used to remove successive $\psi_{r-1}$. As for $\alg{CutRank}_r$, it suffices to check that each iteration of this loop has  output no longer than its input. 

\alg{CutRankII}$_r$ returns $f' = f$ if $u$ has rank less than $r-1$, so assume $\psi_{r-1}$ appears in $u$ so $\rank(u)=r-1$.
If $u\psi_rv(0) = -1$, then  as we show for  $\alg{CutRank}_r$, after each iteration of the loop, there is no increase in length.
If $u\psi_rv(0) \neq -1$,  express $f$ as $\psi_r\inv u' \psi_{r-1} u''\psi_r v$ where $\eta(u')=\eta(u'')=0$,  $\rank(u')<k-1$ and $\rank(u'') \leq  k-1$. 
Substituting  $\psi_1 \psi_{r-1} \psi_{r} \psi_1\inv$  for $\psi_r$ adds $3$ letters. 
There is at least one letter between $\psi_{r-1}\inv$ and $\psi_{r-1}$, so applying $\alg{PinchII}_{r-1}$ then decreases length by at least $3$. 
Hence when $\alg{CutRankII}_r$ does not encounter any special cases in the \textbf{while} loop,  $\ell(f') \leq \ell(f)$.  

\end{proof}

To adapt \alg{FinalPinchII}$_r$ to give \alg{FinalPinch}$_r$:
\begin{enumerate}
\item In line~\ref{finalpinch switch to onetozero}, check whether $u\psi_rv(0) =-1$ and,  if so, run and return the result of $\alg{OneToZeroII}(f)$.
\item In line~\ref{finalpinch induction}, use Lemma~\ref{Lem: OneToZero Equiv bphi} instead of Lemma~\ref{Lem: OneToZero Equiv} to make the analogous substitutions, $\psi_r\inv = \psi_1\psi_{r}\inv\psi_{r-1}\inv \psi_1\inv$ and $\psi_r = \psi_1\psi_{r-1}\psi_{r} \psi_1\inv$.
\end{enumerate}

\begin{algorithm}[H]
\caption{ --- $\alg{FinalPinchII}_r$ for $r \geq 2$.  
 \newline $\circ$ \ Input  a word of the form $\psi_r\inv u \psi_r v$ with $\eta(u)=\eta(v)=0$ and $\rank(u') < r-1$.
 \newline $\circ$ \ Either return $\Invalid$ or return an equivalent word of the form $\psi_1^l v$.
  \newline $\circ$ \ Halt in $O(\ell(f)^{4+(r-2)})$  time.}
\end{algorithm}

\begin{proof}[Correctness of $\alg{FinalPinchII}_r$ assuming correctness of $\alg{PinchII}_r$.]  Consider the special cases:
\begin{itemize}
\item \textbf{$u$ is the empty word}: the argument is similar to the case where $u$ is the empty word in the main routine.  
\item \textbf{$u\psi_rv(0) = -1$ and $u$ is not the empty word}:  the argument is similar to the case where $u$ is the empty word in $\alg{PinchII}_r$.
\item \textbf{$v(0)=0$}: substituting $\psi_1 \psi_r\inv \psi_{r-1}\inv \psi_1\inv$ for $\psi_r\inv$ adds $3$ letters. Substituting for $\psi_r$ by $\psi_{r-1}$ results in no increase in length in this case.
As in $\alg{CutRankII}_r$, the substitution for $\psi_r\inv$ ensures that there is at least one letter between $\psi_{r-1}\inv$ and $\psi_{r-1}$, so if $\alg{PinchII}_r$ returns an equivalent word, that word is at least $4$ letters shorter than the input word by the induction hypothesis.
\item  $u\psi_rv(0)<-1$ and $v(0)<0$: substituting $\psi_1 \psi_{r-1} \psi_r \psi_1\inv$ and $\psi_1 \psi_r\inv \psi_{r-1}\inv \psi_1\inv$ for $\psi_r$ and $\psi_r\inv$, respectively, adds $6$ letters. 
Applying $\alg{PinchII}_{r-1}$ to $$\psi_{r-1}\inv\psi_1\inv u \psi_1\psi_{r-1}\psi_r\psi_1\inv \psi_1\inv v,$$ whose length is at most $\ell(f)+6$.
There are non-trivial letters between $\psi_{r-1}\inv,\psi_{r-1}$. So the equivalent word returned by $\alg{PinchII}_{r-1}$ is at least three letters shorter. 
Therefore, the result is of the form
\[\psi_1 \psi_r\inv \psi_1^l \psi_r \psi_1\inv v \] 
for some $l\in\Z$ and has length at most $\ell(f)+3$.
If $l=0$, running $\alg{BasePinchII}$ triggers a trivial case where $f'=v$ is returned and $\ell(v)\le  \ell(f)-3$ since $u$ is non-empty.
Otherwise, applying $\alg{BasePinchII}$ to $\psi_r\inv \psi_1^l \psi_r \psi_1\inv v$, if an equivalent word of the form $\psi_1^{l'}\psi_1\inv v$ is returned, its length is $4$ letters shorter than the input to $\alg{BasePinchII}$.
Hence we have a word equivalent to $f$ of the form 
\[\psi_1 \psi_1^{l'} \psi_1\inv v\]
whose length is at most $\ell(f)-1$, and the word is equivalent to:
\[\psi_1^{l'} v\]
yielding an equivalent word whose length is at most $\ell(f)-3$.\qedhere
\end{itemize}
\end{proof}
	
	\begin{proof}[Correctness of $\alg{PinchII}_r$ assuming the correctness of $\alg{PinchII}_{r-1}$.]
Correctness can be proved by mimicking our proof of correctness for \alg{Pinch}$_r$. However, the substitution $A_r = A_{r-1}A_r A_0\inv$ for Ackermann functions increases the length of the word by $2$ letters, but the substitution $\psi_r^{\pm 1} = (\psi_1\psi_{r-1}\psi_{r}\inv \psi_1\inv)^{\pm 1}$ increases length by $3$ letters, so we will need to  account carefully for this difference.

When $r=2$, the bound on $\ell(f')$ comes directly from the bound for $\alg{BasePinchII}$. 
 
Let $r\ge 3$.  The calls to $\alg{PositiveII}$ in the main routine are on words no longer than $f$.
We also have the special case where $u$ is the empty word, where the algorithm halts and returns $v$ which has length $\ell(f)-2$.
If $u\psi_kv(0) = -1$ and $u$ is not the empty word, by part of the justification for $\alg{BasePinchII}$, $\psi_kv(0)\le v(0)-3$. 
Since $\eta(u)=0$, the only letter in $u$  that decreases  absolute value  when evaluating $f(0)$ letter-by-letter from right to left is $\psi_1\inv$.
If $u\psi_rv(0) = -1$, then $\psi_rv(0)\le v(0)-3$ by the specifications of $\alg{OneToZeroII}$. 
So  $u$ must contain $\psi_1\inv$ at least $|v(0)|+2$.
Therefore, the   $\ell( \psi_r\inv u \psi_r)  \geq \abs{v(0)}+4$.
Thus   $f'=\psi_1^{v(0)}v$ has $\ell(f') \leq \ell(f)-4$ as required. 
\end{proof}

Correctness and construction of \alg{ReduceII} are nearly immediate by following those of $\alg{Reduce}$, replacing $A_i$ by $\psi_{i+1}$ and changing the subroutines to the $\psi$-word versions. The bound $\ell(f') \leq \ell(f) + 3k$ contrasts with the bound $\ell(w') \leq \ell(w) + 2k$ of  $\alg{Reduce}$ because Lemma~\ref{Lem: OneToZero Equiv bphi} requires a substitution that results in a gain of $3$ letters rather than the gain of $2$ required by Lemma~\ref{Lem: OneToZero Equiv}.

\begin{algorithm}[h]
\caption{ --- $\alg{ReduceII}$.
\newline   
$\circ$ \ Input a $\psi$-word $f$ with $\eta(f)>0$.
 \newline $\circ$ \ Either declare  that $f$ is invalid or  return an equivalent word of the form $f'$ with $\ell(f')\le \ell(f)+3k$ and $\eta(f')=\eta(f)-1$.
  \newline $\circ$ \ Halt in $O(\ell(f)^{4+(k-1)})$  time.}
\label{Alg: ReduceII}
\end{algorithm}

Finally, \alg{Psi} can be constructed similarly to $\alg{Ackermann}$ by replacing all $A_i$ by $\psi_{i+1}$ and replacing subroutines by their counterparts.  The proof of its correctness then essentially follows that of $\alg{Ackermann}$.  (The special case $k=1$ is trivial; we  distinguish it   to make an estimate at the end of Section~\ref{Member construction} cleaner.)

\begin{algorithm}[h]
\caption{ --- \alg{Psi}.
\newline $\circ$ \  Input a $\psi$-word $f$.
 \newline $\circ$ \ Either return that $f$ is invalid,  or return that it is valid and declare whether $f(0) >0$, $f(0) =0$, or $f(0) <0$. 
  \newline $\circ$ \ Halt in $O(\ell(f)^{4+k})$  time when $k>1$ and $O(\ell(f))$  time when $k=1$.}
\end{algorithm}

\section{An efficient solution to  the membership problem for hydra groups}   \label{MP section}

\subsection{Our algorithm in outline} \label{MP algorithm outline} \label{MP problem intro}

Our aim is  to give a polynomial-time algorithm $\alg{Member}_k$ which, given a word $w = w(a_1, \ldots, w_k,t)$ on the generators  of  the hydra group $$G_k \ = \  \cyc{a_1,\ldots,a_k,t  \mid  t\inv a_i t = \theta(a_i)},$$ where $\theta(a_i) = a_i a_{i-1}$ for all $i>1$ and $\theta(a_1)=a_1$, will tell us whether or not $w$ represents an element of  $H_k = \cyc{a_1t,\ldots,a_kt}$.    

The first step is to convert $w$ into a normal form: we use the defining relations for $G_k$ to collect all the $t^{\pm 1}$ at the front, and then we freely reduce, to give  $t^rv$ where  $r$ is an integer with $\abs{r} \leq \ell(w)$  and   $v = v(a_1, \ldots, a_m)$ is reduced.  Pushing a $t^{\pm 1}$ past an $a_i$ has the effect of applying $\theta^{\pm 1}$ to   $a_i$, so it follows from the lemma below that  $$\ell(v)  \ \leq \  \ell(w) (\ell(w) +1)^{k-1}$$ and that   $t^rv$ can be produced in time $O(\ell(w)^k)$.  

\begin{lemma}\label{Lem: automorphism growth}
For all $k =1, 2, \ldots$ and all $n \in \Z$,  $$\ell(\theta^n(a_k))  \  \leq \   (\abs{n}+1)^{k-1}.$$  
\end{lemma} 

\begin{proof}
For $n \in \N$ define $f(n,k):= \ell(\theta^n(a_k))$ and $g(n,k) = \ell(\theta^{-n}(a_k))$.  To establish the lemma we will show by induction on $k$ that $f(n,k)$ and $g(n,k)$ are each at most $(n+1)^{k-1}.$

For the case $k=1$, note that $f(n,1) = g(n,1) = 1$ because $\theta^n(a_1)=a_1$ for all $n \in \Z$. 

For the induction step, consider $k>1$. As  $\theta^{n}(a_k) =  \theta^{n-1}(\theta(a_k)) =  \theta^{n-1}(a_k)\theta^{n-1}(a_{k-1})$, we have
\begin{align*}
f(n,k) & \ = \ f(n-1,k) + f(n-1,k-1) \\
    & \ = \ f(0,k) + f(0,k-1) +  \cdots + f(n-1, k-1)  \\ 
    &  \ \leq \ 1 +   1^{k-2} + \cdots + n^{k-2}  \\
    &  \ \leq \   (n+1)^{k-1} 
\end{align*}
where the first inequality uses   $f(0,k) = \ell( \theta^{0}(a_k)) = \ell(a_k) = 1$ and the induction hypothesis, and the second that each of the $n+1$ terms in the previous line is at most $(n+1)^{k-2}$.

Next, note that $\theta\inv(a_{k}) = a_k \theta\inv(a_{k-1}\inv)$ because $\theta(a_{k}) = a_k a_{k-1}$. So, for all $n \in \Z$ $$\theta^{-n}(a_k) \ = \  \theta^{-(n-1)}\theta\inv(a_k) \ = \  \theta^{-(n-1)}(a_k \theta\inv(a_{k-1}\inv)) \ = \  \theta^{-(n-1)}(a_k) \theta^{-n}(a_{k-1}\inv)$$
and therefore
\[ \ell(\theta^{-n}(a_k)) \ = \ \ell(\theta^{-(n-1)}(a_k)) + \ell(\theta^{-n}(a_{k-1}\inv)) \ = \ \ell(\theta^{-(n-1)}(a_k))+\ell(\theta^{-n}(a_{k-1})).\]
So for all  $n >0$      
\begin{align*}
g(n,k) & \ \le \  g(n-1,k) + g(n, k-1) \\
 & \ \le \ g(0,k) + g(1,k-1) + \cdots +g(n,k-1)   \\  
& \ \leq \ 1 +   1^{k-2} + \cdots + (n+1)^{k-2}   \\
& \ \leq \  (n+1)^{k-1} 
\end{align*}
since $g(0,k) = 1$ and $1+1^{k-2}$ and  each of the other $n$ terms in the penultimate line is at most  $(n+1)^{k-2}$. 
\end{proof}

Next $\alg{Member}_k$  calls a subroutine $\alg{Push}_k$  which `pushes' the power of $t$ back through $v$ from the left to the right (the power varying in the process), leaving the prefix to its left as a word on $a_1t,\ldots,a_kt$.  
 The powers of $t$ that occur as this proceeds are recorded by $\psi$-words,  as they may 
 be too large to record explicitly in polynomial time.

Here are some more details on how we `push the power of $t$  through $v$.'
We do not try to progress the power of $t$ past one $a_i^{\pm 1}$ at a time.  (There are words representing elements of $H_k$ for which that is  impossible.)  Instead, we first consider  the locations of the $a_k^{\pm 1}$, then the  $a_{k-1}^{\pm 1}$, and so on.  
Following \cite{DR}, we define the \emph{rank-$k$ decomposition  of  $v$  into  pieces} as  the (unique) way of expressing $v$ as a concatenation  $\pi_1 \cdots \pi_p$ of the minimal number of    subwords (`\emph{pieces}') $\pi_i$ of the form $a_k^{\epsilon_1}u a_k^{-\epsilon_2}$ where $\Rank(u) \leq k-1$ and $\epsilon_1,  \epsilon_2 \in \{0,1\}$.    For example,   the rank-5 decomposition of    
$$a_5a_3a_5\inv a_2 a_5 a_1 a_5\inv a_1 a_5\inv$$ is $$(a_5a_3a_5\inv) (a_2) (a_5 a_1a_5\inv) (a_1a_5\inv).$$ 

We use pieces because $t^r v \in H_k t^s$ for some $s \in \Z$ if and only if it is possible to advance the power of $t^r$ through $v$ one piece at a time, leaving behind an element of $H_k$.  More precisely, 
 $t^r v\in H_kt^s$ if and only if there exists a sequence $r=r_0,\ldots,r_p=s$ such that $t^{r_i}\pi_{i+1} \in H_kt^{r_{i+1}}$ (Lemma 6.2 of \cite{DR}). 

Let   $f_0 := \psi_1^{-r}$, so $f_0(0)=r$.  Then, for each successive $i$, we determine,  using a subroutine $\alg{Piece}_k$, whether or not there exists  $r_{i} \in \Z$ (unique if it exists) such that 
$$t^{f_{i-1}(0)} \pi_i \  \in \  H_k t^{r_{i}}$$  
 and if so, it gives a  $\psi$-word $f_i$ such that $f_i(0) =r_i$.  $\alg{Piece}_k$ expresses $\pi_i$ as $a_k^{\epsilon_1} u a_k^{-\epsilon_2}$ where $\epsilon_1, \epsilon_2 \in \{0, 1\}$. It operates in accordance with  Proposition~\ref{Prop: Piece Criterion} which is a technical result that we call `The Piece Criterion.'  $\alg{Piece}_k$ has two subroutines.  The first,  $\alg{Front}_k$, reduces the problem of whether $r_i$ exists 
 to determining whether, for a certain  $\psi$-word $f_{i-1}'$  and a   certain rank-$k$  piece $\pi_i'$ which does not have $a_m$ as its first letter,  there exists  $r_i'\in \Z$ such that $t^{f_{i-1}'(0)} \pi_i' \in H_{k-1} t^{r_i'}$. 
 Then the second,  $\alg{Back}_k$, makes a similar reduction to a situation when there is no $a_m\inv$ at the end.  It then   inductively calls $\alg{Push}_{k-1}$ on the modified piece (which is now a word of rank less than $k$)  to find a $\psi$-word $f_i'$ representing $r_i'$, and then modifies $f_i'$ to get $f_i$.  It detects that  the $r_i$ fails to exist  by recognizing (using \alg{Psi}) an emerging   $\psi$-word not  being valid,  or noticing that $\pi_i$ fails to have a suffix or prefix of a particular form. 
 
This inductive construction has base cases $\alg{Push}_1$ and $\alg{Piece}_2$, which use elementary  direct manipulations.
 
If  $r_1, \ldots, r_p$ all exist, then \alg{Psi}  determines whether or not $f_p(0)=0$, and  concludes that $w$  does or does not represent an element of $H_k$, accordingly.

\subsection{Examples} \label{alg examples}

The algorithms and subroutines named here are those we will construct in Section~\ref{Member construction}.

\begin{Example} \label{that follows}

 Let $w   =  a_3^4a_2 t a_1 a_2\inv a_3^{-4}$.  As we saw in Section~\ref{hydra phenomenon},  $w = u_{3,4} \, (a_2t) \,  (a_1 t) \, ({a_2}t)^{-1} \, {u_{3,4}}^{-1}$   in $G_3$ which has length  $2 \H_3(4) +3 =   2^{47} \cdot 3 -1$ as a word on the generators $a_1t$, $a_2t$, $a_3t$ of $H_3$.   Here is how our algorithm $\alg{Member}_k$ discovers that $w$ represents an element of $H_3$ without working with this prohibitively long word.
\begin{enumerate}
\renewcommand{\labelenumi}{\arabic{enumi}. }
\item   Convert    $w$ to a word $tv$ representing the same element of $G_3$ by using that $a_it = t\theta(a_i)$ in $G_3$ for all $i$ to shuffle the $t$ to the front.  This produces  
$$v \ = \  \theta(a_3)^4 \theta(a_2) a_1 a_2\inv a_3^{-4} \ = \ (a_3a_2)^4 a_2  a^2_1 a_2\inv a_3^{-4}.$$
\item Define $f_0:=\psi_1\inv$, to express the power $f(0)=1$ of $t$ here.
\item The rank-3 decomposition of $v$ into pieces is: $$v \  = \  (a_3a_2)(a_3a_2)(a_3a_2)(a_3 a_2^2 a^2_1 a_2\inv a_3\inv)(a_3\inv)(a_3\inv)(a_3\inv).$$
Accordingly, define $$\pi_1 := \pi_2 := \pi_3 := a_3 a_2, \qquad \pi_4 := a_3 a_2^2 a^2_1 a_2\inv a_3\inv, \qquad \pi_5 := \pi_6 := \pi_7 := a_3\inv.$$
A subroutine $\alg{Push}_3$ now aims to find  $\psi$-words $f_1$, \ldots, $f_7$ such that $t^{f_{i-1}(0)} \pi_i  \in H_3 t^{f_i(0)}$ for $i=1, \ldots, 7$, by `pushing the power of $t$ through successive pieces.'   
\item So first a subroutine $\alg{Piece}_3$ is called to try  to pass $t^{f_0(0)}$ through $\pi_1$. 
   The subroutine $\alg{Front}_k$ calls a further subroutine $\alg{Prefix}_3$ to find the longest prefix (if one exists) of $\pi_1$ of the form  $\theta^{i-1}(a_3)a_2$ for some $i \geq 1$.  $\alg{Prefix}_3$ does so  by  generating $\theta^{0}(a_3)a_2$, $\theta^{1}(a_3)a_2$, and so on, and  comparing, until the length of $\pi_1$ is exceeded.   In this instance $\alg{Prefix}_3$ returns $i=1$. It follows from the Piece Criterion that 
   $t^{f_0(0)} \pi_1 =  a_3 t \in H_3 t^0 = H_3 t^{\psi_1 \psi_1\inv(0)}$.  Accordingly define $f_1:=\psi_1 \psi_1\inv$. 
\item   $\alg{Piece}_3$ next looks to pass $t^{f_1(0)} = t^0$ through  $\pi_2$.
   $\alg{Front}_k$ uses \alg{Psi} to check that $f_1(0) =0 \leq 0$.  By the  Piece Criterion, it then follows from the fact that there are no inverse letters in $\pi_2$  that $t  a_3 a_2 \in Ht^{\psi_2\psi_3 (0)}$.  So define  $f_2 := \psi_2\psi_3\psi_1 \psi_1\inv$.   
 \item Next  $\alg{Piece}_3$ tries to pass $t^{f_2(0)}$ through  $\pi_3 = a_3a_2$.  Likewise this is possible as $f_2(0) \leq 0$,  and it defines  $f_3 := (\psi_2\psi_3)^2 \psi_1 \psi_1\inv$.  
\item Next,  $\alg{Piece}_3$ tries to pass $t^{f_3(0)}$ through $\pi_4$.
\begin{enumerate}
\item  $\alg{Front}_3$ uses  \alg{Psi} to check that $f_3(0)\le 0$. It follows that $t^{f_3(0)} a_3 \in H_3 t^{\psi_3f_3(0)}$
and  the problem is reduced (by the Piece Criterion) to finding an $s\in \Z$ (if one exists) such that $$t^{\psi_3 f_3(0)} a_2^2a_1^2 a_2\inv a_3\inv \in  H_3 t^s.$$  This will represent progress as (unlike $\pi_4$)   $a_2^2a_1^2 a_2\inv a_3\inv$ is a piece without an $a_m$ at the front.  
\item  Then the subroutine $\alg{Back}_3$  recursively calls   $\alg{Piece}_2$ 
to find the $s \in \Z$ (if there is one) such that $t^{\psi_3 f_3(0)} a_2^2a_1^2 a_2\inv \in H_3t^s$.  It returns   $\psi_2\inv (\psi_1)^2\psi_2^2\psi_3f_3$. (We omit the steps $\alg{Piece}_2$ goes through.)  \alg{Back}$_3$   then uses $\alg{Psi}$ to test whether $f_4 := \psi_3\inv \psi_2\inv (\psi_1)^2\psi_2^2\psi_3f_3$ is  valid, which  it is: we examined it in Section~\ref{psi word example}.  Also $\alg{Psi}$ declares that $f_4(0) \leq 0$.  It follows (using the Piece Criterion) that $t^{f_3(0)} \pi_4 \in H_3 t^{f_4(0)}$. 
\end{enumerate}
\item Next $\alg{Piece}_3$ tries to pass  $t^{f_4(0)}$ through $\pi_5$.  This is done by  $\alg{Back}_3$.  
By the Piece Criterion, it suffices to check that  $f_5 := \psi_3\inv f_4$ is valid, which is done using $\alg{Psi}$.  
\item $\alg{Piece}_3$  likewise passes   $t^{f_5(0)}$ through $\pi_6$ giving  $f_6 := \psi_3^{-2} f_4$, and then $t^{f_6(0)}$ through $\pi_7$ giving  $f_7 := \psi_3^{-3} f_4$.  
\item Finally, let $g:= f_7$.  We have that $w = t  v \in H_3 t^{g(0)}$.  So use $\alg{Psi}$ to check that $g(0)=0$.   On success, declare  that $w\in H_3$. 
\end{enumerate} 
\end{Example}

In the example above $f_i(0) \leq 0$ for all $i$---we never looked to push a positive power of $t$ through a piece.  
 Next we will see an example of $\alg{Member}_k$ handling such a situation.
 
\begin{Example}
Let $w = t a_3 a_2 t^2 a_1\inv a_2^{-2} a_3\inv a_1^2 t\inv a_3\inv$.  We will show how $\alg{Member}_k$ discovers that $w \in H_3$.  

\begin{enumerate}
\renewcommand{\labelenumi}{\arabic{enumi}. }
\item Shuffle the $t^{\pm 1}$ in $w$ to the front, applying $\theta^{\pm 1}$ to letters they pass, so as to convert  $w$ to the word  $t^2 v$ representing the same element of $G_3$, where $v=a_3a_2^2a_1^2a_2\inv a_3\inv a_1^2 a_3\inv$. 
 Let $f=\psi_1^{-2}$ so that $f(0)=2$ records the power of $t$. 
\item Express $v$ as its the rank-$3$ decomposition  into pieces: $v=\pi_1\pi_2$ where $$\pi_1 := a_3a_2^2a_1^2 a_2\inv a_3\inv, \qquad \pi_2 := a_1^2a_3\inv.$$ 
 Set $f_0:=f$.  $\alg{Push}_3$ now looks for valid  $\psi$-words $f_1$ and $f_2$ such that $t^{f_{0}(0)} \pi_1  \in H_3 t^{f_1(0)}$ and $t^{f_{1}(0)} \pi_2  \in H_3 t^{f_2(0)}$, by twice calling its subroutine  $\alg{Piece}_3$.
 \item $\alg{Piece}_3$ calls $\alg{Front}_3$ to `try to move $t^{f_0(0)}$ past  $\pi_1$.'
 As $a_3$ is the first letter of $\pi_1$,  $\alg{Front}_3$ calls $\alg{Psi}$ to   determine the sign of $f_0(0)$, which is positive.   The Piece Criterion then says that to pass $t^{2}$ past $a_3$ requires that $\pi_1$ has a prefix  $\theta^{i-1}(a_3)a_{2}$ for some $i$ which is `approximately' $\theta^2(a_3) = a_3a_2^2a_1$. The subroutine $\alg{Prefix}_3$ looks for this prefix by generating $\theta^{0}(a_3)a_{2} = a_3 a_2$, then $\theta^{1}(a_3)a_{2}=a_3 a_2^2$, then $\theta^{2}(a_3)a_{2} = a_3a_2^2a_1a_2$, and so on, until the length of $\pi$ is exceeded, and comparing with the start of $\pi_1$.  
Here, $a_3 a_2$ and $a_3 a_2^2$ are prefixes of $\pi_1$, but $a_3a_2^2a_1a_2$ is not, and $\alg{Prefix}_3$ returns $i=2$.
\item Call $\alg{Psi}$ to check that $i$ is at least $f_0(0)=2$. 
\item Intuitively speaking, as this prefix $a_3 a_2^2$ is `approximately' $\theta^2(a_3)$,   the length of the `correction' $a_1 a_1\inv$ that has to be made for the discrepancy between $\theta^2(a_3)$ and the prefix $a_3 a_2^2$ is minimal compared to the length of the prefix that the power of $t$ advances  past.  In this instance:
 $$t^2 \pi_1 \ = \ t^2  \theta^2(a_3)  a_1 a_1\inv    a_1a_2\inv a_3\inv  \ = \  (a_3 t) t   a_1a_2\inv a_3\inv.$$ \label{prefix trick}
and have reduced the problem  to pushing $t$ past $a_1a_2\inv a_3\inv$. The power of $t$ being advanced through the word is now $t^1$, and this is recorded by 
$\psi_1 f_0$, as $\psi_1 f_0(0)=1$.
\item Next $\alg{Piece}_3$ calls $\alg{Back}_3$ on input $a_1a_2\inv a_3\inv$ and $\psi_1 f$  to try to advance  $t$ past $a_1a_2\inv a_3\inv$.  

\item First, it searches for an $s\le 0$ such that $t a_1a_2\inv a_3\inv \in H_k t^s$.  It calls $\alg{Push}_2$, which calls $\alg{Piece}_2$ to attempt to push $t$  through   $a_1a_2\inv$.   $\alg{Piece}_2$ calls $\Psi$ to find out whether $\psi_2\inv\psi_1\psi_1f$ is  valid.
It is not, and it follows from the Piece Criterion that there is no $s \leq 0$ such that $t a_1a_2\inv a_3\inv \in H_k t^s$. 
\item So, instead $\alg{Piece}_3$   searches for an $s > 0$ such that $t a_1a_2\inv a_3\inv \in H_k t^s$ or, equivalently, 
 $t^s a_3 a_2 a_1\inv \in H_3 t$. 
\item  We check for $s=1,2,\ldots$  whether we can move $t^s$ past  $a_3a_2a_1\inv$. Use the same approach that we used for the prefix in Step~\ref{prefix trick}.    First try $s=1$.  Detect the prefix $a_3a_2$ of $a_3a_2a_1\inv$ and as, $ta_3a_2 = t\theta(a_3) = (a_3 t)\in H_3$,   the problem reduces to determining whether 
$t^0 a_1\inv \in H_3t$ or, equivalently,   $t a_1 \in H_3 t^0$.  This shown to be the case by $\alg{Push}_2$ which finds that $t  a_1 = (a_1 t) \in H_3$ and returns  $\psi_1\psi_1 f$, which satisfies $\psi_1\psi_1 f(0)=0$, to indicate the coset $H_3t^0$ of $H_3$.  Finally, $\alg{Back}_3$ checks that 
$H_3t^0 = H_3 t^{\psi_1 \psi_1 f_0}$ by calling $\alg{psi}$ on  $\psi_1^{0} \psi_1 \psi_1 f_0(0)=0$, and returns $f_1: = \psi_1\inv \psi_1^2 f_0$ (which satisfies $f_1(0) =1$) to indicate that  $\pi_1\in H_3t^{f_1(0)}$.

(In this instance, we were successful with $s=1$, but in general, we may have to repeat the process   for $s =2, 3, \ldots$.  This does not continue indefinitely: we can stop when $s$ exceeds the length of of the word inputted into $\alg{Back}_3$ because the prefixes we check for must be no longer than that word.)

\item We now seek to pass $t^{f_1(0)}$ through $\pi_2$ by another call on $\alg{Piece}_3$. Recall $\pi_2=a_1^2a_3\inv$ and $f_1: = \psi_1\inv \psi_1^2 f_0$, and $f_1(0) =1$.

\item $\alg{Piece}_3$ first calls $\alg{Front}_3$ but the first letter of $\pi_2$ is not $a_3$, so $\alg{Front}_3$ does nothing.  

\item $\alg{Piece}_3$ then calls $\alg{Back}_3$.  It first looks  for $s \leq 0$ such that $t^{f_1(0)} \pi_2\in H_kt^s$, which it succeeds in finding as follows. 
\begin{enumerate}
\item $\alg{Push}_2$ tries to pass $t^{f_1(0)}$ through $a_1^2$, which is elementary  since $a_1$ commutes with $t$: $ta_1^2   =    (a_1 t)(  a_1t)    t^{-1}$ and so  $\alg{Push}_2$ returns $\psi_1^2 f_1$, representing $\psi_1^2 f_1(0)= -1$. 
\item Call $\alg{Psi}$ to check that $\psi_1^2f_1$ is valid.  Then  to pass $t$ through $a_3\inv$,  call $\alg{Psi}$ to check that $\psi_3\inv \psi_1^2 f_1$ is valid.   Return $f_2:=\psi_3\inv \psi_1^2 f_1$ to indicate that  $t^{f_1(0)} \pi_2\in H_3 t^{f_2(0)}$.
\end{enumerate}
\item $\alg{Member}_3$  checks that $f_2(0)=0$ and  declares that $w\in H_3.$
\end{enumerate}
\end{Example}
These   examples illustrate the tests  $\alg{Member}_k$ uses and give  a  sense of how it works in general.  But, it is difficult to show that these tests amount to  the \emph{only} conditions under which a word $t^r v$ is in $Ht^s$ for some $s \in \Z$.  A result we call the `Piece Criterion' is at the heart of that and presentation and  proof of is involved and will occupy the next two sections.

\subsection{Constraining cancellation} \label{Constraining cancellation}

This section contains   preliminaries toward Proposition~\ref{Prop: Piece Criterion}  (The Piece Criterion), which will be  the subject of the next section.  

When discussing words representing elements of $F(a_1, \ldots, a_m)$, we use $\theta^r(a_m^{\pm 1})$, for $m \geq 1$ and $r\in \Z$, to refer to the freely reduced word on $a_1,\ldots, a_m$ equal to $\theta^r(a_m^{\pm 1})$. The following lemma will be useful for calculating with iterations of $\theta$.
\begin{lemma}\label{Lem: Theta breakdown}
If $r>0$ and $m>1$, then
\begin{equation}
 \theta^r(a_m) \ = \ a_m \theta^0(a_{m-1})\theta^1(a_{m-1}) \theta^2(a_{m-1}) \cdots \theta^{r-1}(a_{m-1})  \label{expand theta1}
 \end{equation}
as words.  Moreover, if $r < m$, then the final letter of $\theta^r(a_m)$ is $a_{m-r}$, and if $r \geq m$, then $\theta^{r-m+1}(a_1)=a_1$, $\theta^{r-m+2}(a_{2})$, \ldots, $\theta^{r-1}(a_{m-1})$ are all suffixes of $\theta^r(a_m)$.

If $r< 0$ and $m>1$, then 
\begin{equation}
\theta^r(a_m)  \ = \  a_m \theta\inv(a_{m-1}\inv)\theta^{-2}(a_{m-1}\inv) \cdots \theta^{r}(a_{m-1}\inv), \label{expand theta3}
\end{equation}
as words, and its first letter is $a_m$ and its final letter  is   $a_{m-1}\inv$.
\end{lemma}

\begin{proof}
For \eqref{expand theta1}, observe that the identity $\theta^r(a_m) = \theta^{r-1}(a_m)\theta^{r-1}(a_{m-1})$  and   inducting on $r$ gives that the words are equal in the free group.  The words are identical because that on the right is positive (that is, contains no inverse letters) and so is freely reduced.
If $r<m$, the same identity shows that the final letter of $\theta^r(a_m)$, is the same as that of $\theta^{r-1}(a_{m-1})$, and so   the same as that of $\theta^{r-2}(a_{m-2})$, \ldots, and of $\theta^{r-r}(a_{m-r})=a_{m-r}$.  If, on the other hand, $r \geq m$, then \eqref{expand theta1} shows that $\theta^{r-1}(a_{m-1})$ is a suffix of $\theta^r(a_m)$, and therefore, so are  
 $\theta^{r-2}(a_{m-2})$, $\theta^{r-3}(a_{m-3})$, \ldots,   $\theta^{r-m+1}(a_{1})$.

   Lemma 7.1 in \cite{DR} tells us that the two words in \eqref{expand theta3} are freely equal.  Induct on $m$ as follows to establish the remaining claims.  
In the case $m=2$ we have $$\theta^r(a_2) \ = \ a_2 \theta^{-1}(a_1^{-1})  \theta^{-2}(a_1^{-1})  \cdots \theta^{r}(a_1^{-1}) \ = \ a_2 a_1^{r},$$ and the result holds. For $m>2$,  the induction hypothesis tells us that the first letter of each subword $\theta^{-i}(a_{m-1}^{-1})$ is    $a_{m-2}$ and  the final letter is $a_{m-1}\inv$, and it follows that the  word on the right of  \eqref{expand theta3}    is freely reduced.  It is then  evident that its first letter is $a_m$ and its final letter  is   $a_{m-1}\inv$.
\end{proof}

The remainder of this section concerns words $w$ expressed as 
$$w \ = \  \theta^{e_0}(a_{i_0}^{\epsilon_0})  \theta^{e_1}(a_{i_1}^{\epsilon_1})  \cdots   \theta^{e_{l+1}}(a_{i_{l+1}}^{\epsilon_{l+1}})$$
where   $\epsilon_{x}\in\{\pm 1\}$ for  $x=0, \ldots, l+1$,  and $a^{\epsilon_x}_{i_x} \neq a^{-\epsilon_{x+1}}_{i_{x+1}}$  and
\begin{equation} e_{x+1} = \begin{cases}
e_x & \text{if } \epsilon_x = -\epsilon_{x+1} \\
e_x-1 & \text{if } \epsilon_x = \epsilon_{x+1} = 1 \\
e_x+1 & \text{if } \epsilon_x = \epsilon_{x+1} = -1 \\
\end{cases} \label{condition on e_i}
\end{equation} 
for  $x= 0, \ldots, l$.  We refer to the $a_{i_0}^{\epsilon_0},  \ldots, a_{i_{l+1}}^{\epsilon_{l+1}}$ in the subwords $\theta^{e_0}(a_{i_0}^{\epsilon_0})$,  $\theta^{e_1}(a_{i_1}^{\epsilon_1})$,  \ldots,   $\theta^{e_{l+1}}(a_{i_{l+1}}^{\epsilon_{l+1}})$ of $w$ as the \emph{principal letters} of $w$. 

\begin{lemma}\label{Lem: H cancellation}
 If $w$ (as above) freely equals the empty word, then   $a_{i_x} = a_{i_{x+1}}$ and $\epsilon_{i_x} = -\epsilon_{x+1}$ for some $0 \leq x < l+1$.
\end{lemma}

\begin{proof}
The point of the hypotheses is that  $w$ is  the word obtained by shuffling all $t^{\pm 1}$
rightwards in 
$$\begin{cases}
 t^{-e_0} (a_{i_0}t)^{\epsilon_0}\cdots (a_{i_{l+1} }t)^{\epsilon_{j+1}} & \text{if }\epsilon_0=1 \\
t^{-e_0+1} (a_{i_0}t)^{\epsilon_0}\cdots (a_{i_{l+1} }t)^{\epsilon_n} & \text{if }\epsilon_0=-1, \end{cases}$$
and then discarding the power of $t$  that emerges on the right.  

Now $(a_{i_0}t)^{\epsilon_1}\cdots (a_{i_{l+1} }t)^{\epsilon_{l+1}} = 1$ in $H_k$ because  $w =1$ in $G_k$ 
and   $H_k \cap \langle t \rangle = \set{1}$ (Lemma~6.1   in \cite{DR}).  The result then follows from the fact that  $H_k$ is free on $a_1t, \ldots, a_kt$  (Proposition~4.1 in \cite{DR}).  
\end{proof}

The following definition and Proposition~\ref{types prop} concerning it are for analyzing free reduction of  $w$.  They will be used in our proof of Proposition~\ref{Prop: set prefix}, where we will subdivide a word such as $w$ into  subwords of certain types and argue that all free reduction is contained within them.   There are two ideas behind the definitions of these types.  One is that the rank-1 and rank-2 letters are the most awkward for understanding free reduction, but in these subwords such letters are \emph{controlled} by being buttressed by higher rank words.  The other idea concerns where new letters appear when $\theta^{\pm 1}$ is applied to some $a^{\pm 1}_n$.  It is evident from the definition of $\theta$ that when $i \geq 0$,  the lower rank letters produced by applying $\theta^i$ to $a_n$ or $a_n\inv$ appear to the right of $a_n$ and to the left of $a_n\inv$.  The same is true when $i <0$ --- see Lemma 7.1 of \cite{DR}.

\begin{defn} \label{five types}
  We will define various \emph{types} a subword   $$z \ = \ \theta^{e_p}(a^{\epsilon_p}_{i_p})   \cdots \theta^{e_q}(a^{\epsilon_q}_{i_q})$$ of $w$ may take, and will denote the freely reduced form of $z$ by $z'$.  To the left, below, are the conditions that define the types.  To the right are facts established in the proposition that follows:   what $z'$  is in cases~ii and ii${}^{\inv}$,  and   prefixes and suffixes  it has in  cases~\emph{i}--\emph{iv}.  When it appears below, $u$ denotes a (possibly empty) subword $\theta^{e_x}(a^{\epsilon_x}_{i_x})  \cdots \theta^{e_y}(a^{\epsilon_y}_{i_y})$ such that $i_x, \ldots, i_y \leq 2$.     

$\begin{array}{rlll}
\text{(\emph{i})} & \parbox{46mm}{$\epsilon_p =1, \  \  \epsilon_q =-1$}  &   z & \hspace*{-3mm}  =  \theta^{e_p}(a_{i_p}) u \theta^{e_q}(a_{i_q}\inv) \\
\rule{8mm}{0mm} & i_p,i_q\ge 3, \  \ i_{p+1}, \ldots, i_{q-1} \leq 2 &  z' & \hspace*{-3mm} =  \theta^{e_{p}-1}(a_{i_p})  \rule[1.5pt]{25pt}{0.5pt} \, a_{i_q}\inv   \ \  \textup{if } e_p >0   \\   
&  e_p,e_q\ge 0  &  &  \hspace*{-3mm}  =  a_{i_p} \rule[1.5pt]{25pt}{0.5pt} \, a_{i_q}\inv  \ \  \textup{for } e_p \geq 0  \\
\end{array}$ 

$\begin{array}{rlll}
\text{(\emph{ii})} &  \parbox{46mm}{$\epsilon_p, \ldots, \epsilon_q = 1$} &  z & \hspace*{-3mm}  = \theta^{e_p}(a_{i_p}) \cdots \theta^{e_q}(a_{i_q})   \\
\rule{8mm}{0mm} 
&  i_p \ge 3, \, i_q\ge 2      &  z' & \hspace*{-3mm} =     \theta^{e_p+1}(a_{i_p}) \theta^{e_q}(a_{i_q -1}\inv) \\ 
&   i_j = i_{j+1} + 1 \text{ for } j=p, \ldots, q-1     &  &  \hspace*{-3mm}  =  a_{i_p} \rule[1.5pt]{25pt}{0.5pt} \, a_{i_q -1}\inv  \\ 
& e_p <0    \\ 
&  (\text{so } e_{p+1}, \ldots, e_{q} <0  \text{ by \eqref{condition on e_i}})  & \\
\end{array}$ 

$\begin{array}{rlll}
\text{(\emph{ii}${}^{-1}$)} &  \parbox{46mm}{$\epsilon_p, \ldots, \epsilon_q = -1$} & z & \hspace*{-3mm}  = \theta^{e_p}(a_{i_p}\inv) \cdots \theta^{e_q}(a_{i_q}\inv) \\
\rule{8mm}{0mm} &   i_q\ge 3,\, i_p\ge 2   &  z' & \hspace*{-3mm} =  \theta^{e_p}(a_{i_p -1}) \theta^{e_q+1}(a_{i_q}\inv)  \\ 
\rule{8mm}{0mm} &  i_{j} = i_{j-1} + 1 \text{ for } j=p+1, \ldots, q  & & \hspace*{-3mm}  =  a_{i_p-1} \rule[1.5pt]{25pt}{0.5pt} \, a_{i_q}\inv  \\ 
& e_q <0  \\  
& (\text{so } e_{p}, \ldots, e_{q-1} <0  \text{ by \eqref{condition on e_i}})   & \\
\end{array}$

$\begin{array}{rlll}
\text{(\emph{iii})} &  \parbox{46mm}{$p <    q' \leq q $}  &  z & \hspace*{-3mm} = \theta^{e_p}(a_{i_p}) u \theta^{e_{q'}}(a_{i_{q'}}\inv) \cdots\theta^{e_q}(a_{i_q}\inv)   \\
\rule{8mm}{0mm} & \epsilon_p =1, \ \ \epsilon_{q'}, \ldots, \epsilon_q = -1   &  z' &  \hspace*{-3mm} = \theta^{e_{p}-1}(a_{i_p}) \rule[1.5pt]{25pt}{0.5pt} \, a_{i_q}\inv   \ \  \textup{if } e_p > 0   \\ 
& i_{p}, i_{q'}, \ldots, i_{q} \geq 3, & &   \hspace*{-3mm}  =  a_{i_p} \rule[1.5pt]{25pt}{0.5pt} \, a_{i_q}\inv  \ \  \textup{for } e_p \geq 0 \\ 
&    i_{p+1}, \ldots, i_{q'-1} < 3    \\
&  i_j = i_{j-1} + 1 \text{ for } j=q'+1,\ldots,q    & \\
&  e_p \geq 0, \ \ e_q <0  \ \  \\ &  (\text{so } e_{q'}, \ldots, e_{q-1} <0  \text{ by \eqref{condition on e_i}})  & \\
\end{array}$ 

$\begin{array}{rlll}
\text{(\emph{iii}${}^{-1}$)} &  \parbox{46mm}{$ p \leq    p' < q $} &  z & \hspace*{-3mm} = \theta^{e_{p}}(a_{i_{p}}) \cdots\theta^{e_{p'}}(a_{i_{p'}})  u \theta^{e_q}(a_{i_q}\inv)  \\
\rule{8mm}{0mm} &    \epsilon_{p}, \ldots, \epsilon_{p'} =-1, \ \  \epsilon_q =1   &  z' & \hspace*{-3mm} =   a_{i_p} \rule[1.5pt]{25pt}{0.5pt} \, a_{i_q}\inv \\  
& i_{p}, \ldots, i_{p'}, i_q \ge 3  & &     \\
&  i_j = i_{j+1} + 1 \text{ for } j=p,\ldots,p'-1     & \\
&  e_p < 0, \ \ e_q  \geq 0    \ \ \\  
& (\text{so } e_{p+1}, \ldots, e_{p'} <0 \text{ by \eqref{condition on e_i}})  & \\ 
\end{array}$ 

$\begin{array}{rlll}
\text{(\emph{iv})} &  \parbox{46mm}{$ p \leq    p' < q' \leq q  $} & z & \hspace*{-3mm} =  \theta^{e_p}(a_{i_p})\cdots\theta^{e_{p'}}(a_{i_{p'}})  u \theta^{e_{q'}}(a_{i_{q'}}\inv) \cdots \theta^{e_q}(a_{i_q}\inv)  \\
\rule{8mm}{0mm} & \epsilon_{p}, \ldots, \epsilon_{p'} =1, \ \ \epsilon_{q'}, \ldots, \epsilon_q = -1  &  z' & \hspace*{-3mm} =  a_{i_p} \rule[1.5pt]{25pt}{0.5pt} \, a_{i_q}\inv    \\ 
&  i_p, \ldots, i_{p'},i_{q'}, \ldots, i_q \ge 3 & \\
&     i_{p'+1}, \ldots, i_{q'-1} < 3 & \\
&  {i_j = i_{j+1} + 1 \text{ for } j=p,\ldots,p'-1  }   & \\
&   {i_j = i_{j-1} + 1 \text{ for } j=q'+1,\ldots,q  }  & \\
&  e_p,  e_q < 0  \ \    & \\
&  (\text{so } e_{p+1}, \ldots, e_{p'} <0  \\
&  \text{and } e_{q'}, \ldots, e_{q-1}  <0 \text{ by \eqref{condition on e_i}})  & \\
\end{array}$

$\begin{array}{rlll}
\text{(\emph{v})} &  \parbox{46mm}{\text{For no } $0 \leq p' <  q' \leq l+1$ } &  z & \hspace*{-3mm} = \theta^{e_{p}}(a_{i_p}^{\epsilon_{p}}) \cdots \theta^{e_{q}}(a_{i_q}^{\epsilon_{q}}) \\
\rule{8mm}{0mm} &  \text{with } p \leq q' \leq q   &  z' & \hspace*{-3mm} =   \theta^{e_{p}-1}(a_{i_p}) \rule[1.5pt]{25pt}{0.5pt}   \ \  \text{ if }  \epsilon_p=1, i_p\ge 3 \text{ and } e_p >0     \\
\rule{8mm}{0mm} &  \text{is }   \theta^{e_{p'}}(a_{i_{p'}}^{\epsilon_{p'}}) \cdots \theta^{e_{q'}}(a_{i_{q'}}^{\epsilon_{q'}}) \\   
\rule{8mm}{0mm} & \text{one of the above types.} & \\ 

\end{array}$
\end{defn}

\begin{prop} \label{types prop} 
In types \textit{i}, \textit{ii}${}^{\pm1}$, \textit{iii}${}^{\pm1}$, \textit{iv} and \textit{v}  the form of $z'$ is as indicated in Definition~\ref{five types}.  In type \textit{v},    no letter of rank $3$ or higher in $z$ cancels away  on free reduction to $z'$.
\end{prop}

\begin{proof}[Proof of Proposition~\ref{types prop} in type i]

We have $$z \ = \ \theta^{e_p}(a_{i_p}) u \theta^{e_q}(a_{i_q}\inv)$$ where $i_p,i_q \geq 3$, and $e_p,e_q \geq 0$, and $u$ is a subword of $w$ of rank at most $2$.   
By definition 
\begin{equation}
u \ = \ \theta^{e_{p+1}}(a^{\epsilon_{p+1}}_{i_{p+1}})   \cdots \theta^{e_{q-1}}(a^{\epsilon_{q-1}}_{i_{q-1}}), \label{u}
\end{equation}
and by Lemma~\ref{Lem: H cancellation}, no $a_2$ and $a_2\inv$ can cancel in the process of freely reducing $u$.    
We aim to show that the first and last letters of the freely reduced form $z'$ of $z$ are $a_{i_p}$ and $a_{i_q}\inv$, respectively, and that if  $e_p >0$, then  $\theta^{e_{p}-1}(a_{i_p}) a_{i_{p}-1}$ is a prefix of $z'$.  We will also show that if $e_q >0$, then $a_{i_{q}-1}\inv \theta^{e_q-1}(a_{i_q}\inv)$ is a suffix of $z'$.  This is more than claimed in the proposition, but having a conclusion that is `symmetric' with respect to inverting  $z'$ will expedite   our proof.   

 We   organize our proof by cases.

\renewcommand{\labelenumi}{\arabic{enumi}.  }
\renewcommand{\labelenumii}{\arabic{enumi}.\arabic{enumii}. }
\renewcommand{\labelenumiii}{\arabic{enumi}.\arabic{enumii}.\arabic{enumiii}. }
\renewcommand{\labelenumiv}{\arabic{enumi}.\arabic{enumii}.\arabic{enumiii}.\arabic{enumiv}. }

\begin{enumerate}
 
\item \emph{Case:  $u$ freely equals the empty word.}  In this case $u$ is empty else Lemma~\ref{Lem: H cancellation} (applied to $u$ rather than to $w$) would be contradicted.  So $z = \theta^{e_p}(a_{i_p})\theta^{e_q}(a_{i_q}\inv)$ and by \eqref{condition on e_i}, $e_p=e_q$. Now $\theta^{e_p}(a_{i_p})$  contains an $a_2$ if and only if $i_p-2\le e_p$, and in that event $\theta^{e_p-i_p+2}(a_{2}) =a_2 a_1^{e_p-i_p+2}$ is a suffix of $\theta^{e_p}(a_{i_p})$.    
Similarly, $\theta^{e_q}(a_{i_q}\inv)$ contains an $a_2\inv$  if and only if $i_q-2\le e_q$, and in that event $\theta^{e_q-i_q+2}(a_{2}) = a_1^{-(e_q -i_q +2)} a_2\inv$ is a prefix of $\theta^{e_q}(a_{i_q}\inv)$.   If $i_p -2 > e_p$, then $i_p > e_p$, and so the final letter of $\theta^{e_p}(a_{i_p})$ is $a_{i_p-e_p}$.    Likewise,  if $i_q -2 > e_q$, then    $a_{i_q-e_q}\inv$ is the first letter of  $\theta^{e_q}(a_{i_q}\inv)$.

\begin{enumerate}
\item \emph{Case:   cancellation occurs between some letters  $a_2^{\pm 1}, \ldots, a_k^{\pm 1}$  when $z$ is freely reduced to $z'$.}  If  $i_p-2\le e_p$, then the final $a_2$   in     $\theta^{e_p}(a_{i_p})$ must cancel with the first   $a_2\inv$ in  $\theta^{e_q}(a_{i_q}\inv)$.  So $i_q-2\le e_q $,  and  the whole suffix $a_2 a_1^{e_p-i_p+2}$ of $\theta^{e_p}(a_{i_p})$ cancels with the whole prefix  $a_1^{-(e_q -i_q +2)} a_2\inv$ of $\theta^{e_q}(a_{i_q}\inv)$.  But that implies that $i_p =i_q$ (since $e_p=e_q$), which is a contradiction.  If, on the other hand, $i_p-2 > e_p$, then $i_q -2 > e_q$, and  the last and first letters  $a_{i_p-e_p}$  and  $a_{i_q-e_q}\inv$  of $\theta^{e_p}(a_{i_p})$ and  $\theta^{e_q}(a_{i_q}\inv)$, respectively, must be mutual inverses, and so again  we get the contradiction $i_p=i_q$.   

\item \emph{Case: no cancellation occurs between letters $a_2^{\pm 1}, \ldots, a_k^{\pm 1}$   when $z$ is freely reduced to $z'$.}  
If $i_p -2 > e_p$ or $i_q -2 > e_q$, then the last letter of  $\theta^{e_p}(a_{i_p})$ or the first letter of  $\theta^{e_q}(a_{i_q}\inv)$, respectively, has rank greater than $2$ and so is not cancelled away, and therefore $z'=z$.   
If $i_p -2 \leq e_p$ and $i_q -2 \leq e_q$, then there is only cancellation between some of the $a_1^{e_p-i_p+2}$ at the end of $\theta^{e_p}(a_{i_p})$  and  some of the $a_1^{-(e_q -i_q +2)}$ at the start of $\theta^{e_q}(a_{i_q}\inv)$ (but not all as $i_p \neq i_q$).  In either event  the first and last letters of $z'$  are $a_{i_p}$ and $a_{i_q}\inv$, respectively.  Moreover, if $e_p >0$, then  $\theta^{e_{p}-1}(a_{i_p}) a_{i_{p}-1}$ is a prefix of $z'$ as $a_{i_{p}-1}$ has rank at least $2$ and so is not cancelled away.  Likewise,  if $e_q >0$,   then $a_{i_{q}-1}\inv \theta^{e_q-1}(a_{i_q}\inv)$ is a suffix of $z'$.
\end{enumerate}
 
\item \emph{Case:    $u$ does not freely equal the empty word.}  

\begin{enumerate}

\item \emph{Case:  no letter  $a_3^{\pm 1}, \ldots, a_k^{\pm 1}$ in $z$ is cancelled away when  $z$ is freely reduced to give $z'$.}  The first and last letters, $a_{i_p}$ and $a_{i_q}\inv$, of $z$ are also  the first and last letters of $z'$, because $i_p,i_q \geq 3$.    
Here is why the prefix $\theta^{e_{p}-1}(a_{i_p}) a_{i_{p}-1}$ of $z$ survives in $z'$ when $e_p>0$.   If $i_p \geq 4$, then its final letter   $a_{i_{p}-1}$ has rank at least $3$ and so is not cancelled away.   Suppose then that $i_p=3$,  so that the prefix $$\theta^{e_p}(a_{i_p})  \ = \  \theta^{e_p}(a_3) \ = \  \theta^{e_p-1}(a_3) \theta^{e_p-1}(a_2) \ = \  \theta^{e_p-1}(a_3)  a_2 a_1^{e_p-1}.$$ We must show that the $a_2$ of  $\theta^{e_p-1}(a_3)  a_2$ is not cancelled away when $z$ is freely reduced to $z'$.  Suppose it is cancelled away.  Then $u$ must have a prefix freely equal to $a_1^{-(e_p-1)}a_2\inv$ (since no $a_2$ and $a_2\inv$ can cancel when $u$ freely reduces).  But $u$ has the form \eqref{u}, and by a calculation we will see in a more extended form in \eqref{u revealed},   $a_1^{-e_p + 2m_1} a_2\inv$  freely equals a prefix of $u$ for some integer $m_1$. But then $-(e_p-1) = -e_p + 2m_1$, contradicting $m_1$ being an integer.   Conclude that $\theta^{e_p-1}(a_3)  a_2$  is a prefix of $z'$ as required.  Likewise, if $e_q >0$,   then $a_{i_{q}-1}\inv \theta^{e_q-1}(a_{i_q}\inv)$ is a suffix of $z'$.

\item \emph{Case:  some letter  $a_3^{\pm 1}, \ldots, a_k^{\pm 1}$ in $z$ is cancelled away when  $z$ is freely reduced to give $z'$.} The prefix $\theta^{e_p}(a_{i_p})$ of $z$ is a positive word and the suffix $\theta^{e_q}(a_{i_q}\inv)$  is a negative word since $e_p, e_q \geq 0$. 

There is an $a_3$ in  $\theta^{e_p}(a_{i_p})$ if and only if $e_p - i_p +3 \geq 0$. Likewise there is an $a_3\inv$ in $\theta^{e_q}(a_{i_q}\inv)$ if and only if $e_q - i_q +3 \geq 0$.  

\begin{enumerate}

\item \emph{Case:  $e_p - i_p +3 < 0$.} The last letter of  $\theta^{e_p}(a_{i_p})$ (a positive word) has rank greater than $3$ and so must cancel.  So    $e_q - i_q +3 < 0$ also, as otherwise $\theta^{e_q}(a_{i_q}\inv)$ (a negative word) the leftmost letter in  $\theta^{e_q}(a_{i_q}\inv)$ with rank at least $3$ would be an $a_3\inv$, which would block any cancelation of other letters $a_3^{\pm 1}, \ldots, a_k^{\pm 1}$ in $z$.   So, in fact, the last letter of $\theta^{e_p}(a_{i_p})$ must cancel with the first letter of $\theta^{e_q}(a_{i_q}\inv)$, and so  $u$ must equal freely the identity, which is a case addressed above.  

\item \emph{Case:   $e_q - i_q +3 < 0$.}  Likewise, this reduces to the earlier case.  

The remaining possibility is:

\item \emph{Case:  $e_p - i_p +3 \geq 0$ and $e_q - i_q +3 \geq 0$.}  So $\theta^{e_p}(a_{i_p})$ has suffix
$$\theta^{e_p - i_p +3}(a_3)  \ = \  a_3a_2 \, a_2a_1 \, a_2a_1^2 \, \cdots a_2 a_1^{e_p - i_p +2}$$
and  $\theta^{e_q}(a_{i_q}\inv)$ has  prefix
$$\theta^{e_q - i_q +3}(a_3\inv)  \ =  \    a_1^{-(e_q - i_q +2)} a_2\inv \, \cdots \, a_1^{-2} a_2\inv \, a_1\inv a_2\inv \, a_2\inv a_3\inv $$   
 and the    subword 
\begin{equation}
\theta^{e_p - i_p +3}(a_3)   u  \theta^{e_q - i_q +3}(a_3\inv) \label{disappearing subword}
\end{equation} of $z$  freely equals the identity.  Now    $u$ has rank at most $2$,    so
$$u \  = \   a_1^{f_1} a_2\inv a_1^{f_2} a_2\inv \cdots a_1^{f_{\lambda}} a_2\inv a_1^{\xi} a_2 a_1^{g_{\mu}}  \cdots a_2 a_1^{g_2} a_2 a_1^{g_1}$$ 
for some $\lambda, \mu \geq 0$, some $\xi  \in \Z$, some   $f_1, \ldots, f_{\lambda}  \leq 0$, and some $g_1, \ldots, g_{\mu} \geq 0.$ 
And because of   cancellations that must occur,
$$\begin{array}{rlrl}
f_1    & =  \ -(e_p - i_p +2)  & g_1  & =  \ e_q - i_q +2  \\
f_2     & =  \ -(e_p - i_p +1)  & g_2  & =  \ e_q - i_q +1  \\
  &  \ \vdots & & \ \vdots  \\ 
f_{\lambda}     & =  \ -(e_p - i_p +3 - \lambda)  \qquad & g_{\mu}  & =  \ e_q - i_q +3 - \mu.
     \end{array}$$
These cancellations reduce  $\theta^{e_p - i_p +3}(a_3)   u  \theta^{-(e_q - i_q +3)}(a_3)$
to   $$a_3a_2 \, a_2a_1 \, a_2a_1^2 \, \cdots  a_2 a_1^{e_p - i_p +2 -\lambda} \  a_1^{\xi} \   a_1^{-(e_q - i_q +2 - \mu)} a_2\inv  \cdots  \, a_1^{-2} a_2\inv  \, a_1\inv a_2\inv  \, a_2\inv a_3\inv.$$

 As this   freely equals the identity, the exponent sum of the $a_2^{\pm 1}$ is zero, and so
\begin{equation}\label{balancer eqn2}
e_p-i_p+3 -  \lambda  \ = \ e_q - i_q+3 -  \mu.
\end{equation}
Also, as the $a_1^{\pm 1}$ between the rightmost $a_2$ and the leftmost $a_2\inv$  cancel, 
\begin{equation}\label{balancer eqn}
 e_p-i_p+2 + \mu + \xi \ = \ e_q-i_q +2  + \lambda.
\end{equation}
Together \eqref{balancer eqn2} and \eqref{balancer eqn}  tell  us that  $\xi=0$.  
But then $\lambda=0$ or $\mu=0$ because of the hypothesis $a_{i_x}^{\epsilon_x} \neq  a_{i_x}^{-\epsilon_{x+1}}$ in the instance of the $a_2\inv$ and $a_2$ (which must be principal letters) in $u$ each side of the $a_1^{\xi}$.

Suppose $\mu=0$, which we can  do without loss of generality because what we are setting out to prove is symmetric with respect to inverting  $z$ and $z'$.   
Then
\begin{equation}
u \  = \   a_1^{-(e_p - i_p +2)} a_2\inv   a_1^{-(e_p - i_p +1)} a_2\inv  \cdots a_1^{-(e_p - i_p +3 - \lambda)} a_2\inv. \label{u written out}
\end{equation}
 
  After $u$ has cancelled into $\theta^{e_p}(a_{i_p})$, the word $\theta^{e_p - i_p +3}(a_3)   u  \theta^{-(e_q - i_q +3)}(a_3)$ becomes 
\begin{equation}
a_3a_2 \, a_2a_1 \, a_2a_1^2 \, \cdots a_2 a_1^{e_p - i_p +3-\lambda -1}  \  a_1^{-(e_q - i_q +2)} a_2\inv \, \cdots \, a_1^{-2} a_2\inv \, a_1\inv a_2\inv \, a_2\inv a_3\inv \label{u gone}
\end{equation}
and, as the powers of $a_1$ and $a_1\inv$ must cancel in the middle of this word, 
\begin{equation}
e_p - i_p  -\lambda   \ = \ e_q -i_q.     
\label{x-1}
\end{equation}

There are no $a_2$ among the principal letters in $u$ (expressed as \eqref{u}), and the $a_2\inv$ principal letters are those that occur in \eqref{u written out}.  The final principal letter $a^{\epsilon_{q-1}}_{i_{q-1}}$ must be $a_2\inv$ as that is the final letter in \eqref{u written out}. The remaining principal letters are $a_1$ or $a_1\inv$, and an $a_1$ principal letter is never adjacent to an $a_1\inv$ principal letter.  So  we can encode the sequence $a^{\epsilon_{p+1}}_{i_{p+1}}, \ldots, a^{\epsilon_{q-1}}_{i_{q-1}}$ using  integers $m_1, \ldots, m_{\lambda} \in \Z$, as:

$$\underbrace{a_1^{\text{sign}(m_1)}, \ldots, a_1^{\text{sign}(m_1)}}_{\abs{m_1}}, a_2\inv,  \underbrace{a_1^{\text{sign}(m_2)}, \ldots, a_1^{\text{sign}(m_2)}}_{\abs{m_2}}, a_2\inv, \ \  \ldots, \  \  \underbrace{a_1^{\text{sign}(m_{\lambda})}, \ldots, a_1^{\text{sign}(m_{\lambda})}}_{\abs{m_{\lambda}}}, a_2\inv.$$

But   \eqref{condition on e_i} and the hypothesis that $\epsilon_p =1$ allow  us to determine $e_{p+1}, \ldots, e_{q-1}$ from $e_{p}$ and $m_1, \ldots, m_{\lambda}$, so as to deduce that
\begin{eqnarray}
u & = \  a_1^{m_1} \theta^{e_p -m_1} (a_2\inv) a_1^{m_2} \theta^{e_p -m_1 -m_2 +1 } (a_2\inv) \cdots a_1^{m_{\lambda}}\theta^{e_p -m_1 - \cdots -m_{\lambda}  + \lambda -1} (a_2\inv) \label{u first} \\ 
 & =  \ a_1^{-e_p + 2m_1} a_2\inv  a_1^{-1-e_p + m_1+ 2m_2} a_2\inv \cdots a_1^{-\lambda+1 -e_p +m_1 + \cdots + m_{\lambda -1} + 2 m_{\lambda} }a_2\inv. \label{u revealed}
\end{eqnarray} 
Comparing the powers of $a_1$ here with those in \eqref{u written out}, we   get: 
\begin{eqnarray}
\quad  \left\lbrace \begin{array}{rrrrrrrrrrrr}
-2+ i_p    \ = &   & & 2m_1 \\ 
-1 + i_p    \ = & -1 & +   &   m_1  & +   &  2m_2 \\ 
i_p    \  = & -2  & +    &   m_1 & +   & m_2  & + &  2m_3 \\ 
   \vdots \ & \\
\lambda - 3 + i_p   \ = & \!\!\!  1- \lambda  & +   &   m_1  & + &     m_2 & + &  \cdots & + &  m_{\lambda-1} & + &   2m_{\lambda}, 
\end{array}  \right.    \label{family of eqns}
\end{eqnarray} 
which simplifies to 
\begin{equation}
i_p + 2^{j +1} -6 \ = \ 2^j m_j \qquad \text{for } j=1, \ldots, \lambda.  \label{i vs m}   
\end{equation}

\begin{enumerate}
\item \emph{Case  $\lambda=0$.}  This is  a case we have previously addressed:  $u$ is the empty word.  

So we can assume that $\lambda \geq 1$, and  then the $j=1$ instance of \eqref{i vs m} tells us that $i_p$ is even, and so 
\begin{equation}
i_p  \ \geq  \ 4. \label{ip ge 4}   
\end{equation}

\item \emph{Case  $\lambda=1$.}  By \eqref{x-1},
\begin{equation}
e_p-i_p- 1  \ = \  e_q-i_q. \label{1 case L}
\end{equation}
 Also   $$z  \ = \  \theta^{e_p}(a_{i_p})   \underbrace{\theta^{e_{p+1}}(a_1^{\text{sign}(m_1)}) \cdots  \theta^{e_{p+\abs{m_1}}}(a_1^{\text{sign}(m_1)})}_{\abs{m_1}} \theta^{e_p -m_1} (a_2\inv) \theta^{e_q}(a_{i_q}\inv)$$  by \eqref{u first}, and so  \eqref{condition on e_i} applied to $ \theta^{e_p -m_1} (a_2\inv)$ and $\theta^{e_q}(a_{i_q}\inv)$ tells us that $e_q = e_p -m_1 +1$.
But $i_p -2 = 2m_1$  by the $j=1$ case of  \eqref{i vs m}, and so 
\begin{equation}
e_q  \ = \  e_p - \frac{i_p -2}{2}+1 \label{2 case}. 
\end{equation}
By \eqref{1 case L} and \eqref{2 case},   
$$i_p+1  \ = \ i_q + \frac{i_p -2}{2} - 1,$$ and so 
\begin{equation}
i_p +6 \ = \  2i_q.   \label{ip iq}
\end{equation}
So   \eqref{ip ge 4} implies    $i_q \geq 5$.   And we can assume that it is not the case that $e_p-i_p +3  = e_q-i_q + 3 =0 $, else \eqref{1 case L} would be contradicted.   So $e_p-i_p +3 >0$ or $e_q-i_q +3 >0$.  If $e_p-i_p +3 >0$, there are at least  two $a_3$ in $\theta^{e_p}(a_{i_p})$  (because $i_p \geq 4$) and hence   at least two $a_3\inv$ in $\theta^{e_q}(a_{i_q}\inv)$.  Likewise,  if  $e_q-i_q +3 >0$, then   there are  
at least two $a_3\inv$ in $\theta^{e_q}(a_{i_q}\inv)$ (because $i_q \geq 4$), and so two $a_3$ in $\theta^{e_p}(a_{i_p})$.  In either case,  
using Lemma~\ref{Lem: Theta breakdown} to identify the relevant suffix of $\theta^{e_p}(a_{i_p})$ and prefix of $\theta^{e_q}(a_{i_q}\inv)$, there is a subword
\begin{equation}
\theta^{e_p - i_p +2}(a_3)\theta^{e_p - i_p +3}(a_3)   u  \theta^{e_q - i_q +3}(a_3\inv) \theta^{e_q - i_q +2}(a_3\inv),  \label{disappearing subword2}
\end{equation}
of $z$, which contains exactly two $a_3$ and two $a_3\inv$.  
If \eqref{disappearing subword2} freely reduces to the empty word, then, once the inner $a_3$ and $a_3\inv$ pair have cancelled, it reduces to   $\theta^{e_p - i_p +2}(a_3)  \theta^{e_q - i_q +2}(a_3\inv)$, which must therefore also freely reduce to the empty word.   But then   $e_p-i_p+2 = e_q - i_q+2$, also contradicting \eqref{x-1}.     So \eqref{disappearing subword2} must not freely reduce to the empty word, and its first letter (an $a_3$) and its last letter (an $a_3\inv$) are not cancelled away.  
If $i_p \ne 4$, then the required conclusions about the prefix and suffix of $z'$ follow because the $a_3$ and $a_3\inv$ bookending \eqref{disappearing subword2} do not cancel away and cannot cancel with a prefix $\theta^{e_p-1} (a_{i_p}) a_{i_p -1}$ or first letter $a_p$ or suffix $a_{i_q -1}\inv \theta^{e_q-1} (a_{i_q}\inv)$ or final letter $a_q\inv$, because  $i_p \geq 5$ and $i_q \geq 5$.   If $i_p = 4$, then $i_q =5$ by \eqref{ip iq}.  And by \eqref{1 case L}, $e_p=e_q$.  
Now, by \eqref{u first}, $u \ = \  a_1^{m_1} \theta^{e_p -m_1} (a_2\inv)$.

\item \emph{Case  $\lambda \geq 2$.}   Then \eqref{i vs m} in the case $j=2$ tells us that $i_p  =  4m_2 -2$, and in particular $i_p \neq 4$ as $m_2 \in \Z$.

At this point we know $i_p \geq 3$ (by hypothesis), is even, and is not $4$.  So $i_p \geq 6$.  

If $e_p - i_p +3 = 0$, then there is exactly one $a_3$ in  $\theta^{e_p}(a_{i_p})$, specifically its final letter.  
So the subword $a_3 u \theta^{e_q-i_q +3}(a_3\inv)$ must freely equal the empty word.  But $u = a_1^{-e_p + 2m_1} a_2\inv  a_1^{-1-e_p + m_1+ 2m_2} a_2\inv$ by \eqref{u revealed} and $\theta^{e_q-i_q +3}(a_3\inv)$ is a negative word as $e_q-i_q +3 \geq 0$, so no cancellation is possible: a contradiction. 

So, given that $e_p - i_p +3 \geq 0$, we deduce that  $e_p - i_p +2 \geq 0$, and so (as $i_p \geq 6$) there are at least two letters $a_3$ in  $\theta^{e_p}(a_{i_p})$.  But then, as above, if \eqref{disappearing subword2}  freely reduces to the empty word,    $e_p-i_p+2 = e_q-i_q+2$, but then by \eqref{balancer eqn} and that $\mu = \xi =0$, we find $\lambda=0$,  which is a case we have already addressed.  So the first and last letters ($a_3$ and $a_3\inv$, respectively) of \eqref{disappearing subword2} are not cancelled away, and therefore the first and last letters ($a_{i_p}$ and $a_{i_q}\inv$, respectively) of $z$ are also those of  $z'$, as required.  And, as $i_p \ge 6$, if $e_p >0$, then the prefix  $\theta^{e_p}(a_{i_p})$ of $z$ survives into $z'$ as it ends with a letter of rank at least $5$ which is not cancelled away.  And likewise, if $i_q \geq 5$ and $e_q >0$, then  the suffix  $\theta^{e_q}(a_{i_q}\inv)$ of $z$ survives into $z'$.

Suppose then that $i_q$ is $3$ or 4  and $e_q >0$.  

The exponent sum of the $a_2$ in $z$ between the rightmost $a_3$ of $\theta^{e_p}(a_{i_p})$ and the leftmost $a_3\inv$ of $\theta^{e_q}(a_{i_q}\inv)$ is zero, so $$e_p -i_p+3 \ = \  e_q - i_q +3 +\lambda.$$   
Applying \eqref{condition on e_i} to the suffix $\theta^{e_p -m_1 - \cdots -m_{\lambda}  + \lambda -1} (a_2\inv)$  of $u$ (expressed as per \eqref{u first}) and $\theta^{e_q}(a_{i_q}\inv)$, we get $$e_q \  = \  e_p - m_1 - \cdots - m_{\lambda} + \lambda.$$ 
Adding these two equations together and simplifying yields:
\[-i_p  \ = \  -i_q+2\lambda -m_1-\cdots-m_\lambda.\] 
The final equation of \eqref{family of eqns} is
$$\lambda - 3 + i_p   \ = \   1- \lambda  +      m_1    +       m_2   +    \cdots   +    m_{\lambda-1}   +     2m_{\lambda}.$$ 
Summing the preceding two equations and simplifying gives
\[ -4  \ = \  -i_q + m_\lambda. \]
But $i_q$ is $3$ or $4$, so $m_{\lambda}$ is $-1$ or $0$,   
But, $i_p + 2^{{\lambda} +1} -6 \ = \ 2^{\lambda} m_{\lambda} $  by  \eqref{i vs m}, which implies that $m_\lambda > 0$  because $i_p\ge 6$  and $\lambda \geq 0$---a contradiction. 
\end{enumerate}
\end{enumerate}
\end{enumerate}
\end{enumerate}
\end{proof}

\begin{proof}[Proof of Proposition~\ref{types prop} in type ii]
The result will follow from the   type $ii\inv$ instance of the proposition, proved below,  because $z$ is the inverse  of a word of  type $ii\inv$.   
\end{proof}

\begin{proof}[Proof of Proposition~\ref{types prop} in type ii${}^{-1}$]
The hypotheses dictate that  in type ii${}^{-1}$, $z$ has the form:
$$z  \ = \ \theta^{e_p}(a_{i_p}\inv) \theta^{e_p +1}(a_{i_p+1}\inv) \cdots \theta^{e_q}(a_{i_q}\inv),$$ where $e_q-e_p = i_q-i_p$.  We must show that its freely reduced form is
$$z'  \  =  \   \theta^{e_p}(a_{i_p-1}) \theta^{e_q+1}(a_{i_q}\inv).$$
Well, 
\begin{align*}
 \theta^{e_q+1}(a_{i_q}\inv) \ &   = \   \theta^{e_q}(a_{i_q-1}\inv) \theta^{e_q}(a_{i_q}\inv)  \\ 
 & = \   \theta^{e_q-1}(a_{i_q-2}\inv) \theta^{e_q-1}(a_{i_q-1}\inv) \theta^{e_q}(a_{i_q}\inv) \\ 
 & \ \ \vdots   \\
 & =  \   \theta^{e_p}(a_{i_p-1}\inv) \theta^{e_p}(a_{i_p}\inv) \theta^{e_p +1}(a_{i_p+1}\inv) \cdots \theta^{e_q}(a_{i_q}\inv),
\end{align*}
and so $z'$ and $z$ are freely equal.  

When $e_p < 0$ and $i_p-1 >1$,  Lemma~\ref{Lem: Theta breakdown} tells us that the final letter of  $\theta^{e_p}(a_{i_p-1})$ is $a_{i_p-2}\inv$.  And when $e_q+1 < 0$  and $i_q >1$, it tells us that the first letter of  $ \theta^{e_q+1}(a_{i_q}\inv)$ is  $a_{i_q -1}$.  Our hypotheses include that $e_q <0$, which implies that $e_p<0$ as $e_p <e_q$, and that  $i_q >1$, so in all cases except when $i_p=2$ or $e_q=-1$, we learn that $z'$ is freely reduced as required.  

When $i_p=2$ and $e_q \neq -1$,  
$$z'  \  =  \   a_1 \theta^{e_q+1}(a_{i_q}\inv),$$
which is freely reduced because  the first letter of $\theta^{e_q+1}(a_{i_q}\inv)$ is   $a_{i_q}-1$.  And when $e_q = -1$ and $i_p-1 \neq 1$,   $$z'  \  =  \   \theta^{e_p}(a_{i_p-1}) a_{i_q}\inv,$$
which is freely reduced because the last letter of $\theta^{e_p}(a_{i_p-1})$ is $a_{i_p -2}$.  And when  $e_q = -1$ and $i_p-1 = 1$,   $$z'  \ = \  a_1  a_{i_q}\inv,$$ which is freely reduced because $i_q \geq 3$.  

The first letter of $z$ is $a_{i_p-1}$ by Lemma~\ref{Lem: Theta breakdown} applied to $\theta^{e_p}(a_{i_p-1})$.  The final letter of $z$ is  $a_{i_q}\inv$ because the first letter of   $\theta^{e_q+1}(a_{i_q})$ is $a_{i_q}$
by the same lemma.  
\end{proof}

\begin{proof}[Proof of Proposition~\ref{types prop} in type iii]

We have that 
$$z \ = \  \theta^{e_p}(a_{i_p}) u \theta^{e_{q'}}(a_{i_{q'}}\inv) \cdots\theta^{e_q}(a_{i_q}\inv)$$  
where $i_{p}, i_{q'}, \ldots, i_{q} \geq 3$,     $i_{p+1}, \ldots, i_{q'-1} < 3$,   $e_p \geq 0$, $e_q <0$   (and so $e_{q'}, \ldots, e_{q-1} <0$ by \eqref{condition on e_i}).   Also $i_j = i_{j-1} + 1$   for $j=q'+1,\ldots,q$, so $i_q =i_{q'}+q-q'$.    Like in type \emph{i}, we must show that  the first and last letters of the freely reduced form $z'$ of $z$ are $a_{i_p}$ and $a_{i_q}\inv$, respectively, and that if  $e_p >0$, then  $\theta^{e_{p}-1}(a_{i_p})$ is a prefix of $z'$.

   Proposition~\ref{types prop} for   type  \textit{ii}$\inv$, proved above, applied to the suffix $\theta^{e_{q'}}(a_{i_{q'}}\inv) \cdots\theta^{e_q}(a_{i_q}\inv)$, tells us that $z$ freely equals
\begin{align}
 &   \theta^{e_p}(a_{i_p}) \, u  \, \theta^{e_{q'}}(a_{i_{q'}-1})\theta^{e_{q'}+q-q'+1}(a_{i_{q'}+q-q'}\inv)\label{typeiiieqn1}
\end{align}
and that  the new suffix  $\theta^{e_{q'}}(a_{i_{q'}-1})\theta^{e_{q'}+q-q'+1}(a_{i_{q'}+q-q'}\inv)$ is reduced.

By hypothesis, $i_{q'} \ge 3$.  We again organize our proof by cases.

\renewcommand{\labelenumi}{\arabic{enumi}.  }
\renewcommand{\labelenumii}{\arabic{enumi}.\arabic{enumii}. }
\renewcommand{\labelenumiii}{\arabic{enumi}.\arabic{enumii}.\arabic{enumiii}. }
\begin{enumerate}

\item \emph{Case:  $i_{q'} \ge 4$.} As the suffix  $\theta^{e_{q'}}(a_{i_{q'}-1})\theta^{e_q'+q-q'+1}(a_{i_{q'}+q-q'}\inv)$  of \eqref{typeiiieqn1}  is freely reduced, its first letter is  $a_{i_{q'}-1}$, which has rank at least $3$ by hypothesis and so  cannot cancel any letter in $u$, and is positive and so cannot cancel with a letter in   $\theta^{e_p}(a_{i_p})$.   
Therefore letters in $u$  can only cancel with the $\theta^{e_p}(a_{i_p})$ to its left.    So the final letter of $z'$ is $a_{i_{q'}+q-q'}\inv = a_{i_q}\inv$, as required. 
 As  $\Rank(u) \leq 2$ and $i_p \geq 3$,  the first letter $a_p$ of $z$ is also the first letter of $z'$, as required.  
It remains to show that, assuming $e_p >0$, the prefix    $\theta^{e_{p}-1}(a_{i_p}) $ of $z'$ is also a  prefix of $z'$.
If $i_p>3$, this is immediate because $a_{i_{p}-1}$ has rank at least $3$ and so cannot cancel into $u$.  
If $i_p = 3$, then no $a_2^{\pm{1}}$ in $u$ cancel with $\theta^{e_p}(a_{i_p})$ for otherwise  the first equation of 
 \eqref{family of eqns}  the argument from type \emph{i} would  adapt to this setting to give us the contradiction that $i_p$ is even.

\item \label{1222} \emph{Case: $i_{q'}=3$.} 

\begin{enumerate}

\item \emph{Case: $i_{q} \leq 2$.} This   does not occur because, by hypothesis, $i_{q'}\ge 3$ and $q-q'\ge 0$.  

\item \label{2222} \emph{Case: $i_{q} \geq 4$.} 
Suppose, for a contradiction, that the first or last letter of $z$ cancels away on free reduction, or that $e_p >0$ and the prefix $\theta^{e_p-1}(a_{i_p}) a_{i_p-1}$ (which is one letter longer  than we need) of  $\theta^{e_p}(a_{i_p})$ fails to also be a prefix of $z'$.
 \begin{enumerate}
\item \emph{Case: $e_{q'}+q-q' +1 =0$.} Here, as $ i_{q'}+q-q' = i_q \geq 4$, \eqref{typeiiieqn1} is
\[ \theta^{e_p}(a_{i_p}) \, u  \, \theta^{e_{q'}}(a_{2})a_{i_q}\inv.\]

Then $\theta^{e_p}(a_{i_p})$ can contain no $a_3$   since there is no $a_3\inv$ to cancel with. Therefore, $\theta^{e_p}(a_{i_p})$ ends with a letter of rank greater than $3$ by Lemma~\ref{Lem: Theta breakdown}. For this reason, $u$ cannot cancel to its left, and so  $u \theta^{e_{q'}}(a_2)$ freely equals the empty word.  By Lemma~\ref{Lem: H cancellation}, $u$ cannot contain a rank $2$ subword that freely equals the empty word,  so $u = a_1^\mu\theta^{e_{q'-1}}(a_2\inv)$ for some $\mu \in \Z$.  
But then by \eqref{condition on e_i} $e_{q'-1} = e_{q'}-1$,  and  $u = a_1^\mu\theta^{e_{q'}-1}(a_2\inv)$.
 Counting the exponent sum of the $a_1^{\pm 1}$ in $u\theta^{e_{q'}}(a_2)$, we find
\[\mu-e_{q'}+1 + e_{q'} = 0.\]
So $\mu=-1$, and $u$ must be $\theta^{e_{p+1}}(a_1\inv) \theta^{e_{q'}-1}(a_2\inv)$.   
But then applying \eqref{condition on e_i} to $\theta^{e_p}(a_{i_p}) \theta^{e_{p+1}}(a_1\inv)\theta^{e_{q'}-1}(a_2\inv)$, we find that
$e_{q'} -1 = e_p +1 \ge 1$, contradicting the fact that $e_{q'}<0$.  
 
\item  \emph{Case: $e_{q'} +q-q'+1<0$.}  
 Here,  \eqref{typeiiieqn1} is
\[ \theta^{e_p}(a_{i_p}) \, u  \, \theta^{e_{q'}}(a_{2})\theta^{e_{q'}+q-q'+1}(a_{i_q}\inv).\]

The first letter $a_{i_q-1}$  of the suffix $\theta^{e_{q'}+q-q'+1}(a_{i_{q}}\inv)$   has rank at least $3$, and must cancel to the left, but has exponent $+1$. Every other letter to the left with exponent $-1$ has rank at most $2$, so this letter cannot be canceled to its left or right. 
Thus  $z'$ must end with $a_{i_{q}}\inv$ and start with $a_{i_p}$. 

If  $i_p>3$ and  $e_p>0$, the letter immediately after the prefix $\theta^{e_p-1}(a_{i_p})$ of $z$ is $a_{i_p-1}$, which  is of  rank at least $3$, so the prefix $\theta^{e_p-1}(a_{i_p})$ must be preserved because letters of rank $3$ or higher cannot cancel as there are no letters of rank $3$ or higher between and the first letter  $a_{i_q-1}$ (of rank at least $3$)   of the suffix $\theta^{e_{q'}+q-q'+1}(a_{i_{q}}\inv)$. 

If $i_p=3$, it is conceivable that this prefix is partially canceled away by some following subword $u$ of $z$ of rank $2$ or less.  We will show this leads to a contradiction so does not occur.  
If any letters in $\theta^{e_p}(a_{i_p})u$ of rank $2$ or higher cancel, then $e_p-i_p+2\ge 0$ because otherwise $\theta^{e_p}(a_{i_p})$ ends with a letter of rank greater than $3$. 
However, then $u$ must have a prefix that cancels with $\theta^{e_p-i_p+2}(a_{2})$ and so is $\theta^{e_{p+1}}(a_1) \cdots \theta^{e_{s-1}}(a_1)  \theta^{e_s}(a_2\inv)$ or $\theta^{e_{p+1}}(a_1\inv) \cdots \theta^{e_{s-1}}(a_1\inv)  \theta^{e_s}(a_2\inv)$ for some $s$.  In either case, this simplifies to  $a_1^\mu  \theta^{e_s}(a_2\inv)$ for some $\mu \in \Z$  and, by \eqref{condition on e_i}, $e_p-\mu = e_s$. By summing the exponents of  the $a_1^{\pm 1}$ in $\theta^{e_p-i_p+2}(a_{2})$ and in $a_1^\mu  \theta^{e_s}(a_2\inv)$, we find that: $e_p-i_p+2 - e_s+\mu =0$.  
But combined with $e_p-\mu = e_s$, this tells us that 
 $\mu = (i_p-2)/2$, which is not an integer if $i_p=3$.   so we have the required contradiction.  \label{222 case} \label{3222}
\end{enumerate}

\item \emph{Case: $i_{q} = 3$.}    In this instance,   $q=q'$ because $i_{q'}=3$, and so $i_q=3$.  
 So
\[z \ = \ \theta^{e_p}(a_{i_p}) u \theta^{e_{q'}} (a_3\inv).\]
By Lemma~\ref{Lem: Theta breakdown}, there is one $a_3\inv$  in $\theta^{e_{q'}} (a_3\inv)$, specifically its final letter.   Suppose this $a_3\inv$ cancels with an $a_3$ (necessarily the rightmost) in  $\theta^{e_p}(a_{i_p})$.   Then  the intervening subword (which has rank at most $2$)  freely reduces to the empty word.  

Now $\theta^{e_p}(a_{i_p})$ contains no $a_2\inv$ because $e_p \geq 0$.  The same is true of $\theta^{e_{q'}}(a_3\inv)$   by Lemma~\ref{Lem: Theta breakdown} and the fact that $e_{q'}<0$.
So, if $u$ contains an $a_2$, it must cancel with an $a_2\inv$ from $u$, and so $u$ must contain a  subword which starts and ends with principal letters of rank $2$ and which freely equals the empty word, violating Lemma \ref{Lem: H cancellation}.  Conclude that $u$ contains no $a_2$.  

\begin{enumerate}

\item  \label{231}  \emph{Case: $e_p-i_p+2\ge 0$}.   The rightmost $a_3$ in $\theta^{e_p}(a_{i_p})$ is the first letter of the suffix $a_3 a_2 \theta^1(a_2)\cdots \theta^{e_p-i_p+2}(a_2)$, so some prefix of $u$ freely equals the inverse of $a_2 \theta^1(a_2)\cdots \theta^{e_p-i_p+2}(a_2)$.    
This prefix of $u$ must be 
\begin{equation}
\theta^{e_{p+1}}(a_{i_{p+1}}^{\epsilon_{p+1}}) \cdots \theta^{e_{s}} (a_{i_s}^{\epsilon_s}) \label{part of u}
\end{equation}
for some $s$.  (The prefix does not end in the midst of some $\theta^{e_{s}} (a_{i_s}^{\epsilon_s})$, because it must have final letter $a_2\inv$.)     

Similarly to \eqref{u first} and \eqref{u revealed} in the type \emph{i} case, we can use \eqref{condition on e_i} to re-express \eqref{part of u} as
\begin{align*}
a_1^{\nu_{\chi+1}} & \theta^{e_s+\nu_1+ \cdots + \nu_{\chi}-\chi}(a_2\inv) \cdots a_1^{\nu_2} \theta^{e_s + \nu_1-1}(a_2\inv) a_1^{\nu_1} \theta^{e_s}(a_2\inv) \\
& = \ a_1^{\nu_{\chi+1} - (e_s+\nu_1+ \cdots + \nu_{\chi}-\chi)}    a_2\inv  \cdots a_1^{\nu_2 - (e_s + \nu_1-1)}  a_2\inv  a_1^{\nu_1 - e_s} a_2\inv   
\end{align*}   
 for some $s$ where $\chi :=  e_p-i_p+2$ (so that $\chi+1$ is the number of $a_2$ in $a_2 \theta^1(a_2)\cdots \theta^{e_p-i_p+2}(a_2)$) and $\nu_1, \ldots, \nu_\chi \in \Z$ record the number of and exponents of the $a_1^{\pm 1}$ between the $a_2\inv$. As this freely equals $$(a_2 \theta^1(a_2) \theta^2(a_2) \cdots \theta^{\chi}(a_2) )\inv \ = \  a_1^{-\chi} a_2\inv \cdots  a_1^{-2}a_2\inv a_1^{-1}a_2\inv a_2\inv,$$
we find that 
\begin{align*}
\nu_1 -e_s & \ = \ 0 \\
\nu_2 - (e_s + \nu_1-1) & \ = \  -1 \\
\vdots \qquad \qquad  & \ = \  \vdots \\
\nu_{\chi+1} - (e_s+\nu_1+ \cdots + \nu_{\chi}-\chi) & \ =  \  -\chi.  
\end{align*}

It follows that \begin{equation}\nu_{\chi+1} \ = \ 2^{\chi} e_s - 2^{\chi+1} + 2.\label{e_s marker}\end{equation}

 The suffix $\theta^{e_p-i_p+2}(a_2)$ of $\theta^{e_p}(a_{i_p})$ must be the inverse  of the prefix $a_1^{\nu_{\chi+1}} \theta^{e_s+\nu_1+ \cdots + \nu_{\chi}-\chi}(a_2\inv)$ of $u$, so $\theta^{e_p-i_p+2}(a_{2})a_1^{\nu_{\chi+1}} \theta^{e_s+\nu_1+ \cdots + \nu_{\chi}-\chi}(a_2\inv)$ freely reduces to the empty word. By \eqref{condition on e_i} applied to $\theta^{e_p}(a_{i_p})a_1^{\nu_{\chi+1}} \theta^{e_s+\nu_1+ \cdots + \nu_{\chi}-\chi}(a_2\inv)$, 
\[e_p-\nu_{\chi+1} = e_s+\nu_1+ \cdots + \nu_{\chi}-\chi.\]
By counting the $a_1^{\pm 1}$ in $\theta^{e_p-i_p+2}(a_{2})a_1^{\nu_{\chi+1}} \theta^{e_p-\nu_{\chi+1}}(a_2\inv)$, which freely reduces to the empty word, we find
\[e_p-i_p+2 +\nu_{\chi+1} = e_p-\nu_{\chi+1}, \]
so that $\nu_{\chi+1} =   (i_p-2)/2$.  But then $\nu_{\chi+1}>0$, since $i_p\ge 3$. Further, we conclude that for $u$ to even cancel an $a_2$ from $\theta^{e_p}(a_{i_p})$, $i_p$ must be even.  So $i_p\ge 4$. Thus after rewriting \eqref{e_s marker} as
\begin{equation} \label{above this}
e_s  \ = \  \frac1{2^{\chi}}(\nu_{\chi+1}+2^{\chi+1}-2) 
\end{equation}
and using the fact that $\nu_{\chi+1}>0$ and $\chi \geq 1$, we conclude that $e_s>0$.

The remainder 
\begin{equation}
\theta^{e_{s'}}(a_{i_{s'}}^{\epsilon_{s'}})\cdots \theta^{e_{q'-1}}(a_{i_{q'-1}}^{\epsilon_{q'-1}}),
\label{suffix of u}
\end{equation} 
(where $s'=s+1$) of $u$  cancels with all but the $a_3\inv$ of 
\begin{equation}
\theta^{e_{q'}}(a_3\inv)  \ = \  \theta^{e_{q'}}(a_2) \theta^{e_{q'}+1}(a_2) \cdots \theta^{-1}(a_2) a_3\inv.
\label{what it cancels most of}
\end{equation} 

We claim that, similarly to \eqref{u first}, we can rewrite \eqref{suffix of u} as 
\begin{align*}
a_1^{\eta_{r}} & \theta^{e_{q'} + \eta_1 + \eta_2 + \eta_3 +\cdots + \eta_{r-1} - r}(a_2\inv) \cdots a_1^{\eta_2}\theta^{e_{q'} + \eta_1-2}(a_2\inv) a_1^{\eta_1} \theta^{e_{q'}-1}(a_2\inv)    \\ 
& = \ a_1^{\eta_{r} - (e_{q'} + \eta_1 + \eta_2 + \eta_3 +\cdots + \eta_{r-1} - r )}  a_2\inv \cdots a_1^{\eta_2 - (e_{q'} + \eta_1-2)} a_2\inv  a_1^{\eta_1 -(e_{q'}-1)} a_2\inv  
\end{align*}
 where $r$ is the number of $a_2\inv$ in  \eqref{suffix of u},  and $\eta_1,\ldots,\eta_{r} \in \Z$ record the number of and the signs of the intervening terms   $\theta^{\ast}(a_1^{\ast})$.
There is no power of $a_1$ at the righthand end   because the first letter of \eqref{what it cancels most of} is $a_2$.  The iterates of $\theta$ are identified by using  \eqref{condition on e_i}.    

Now compare with \eqref{what it cancels most of}, with which it cancels (to leave only $a_3\inv$),  to see that $r= |e_{q'}|$ and
\begin{align*}
0 & \ = \ \eta_1 - (e_{q'}-1) + e_{q'} \\
0 & \ = \ \eta_2 - (e_{q'}+\eta_1-2) + e_{q'} +1 \\
\vdots & \ = \ \qquad \qquad  \vdots \\
0 & \ = \ \eta_r - (e_{q'}  + \eta_1+\eta_2 +\cdots + \eta_{r-1}-r) + e_{q'} +(r-1).  
\end{align*}

Next we establish by induction that $\eta_i<0$  and 
\begin{equation} 
e_{q'}  + \eta_1+\eta_2 +\cdots + \eta_{i-1} - i  \ < \ 0 \label{power of theta on leftmost a_2^-1}\end{equation}
  for all $1\le i \le r$.  For the  base case, $e_q-1<0$ because of our hypothesis that $e_q<0$, and $\eta_1 = -1$  by the first of the above family of equations.  For the induction step, suppose  $\eta_1,\ldots,\eta_{i-1} <0$ and $e_{q'}  + \eta_1+\eta_2 +\cdots + \eta_{i-2} -(i-1) <0$.  The family of equations above tells us in particular, that  
$$0   \ = \ \eta_i - (e_{q'}  + \eta_1+\eta_2 +\cdots + \eta_{i-1}-i) + e_{q'} +(i-1)$$  
which rearranges to 
\[ (\eta_1+\eta_2 +\cdots + \eta_{i-1})-2i+1 \ = \ \eta_i.\]
So, $\eta_i <0$ because $1\le i$ and $\eta_{1},\ldots,\eta_{i-1}<0$. 
Moreover, 
\begin{align*}
e_{q'}  + \eta_1+\eta_2 +\cdots + \eta_{i-1} -i  & = \  (e_{q'}  + \eta_1+\eta_2 +\cdots + \eta_{i-2} -(i-1)) + \eta_{i-1} -1  
 \ < \ 0 
 \end{align*}
because  $e_{q'}  + \eta_1+\eta_2 +\cdots + \eta_{i-2} -(i-1)  <0$ and $\eta_{i-1}<0$.

Now \[e_{s'} \ = \  \eta_r + (e_{q'}  + \eta_1+\eta_2 +\cdots + \eta_{r-1} -r ) - 1\] by  \eqref{condition on e_i}.    
Conclude that $e_{s'} <0$.

But 
\[ u \  = \  \theta^{e_{p+1}}(a_{i_{p+1}}^{\epsilon_{p+1}}) \cdots \theta^{e_{s}} (a_{i_s}^{\epsilon_s})\theta^{e_{s'}}(a_{i_{s'}}^{\epsilon_{s'}})\cdots \theta^{e_{q'-1}}(a_{i_{q'-1}}^{\epsilon_{q'-1}}) \]
and  by \eqref{condition on e_i}, $e_{s}$ and $e_{s'}$ differ by at most $1$.  So, as we previously established that $e_{s} >0$, we have a contradiction.

We deduce that no $a_3$ and $a_3\inv$ cancel when $z$ freely reduces.

Since no letters of rank $3$ can cancel, if $i_p\ge 4$, then $z'$ has a prefix $\theta^{e_p-1}(a_{i_p})$, since cancelling any part of this prefix in $\theta^{e_p}(a_{i_p}) = \theta^{e_p-1}(a_{i_p})\theta^{e_p-1}(a_{i_p-1})$ requires cancellation of $a_{i_p-1}$.
 Finally consider the case $i_p=3$.  We showed (immediately above \eqref{above this}) that if $i_p$ is odd, then no letters of rank $2$ can cancel from $\theta^{e_p}(a_{i_p})$. The remainder of the argument is the same as in the case $i_p\ge 4$. 

\item \emph{Case: $e_p-i_p+2<0$.} We have $z = \theta^p(a_{i_p}) u \theta^{e_q}(a_{3}\inv)$ where $i_q=3$, $q=q'$,   $u = \theta^{e_{p+1}}(a_{i_{p+1}}^{\epsilon_{p+1}})\cdots \theta^{e_{q'-1}}(a_{i_{q'-1}}^{\epsilon_{q'-1}})$, and $\theta^{e_p}(a_{i_p})$ ends with a letter of rank at least $3$. Suppose, for a contradiction, some letter of the prefix $\theta^{e_p}(a_{i_p})$ is cancelled when  $z$ is freely reduced to $z'$.  No cancellation is possible between $\theta^{e_p}(a_{i_p})$ and $u$ because every letter of $\theta^{e_p}(a_{i_p})$ is rank $3$ or higher. By the   argument used in Case~\ref{231} to show that $e_{s'} <0$, we find here that $e_{p+1}<0$, and by the argument there (immediately after \eqref{power of theta on leftmost a_2^-1})  to show that $\eta_r <0$, we find here that $\epsilon_{p+1} = -1$.  But then by \eqref{condition on e_i}, $e_p   = e_{p+1}$,  and so $e_{p} <0$, which contradicts $e_p\ge 0$.  So the first letter $a_{i_p}$ of $z$ is also   the first letter of $z'$, and the last letter $a_3\inv$ of $\theta^{e_{i_{q'}}}(a_{3}\inv)$ is also the last letter of $z'$.  Moreover,  if $e_p>0$, then the prefix $\theta^{e_p-1}(a_{i_p})$ of  $\theta^{e_p}(a_{i_p})$ is also a prefix of $z'$.  \qedhere
\end{enumerate}
\end{enumerate}
\end{enumerate}
\end{proof}

\begin{proof}[Proof of Proposition~\ref{types prop} in type \textit{iii}${}^{-1}$]
Inverting a  type \textit{iii}${}^{-1}$ word gives a  type \textit{iii} word, so we can apply the type \textit{iii} of  Proposition~\ref{types prop}   proved above to get the result (as in this case we are only concerned with the first and last letters and not with a longer prefix). 
\end{proof}

\begin{proof}[Proof of Proposition~\ref{types prop} in type \textit{iv}]
We must show that if  $i_p, \ldots, i_{p'},i_{q'}, \ldots, i_q \ge 3$ with  $i_j = i_{j+1} + 1$   for $j=p,\ldots,p'-1$ and $i_j = i_{j-1} + 1$   for $j=q'+1,\ldots,q$, and $e_{p},e_q < 0$, the freely reduced form $z'$ of
\[z \ = \ \theta^{e_p}(a_{i_p})\cdots \theta^{e_{p'}}(a_{i_{p'}})u\theta^{e_{q'}}(a_{i_{q'}}\inv)\cdots\theta^{e_q}(a_{i_q}\inv)\]
 starts with $a_{i_p}$ and ends with $a_{i_q}\inv$.   
 
By Proposition~\ref{types prop} in type \textit{ii}${}^{\pm1}$, proved above,  $z$ freely reduces to 
\begin{equation} \theta^{e_p+1}(a_{i_p}) \theta^{e_{p'}}(a_{i_{p'}-1}\inv) u \theta^{e_{q'}}(a_{i_{q'}-1})\theta^{e_q+1}(a_{i_q}\inv) \label{type iv reduction}\end{equation}
where   $\theta^{e_p+1}(a_{i_p}) \theta^{e_{p'}}(a_{i_{p'}-1}\inv)$ and $\theta^{e_{q'}}(a_{i_{q'}-1})\theta^{e_q+1}(a_{i_q}\inv)$ are freely reduced.   
 
We again organize our proof by cases.
 
\begin{enumerate}   \def\theenumi{\arabic{enumi}.}
\renewcommand{\labelenumi}{\arabic{enumi}. }
\item \emph{Case: $i_p = i_q$}.
 Suppose, for a contradiction, that $z'$ does not start with $a_{i_p}$ and end with $a_{i_q}\inv$.   Then the first and last letter must cancel each other since they are the only maximal rank letters (because $i_p > i_{p+1} > \cdots > i_{p'}$ and $i_q > i_{q-1} > \cdots > i_{q'}$).
So $z$ freely reduces to the empty word, which we will show is impossible.   

It will be convenient (for Case~\ref{not and not}) to assume $e_p,e_q<-1$, which we can do because   applying $\theta\inv$ to $z$ gives a type $iv$ word of the same form which also freely reduces to the empty word.   

\begin{enumerate} 
\renewcommand{\labelenumii}{\arabic{enumi}.\arabic{enumii}. }
\renewcommand{\labelenumiii}{\arabic{enumi}.\arabic{enumii}.\arabic{enumiii}. }
\renewcommand{\labelenumiv}{\arabic{enumi}.\arabic{enumii}.\arabic{enumiii}.\arabic{enumiv}. }
 \def\theenumii{\arabic{enumii}.}

\item  \emph{Case:  $u$ is the empty word.}  This leads to a contradiction because it implies that the last letter $a_{{i_{p'}}-1}$ of   $\theta^{e_p+1}(a_{i_p}) \theta^{e_{p'}}(a_{i_{p'}-1}\inv)$ and the first letter $a_{{i_{q'}}-1}$ of $\theta^{e_{q'}}(a_{i_{q'}-1})\theta^{e_q+1}(a_{i_q}\inv)$ cancel---that is,   $i_{p'} = i_{q'}$, so    $ \theta^{e_{p'}}(a_{i_{p'}})\theta^{e_{q'}}(a_{i_{q'}}\inv)$  is a subword of $z$ contrary to the definition of $z$. 

\item  \emph{Case:  $u$ is not the empty word.}

\begin{enumerate}  
\item \label{not and not} \emph{Case: $p\ne p'$ and $q\ne q'$.}
In this case, $i_p,i_q\ge 4$ because of our hypotheses on $i_p, \ldots, i_{p'}, i_{q'}, \ldots, i_q$.  Since we assumed $e_p,e_q<-1$, the word in \eqref{type iv reduction} has a subword of the form 
\begin{equation}
a_{i_p-1}\inv \theta^{e_{p'}}(a_{i_{p'}-1}\inv) u \theta^{e_{q'}}(a_{i_{q'}-1}) a_{i_q-1},  \label{41 in this case}
\end{equation}
and no cancellation is possible with the prefix of $z$ to its the left or the suffix to its right. The maximal rank letters it contains are its first and last letters, so they must cancel, and therefore 
\begin{equation}\label{inner subword type iv i_p=i_q}\theta^{e_{p'}}(a_{i_{p'}-1}\inv) u \theta^{e_{q'}}(a_{i_{q'}-1})\end{equation}
must freely equal the empty word.  

\begin{enumerate}
\item \emph{Case: $i_{p'-1} \ne 2$ or $i_{q'-1}\ne 2$.}
Then $i_{p'-1}=i_{q'-1}$ because otherwise \eqref{inner subword type iv i_p=i_q}
has a single letter of highest rank which (either the $a_{i_{p'-1}}\inv$ or the $a_{i_{q'-1}}$) and hence cannot freely reduce to the empty word. 
However, then $a_{i_{p'-1}}\inv$ and $a_{i_{q'-1}}$ are the letters of highest rank in \eqref{inner subword type iv i_p=i_q} and so must cancel. Since $u$ is the subword separating them, $u$ must freely reduce to the empty word, which is impossible by Lemma~\ref{Lem: H cancellation}. 

\item \emph{Case: $i_{p'-1} = i_{q'-1} = 2$.} By Lemma \ref{Lem: H cancellation}, $u$ cannot have any rank-2 subwords that freely reduce to the empty word.
Since \eqref{inner subword type iv i_p=i_q} freely reduces to the empty word and $u$ contains no rank-2 subwords that freely reduce to the empty word, by \eqref{condition on e_i} $u$ must be
\[\theta^{e_{p'}-1}(a_2) a_1^{\mu} \theta^{e_{q'}-1}(a_2\inv)\] for some $\mu \in \Z$.
By counting the exponent sum of $a_1$ in \eqref{inner subword type iv i_p=i_q}:
\[e_{p'} - (e_{p'}-1) + \mu + (e_{q'}-1) -e_{q'}  \ = \  0,\]
so that $\mu=0$, contradicting the fact that $u$ does not have consecutive principal letters $a_2$ and $a_2\inv$ (by definition of $z$).  
\end{enumerate} 

\item  \label{p=p'} 
 \emph{Case: $p=p'$.} 
 In this case, the word \eqref{type iv reduction} which $z$ freely reduces to has the form 
\[\theta^{e_p}(a_{i_p}) u \theta^{e_{q'}}(a_{i_{q'}-1})\theta^{e_q+1}(a_{i_q}\inv).\]
Recall that the suffix $\theta^{e_{q'}}(a_{i_{q'}-1})\theta^{e_q+1}(a_{i_q}\inv)$ is freely reduced and so its first letter $a_{i_{q'}-1}$ cannot cancel to its right.  So it must cancel to its left, and therefore either $i_{q'}=3$ or it cancels with the terminal $a_{i_p-1}\inv$ of $\theta^{e_p}(a_{i_p})$. In the latter case:
\[i_q-1 \ = \  i_p-1  \ = \  i_{q'}-1,\]
so $i_{q}= i_{q'}$, and so $q=q'$.  Therefore it suffices to analyze the following two   cases.    

\begin{enumerate}  \def\theenumiii{\arabic{enumiii}}
\item \emph{Case: $i_{q'}=3$ and $q\ne q'$.}
Since $q\ne q'$, $i_q>3$.  So $i_q >3$ also as   $i_p=i_q$. 
 Hence \eqref{type iv reduction} has a subword
\begin{equation}\label{i q' 3}a_{i_p-1}\inv u \theta^{e_{q'}}(a_2)  a_{i_q-1} \end{equation}
whose first letter $a_{i_p-1}\inv$ cannot cancel to the left and whose last letter $a_{i_q-1}$ cannot cancel to the right. They have rank at least $3$, so they must cancel each other.  So $u\theta^{e_{q'}}(a_2)$ freely equals the empty word. But $u$ cannot have any rank $2$ subwords that freely equal the empty word by Lemma~\ref{Lem: H cancellation}, so by  \eqref{condition on e_i}  is
\[a_1^\mu \theta^{e_{q'}-1}(a_2\inv)\] for some $\mu \in \Z$.
 So \eqref{i q' 3} is 
\begin{equation*} a_{i_p-1}\inv a_1^\mu \theta^{e_{q'}-1}(a_2\inv) \theta^{e_{q'}}(a_2)  a_{i_q-1} \ = \ \
 a_{i_p-1}\inv a_1^\mu  \       (a_2 a_1^{e_{q'}-1} )\inv    \  a_2 a_1^{e_{q'}}   \  a_{i_q-1}.
\end{equation*} 
By counting the exponent sum of $a_1$  it contains, we find
\[ \mu -( e_{q'}-1) + e_{q'} = 0.\]  So $\mu=-1$. Now $$u \ = \  a_1^{-1} \theta^{e_{q'}-1}(a_2\inv) \ = \ \theta^e(a_1^{-1})\theta^{e_{q'}-1}(a_2\inv)$$ for some $e\in \Z$. So $\theta^{e_p}(a_{i_p}) \theta^{e}(a_1^{-1})\theta^{e_{q'}-1}(a_2\inv)$ is a prefix of $z$ and 
 \eqref{condition on e_i} tells us that $e=e_p$ and $e+1 = e_{q'}-1$, and so $e_p +2 = e_{q'}$. 
  
Now, as $u\theta^{e_{q'}}(a_2)$ freely equals the empty word and $p=p'$,  \eqref{type iv reduction} freely reduces to   
\begin{equation*} \theta^{e_p+1}(a_{i_p}) \theta^{e_{p'}}(a_{i_{p'}-1}\inv)  \theta^{e_q+1}(a_{i_q}\inv) \ = \  \theta^{e_p}(a_{i_p}) \theta^{e_q+1}(a_{i_q}\inv).
\end{equation*} 
 So, as $i_p =i_q >1$, we find $e_p = e_q+1$.  But $e_q \geq e_{q'}$, so this contradicts  $e_p +2 = e_{q'}$.

\item \emph{Case: $q=q'$.  } 
In this instance, 
\[z  \ = \  \theta^{e_p}(a_{i_p}) u \theta^{e_q}(a_{i_q}) \]
freely reduces to the identity. Hence $\theta^{\max(-e_p,-e_q)}(z)$ is a type $i$ word which also freely reduces to the identity, which is impossible by the type \emph{i} case of Proposition \ref{types prop} proved above.

\end{enumerate}

\item \emph{Case: $q=q'$.  }
Inverting $z$ returns us to Case~\ref{p=p'} above.

\end{enumerate} 
\end{enumerate} 
 
\item \label{prev case} \emph{Case: $i_p>i_q$}. By Proposition~\ref{types prop} in type \emph{ii}${}^{\pm1}$, $w$ freely reduces to a word of the form:
\[ \theta^{e_p+1}(a_{i_p}) \theta^{e_{p'}}(a_{i_p-1}\inv) u \theta^{e_{q'}}(a_{i_{q'}-1})\theta^{e_q+1}(a_{i_q}\inv).\]
Observe that $a_{i_q}$ cannot be cancelled because $a_{i_q}\inv$ does not appear. 
To cancel $a_{i_q}\inv$, since $i_q\ge 3$ and $u$ is rank $2$, $a_{i_q}\inv$ must cancel with a letter to the left of $u$, since it is the only rank $i_q$ letter appearing to the right of $u$.
Also, $a_{i_{p'}-1}\inv$, the final letter of $\theta^{e_{p'}}(a_{i_{p'}})$ is an obstruction to cancelling $a_{i_q}$ with any letter from $\theta^{e_p+1}(a_{i_p})$ and $a_{i_{p'}-1}\inv$ and has rank at least $i_q$. Thus the only letters of rank $i_p-1$ in $w$ come from $\theta^{e_p+1}(a_{i_p})$, so every letter of rank $i_p-1$ has exponent $-1$. 
To cancel $a_{i_q}\inv$ with a letter from $\theta^{e_p+1}(a_{i_p})$ requires cancelling the rightmost $a_{i_p-1}\inv$ from $\theta^{e_p+1}(a_{i_p})$ which is impossible.

Similarly, if $a_{i_q}\inv$ cancels with a letter from $\theta^{e_{p'}}(a_{i_{p'}-1}\inv)$, the rightmost letter of $\theta^{e_{p'}}(a_{i_{p'}-1}\inv)$, which is $a_{i_{p'}-1}\inv$, must cancel too. By Proposition~\ref{types prop} in type \emph{ii}${}^{\pm1}$,   $\theta^{e_p+1}(a_{i_p}) \theta^{e_p}(a_{i_{p'}-1}\inv)$ is freely reduced, so its rightmost $a_{i_{p'}-1}\inv$ must cancel to the right.
However, $a_{i_{p'}-1}\inv$ is the highest rank letter in $\theta^{e_p}(a_{i_p-1})\inv$, so $e_{p'}-1\ge i_q$. 
Also $i_{p'}-1\le i_q$ because $a_{i_{p'}-1}\inv$ can only cancel with an $a_{i_{p'}-1}$. 
We cannot cancel $a_{i_{p'}-1}\inv$ from $\theta^{e_{p'}}(a_{i_{p'}-1}\inv)$ because then $a_{i_q}\inv$ would be the only other letter of the same rank. 
Thus it is impossible to cancel $a_{i_q}\inv$. 

\item \emph{Case: $i_p<i_q$.} Invert $w$ and apply the argument from   Case~\ref{prev case} \qedhere
\end{enumerate} 
\end{proof}

\begin{proof}[Proof of Proposition~\ref{types prop} in type v]   
We have    
$$z \ =  \ \theta^{e_{p}}(a_{i_p}^{\epsilon_{p}}) \cdots \theta^{e_{q}}(a_{i_q}^{\epsilon_{q}})$$ 
and no type \emph{i}--\emph{iv} subword $\hat{z}$ of $w$ overlaps with $z$.  More precisely, there is no   $0 \leq p' <  q' \leq l+1$   with $p \leq q' \leq q$ such that $\theta^{e_{p'}}(a_{i_{p'}}^{\epsilon_{p'}}) \cdots \theta^{e_{q'}}(a_{i_{q'}}^{\epsilon_{q'}})$ is of type   \emph{i}--\emph{iv}.  The claim is that free reduction of $z$ to $z'$ removes no letters of rank $3$ or higher.   Moreover, if   $\epsilon_p=1$, $i_p\ge 3$ and $e_p >0$, then  $z'$ (the reduced form of $z$) has prefix $\theta^{e_{p}-1}(a_{i_p})$.
 
Here is our proof of the first claim.  Suppose, for a contradiction, that some letter $a_\alpha^{\epsilon}$ (not necessarily principal)  in $z$ with $\alpha \geq 3$ and $\epsilon = \pm 1$  cancels with some $a_{\alpha}^{-\epsilon}$ to its right when $z$ is freely reduced.

Then $z$ has a subword $a_{\alpha}^\epsilon v a_{\alpha}^{-\epsilon}$ which freely equals the empty word. Since $\alpha\ge 3$, we know that $a_{\alpha}$ comes from some $\theta^{e_{p'}}(a_{i_{p'}}^{\epsilon_{p'}})$ where $i_{p'}\ge 3$ while $a_\alpha\inv$ comes from some $\theta^{e_{q'}}(a_{i_{q'}}^{\epsilon_{q'}})$ where $i_{q'}\ge 3$. 
Note that $p'\ne q'$ because otherwise $a_{\alpha}^\epsilon v a_{\alpha}^{-\epsilon}$ would be a subword of  $\theta^{e_{p'}}(a_{i_{p'}})$, which is freely reduced.
We may assume that $v$ contains no letter $a_{\beta}^{\delta}$ with $\beta\ge 3$ and $\delta\in\{\pm1\}$ that cancels to its right with an $a_{\beta}^{-\delta}$ in $v$, because otherwise we could replace our original choice of $a_{\alpha}^\epsilon v a_{\alpha}^{-\epsilon}$ with a shorter subword $a_{\beta}^\delta\cdots a_{\beta}^{-\delta}$. 
So $\Rank(v) \leq 2$, and $z$ has a subword 
\begin{equation}\theta^{e_{p'}}(a_{i_{p'}}^{\epsilon_{p'}}) u \theta^{e_{q'}}(a_{i_{q'}}^{\epsilon_{q'}})\label{type v subword}\end{equation}
where $u$ is either empty or $\Rank(u) \leq 2$. 

\renewcommand{\labelenumi}{\arabic{enumi}.  }
\renewcommand{\labelenumii}{\arabic{enumi}.\arabic{enumii}. }
\renewcommand{\labelenumiii}{\arabic{enumi}.\arabic{enumii}.\arabic{enumiii}. }
\begin{enumerate}
\item \textit{Case: $\epsilon_{p'}=1$ and $\epsilon_{q'}=-1$.}
In this case, \eqref{type v subword} is  type either $i$, or  $iii^{\pm 1 }$, or $iv$ contrary to the hypothesis that $z$ is type $v$. 

\item \textit{Case: $\epsilon_{p'}=1$ and $\epsilon_{q'}=1$.}
  For $a_{\alpha}^{-\epsilon}$ is to cancel, the $a_{i_{q'}}$ at the start of  $\theta^{e_{q'}}(a_{i_{q'}}^{\epsilon_{q'}})$  must cancel to its left.   If $e_p\ge 0$, then $\theta^{e_{p'}}(a_{i_{p'}}^{\epsilon_{q'}})$ is a positive word, so the only letters to the left of $a_{i_{q'}}$ with exponent $-1$ have lower rank, and such cancellation is not possible.  If $e_p<0$, then the last letter of  $\theta^{e_{p'}}(a_{i_{p'}})$ is $a_{i_{p'}-1}\inv$, so either $i_{p'}-1=2$ or ($u$ is the empty word and $i_{q'} = i_{p'-1}$).
In the former case: $\alpha=3$, but then $a_{\alpha}^\epsilon v a_{\alpha}^{-\epsilon}$ cannot freely equal the empty word because $a_{\alpha}^{\epsilon} = a_{\alpha}$ cannot cancel with the first letter $a_{i_{q'}}$ of $\theta^{e_{q'}}(a_{i_{q'}})$.  In the latter case: by \eqref{condition on e_i}, $e_{q'} =e_{p'}-1<0$, so we have a type $ii$ subword contained in $z$, contrary to the definition of a type \emph{v} subword.

\item \textit{Case: $\epsilon_{p'}=-1$ and $\epsilon_{q'}=-1$.}
Invert and apply the previous case to obtain a contradiction.

\item \textit{Case: $\epsilon_{p'}=-1$ and $\epsilon_{q'}=1$.}
In this case \eqref{type v subword} has   subword 
\[a_{i_{p'}}\inv u a_{i_{q'}}\]
where $a_{i_{p'}}\inv$ does not cancel to the left and $a_{i_{q'}}$ does not cancel to the right, which makes a contradiction because these letters both have rank higher than $2$. 
\end{enumerate}
So the first claim is proved.

The second claim---if  $\epsilon_p=1$, $i_p\ge 3$ and $e_p >0$, then  $z'$  has prefix $\theta^{e_{p}-1}(a_{i_p})$---is proved exactly as per the final paragraph of 
Case~\ref{222 case} of our proof above Proposition~\ref{types prop} in case \emph{iii}.
\end{proof} 

\subsection{The Piece Criterion} \label{Sec: Piece Criterion}

The  \emph{Piece Criterion} is the main technical result behind the correctness of our algorithm $\alg{Member}_k$.
Before we state it, we establish two preliminary propositions.  The first is used in the proof of the second, and the second provides a key step of our proof of the Piece Criterion.  In both  we refer to  a \emph{reduced word $h$ on}  $(a_1t)^{\pm 1}$, \ldots, $(a_kt)^{\pm 1}$, which is to say that $h$ contains no subwords $(a_it)^{\pm 1}(a_it)^{\mp 1}$.

\begin{prop}\label{cor: first letter of $h$}  
 Suppose  $u=u(a_1,\ldots,a_{m-1})$ is freely reduced and non-empty, $h = h(a_1t, \ldots, a_kt)$  
is freely reduced, $r,s\in \Z$,  and $2 \leq  m \leq k$. In $G_k$,  
\begin{center}
\begin{tabular}{rl}
  ($t^{r} a_m u   = ht^s$ \text{ or } $t^{r} a_m u a_m^{-1} = ht^s$)   \!\!\!\!\! & \quad $\implies$ \quad the first letter of $h$ is $(a_mt)$, \\
($t^{r}   u a_m^{-1} = ht^s$ \text{ or }  $t^{r} a_m u a_m^{-1} = ht^s$)   \!\!\!\!\! & \quad $\implies$ \quad  the final letter of $h$ is $(a_mt)^{-1}$. \\
\end{tabular}
\end{center}
\end{prop}

\begin{proof}
The second statement follows from the first as can be seen by inverting both sides  of the equalities and then rearranging so as to interchange the roles of $r$ and $s$.  

We will  prove the first statement in the case $t^{r} a_m u   = ht^s$ only,  as the case $t^{r} a_m u  a_m^{-1} = ht^s$ can be proved in essentially the same way.   

So assume  $t^{r} a_m u   = ht^s$, and so $a_m u   = t^{-r} h t^s$, in $G_k$.  Consider carrying all the $t^{\pm 1}$  in $t^{-r} h t^s$ from left to right through the word, with the effect of applying $\theta^{\mp 1}$ to the intervening letters $a_i^{\pm 1}$, and then freely reducing, so as to arrive at  $a_m u$.  

We will first argue that  $h$ contains no $(a_{m+1}  t)^{\pm 1}, \ldots, (a_{k}  t)^{\pm 1}$.  Suppose otherwise.  Let $i$ be maximal such that $h$ contains an  $(a_i t)^{\pm 1}$.  As carrying all the $t^{\pm 1}$ to the right and cancelling gives  $a_m u$, there must  be  an  $(a_i t)^{\mp 1}$ in    $h$ so that there is an $a_i ^{\mp 1}$ to cancel with the $a_i^{\pm 1}$ in our  $(a_i t)^{\pm 1}$---this is because applying $\theta^{\pm1}$ to $a_1^{\pm 1}, \ldots, a_i^{\pm 1}$, neither creates nor destroys any $a_i^{\pm 1}$.  But then if $h'$ is the subword of $h$ that has first and last (or last and first) letters these $(a_i t)^{\pm 1}$ and $(a_i t)^{\mp 1}$, then  $t^{r'}  h' = t^{s'}$ for some $r' , s' \in \Z$.  That then implies that $h' \in \langle t \rangle$.  But    $H_k \cap \langle t \rangle = \set{1}$ by  Lemma~6.1 of \cite{DR}, so  $h =1$ in $G_k$.   But   
$H_k = F(a_1t, \ldots, a_kt)$ by Proposition~4.1   of \cite{DR}, and so   our assumption that $h$ is freely reduced is contradicted.    

Next notice that there must be an $(a_m t)$ in $h$   because $a_mu$ contains an $a_m$ and applying $\theta^{\pm1}$ to $a_1^{\pm 1}, \ldots, a_m^{\pm 1}$ neither creates nor destroys any $a_m^{\pm 1}$.  
Suppose, for a contradiction, that the first $(a_m t)$ in $h$ is not at the front. 
Express $h$ as  $\alpha (a_mt) \beta$ where $\alpha = \alpha(a_1t, \ldots, a_{m-1}t)$ is non-empty.   

We claim that the  $a_m$ of the first $(a_m t)$ in $h$ must cancel with some subsequent $a_m\inv$.  Suppose otherwise.  We have that  $$t^{-r} h t^s  \ = \  t^{-r} \alpha (a_mt) \beta t^s \ = \  v t^{j} (a_mt) \beta t^s $$  for some $v=v(a_1, \ldots, a_{m-1})$ and some $j \in \Z$.  But then $v =1$ as the first $a_m$ serves as a barrier to cancelling away $v$ when the remaining $t^{\pm 1}$ are carried to the right:  applying $\theta^{\pm 1}$ to $a_m$ only produces  new letters $a_1^{\pm 1}, \ldots, a_{m-1}^{\pm 1}$ (see Lemma~7.1 in \cite{DR}) to its right, and (by assumption) it is not cancelled away by a subsequent $a_m\inv$.    But then $\alpha \in \langle t \rangle$, leading to a contradiction as before.  

Now, if   $a_m$ of the first $(a_m t)$ in $h$ cancel with some subsequent $a_m\inv$, by the same argument as earlier, the subword bookended by that $(a_m t)$ and $(a_m t)\inv$ must freely reduce to the empty word, contradicting the assumption that $h$ is freely reduced. 
\end{proof}

To follow the details of the following proof it will help to have a copy of Definition~\ref{five types} and Proposition~\ref{types prop} to hand.  

\begin{prop}\label{Prop: set prefix}
Suppose $u = u(a_1,\ldots,a_{m-1})$ is freely reduced,   $h = h(a_1t, \ldots, a_kt)$ is  freely reduced, $r,s\in\Z$,  $3 \leq m  \leq k$, and $t^{r} a_m u   = ht^s$ or   $t^{r} a_m u a_m^{-1} = ht^s$ in $G_k$.  If  $r>0$, then $\theta^{r-1}(a_m)$ is a prefix of $a_mu$. 
\end{prop}

\begin{proof}
We will prove the case where  $t^{r} a_m u a_m^{-1} = ht^s$ in $G_k$.  The proof for the case 
$t^{r} a_m u   = ht^s$   is  the same.

Proposition~\ref{cor: first letter of $h$}  tells us that  the first and last letters of $h$ are $(a_mt)$ and $(a_mt)\inv$, respectively.  Express $h$ as $(a_{i_0} t)^{\epsilon_0} \cdots (a_{i_{j+1}} t)^{\epsilon_{j+1}}$ where $\epsilon_0=1$, and $\epsilon_1, \ldots,  \epsilon_j = \pm 1$, and $\epsilon_{j+1} = - 1$, and $i_0=i_{j+1}=m$, and $i_1, \ldots, i_j \in \set{1, \ldots, m-1}$.  

If we shuffle all the $t^{\pm 1}$ in $t^{-r} ht^s$ to the right, then the power of $t$ emerging on the right cancels away since $t^{-r} ht^s$ equals  $a_m ua_m\inv$ and $u = u(a_1,\ldots,a_{m-1})$ in $G_k$, and we get  
\[ \pi \ := \  a_mua_m\inv \ = \  \theta^{e_0}(a_{i_0}^{\epsilon_0})  \cdots \theta^{e_j}(a_{i_j}^{\epsilon_j}) \theta^{e_{j+1}}(a_{i_{j+1}}^{\epsilon_{j+1}})\]
where    $e_l$ is, for $0 \leq l \leq j+1$, the exponent sum of the $t^{\pm 1}$ in $h$ that precede  $a_{i_l}$ in $t^{-r} ht^sa_m$  (which  includes the $t^{-1}$ of   $(a_{i_l} t)^{\epsilon_l}$ if  $\epsilon_l = -1$):       
$$e_l \  = \  
\begin{cases}
r  + \epsilon_1 + \cdots + \epsilon_{l-1} & \text{ if }  \ \epsilon_l =1 \\ 
r  + 1  + \epsilon_1 + \cdots + \epsilon_{l-1}  & \text{ if } \ \epsilon_l = - 1.  
\end{cases}
  $$
 Also  $a^{\epsilon_x}_{i_x} \neq a^{-\epsilon_{x+1}}_{i_{x+1}}$ for $x= 0, \ldots, j$ because $h$ is freely reduced as a word on $(a_1t)^{\pm 1}, \ldots, (a_k t)^{\pm 1}$.    So, $\pi$ is of the form in which it appears in Definition~\ref{five types}.

We will work right to left through $z$ choosing subwords  $z_1$, $z_2$, .... until we have $\pi$ expressed as a concatenation $z_l \cdots z_2 z_1$.    Define $\pi_1 := \pi$ and   define $z_1$ to be the maximal length suffix of $\pi_1$ of one of the five types of Definition~\ref{five types}.  (Such a suffix exists if $\pi_1$ is non-empty, as there must be a type  \textit{v}  suffix if no other type.)   Let $\pi_2$ be $\pi_1$ with the suffix  $z_1$ removed, and then define   $z_2$ to be the maximal length suffix of $\pi_2$ of one of the  five types of Definition~\ref{five types}.  Continue likewise until $z$ is exhausted and we have $\pi = z_l \cdots z_2 z_1$.  

Let $\pi', z'_1, \ldots, z'_l$ denote the freely reduced forms of  $\pi, z_1, \ldots, z_l$, respectively. We will use  Proposition~\ref{types prop} to argue that 
$\pi' =  z'_l \cdots z'_2 z'_1$.  In other words, when freely reducing $\pi$, all cancellation is within the $z_i$---none occurs between a $z_{i+1}$ and the neighboring $z_i$.   

Given how Proposition~\ref{types prop} identifies the first and last letters of each $z'_i$ when of type \textit{i}--\textit{iv}, and given that  
$a^{\epsilon_x}_{i_x} \neq a^{-\epsilon_{x+1}}_{i_{x+1}}$ for $x= 0, \ldots, j$,  cancellation between  $z'_{i+1}$ and $z'_i$ is ruled out except in  these four situations:
\begin{itemize}
\item  $z_i$ is of type  \textit{ii}${}\inv$,
\item $z_{i+1}$ is of type \textit{ii},
\item $z_i$ is of type \textit{v},
\item $z_{i+1}$ is of type \textit{v}.  
\end{itemize}
We will explain why these too do not give rise to cancellation.  Express $z_{i+1}$ and $z_i$ as:
$$z_{i+1}  \ = \  \theta^{e_p}(a^{\epsilon_p}_{i_p})   \cdots \theta^{e_q}(a^{\epsilon_q}_{i_q}) \qquad \text{ and } \qquad z_i  \ = \   \theta^{e_{p'}}(a^{\epsilon_{p'}}_{i_{p'}})   \cdots \theta^{e_{q'}}(a^{\epsilon_{q'}}_{i_{q'}}).$$  (So $p' = q+1$.) 

\emph{Case:  $z_{i+1}$ not type v,  $z_i$ type ii${}\inv$.}   The first letter of   $z'_i$ is $a_{i_{p'} -1}$ by Proposition~\ref{types prop} in type \textit{ii}${}\inv$.  
If $z_{i+1}$ is of type \textit{ii}, then the final letter of $z'_{i+1}$  is $a_{i_{{p'}-1} -1}\inv$ (remember $p'-1 =q$) which cannot cancel with the $a_{i_{p'} -1}$ at the start of $z'_i$ since $a_{i_{{p'}-1}}^{\epsilon_{{p'}-1}}$ and $a_{i_{p'}}^{\epsilon_{p'}}$ are not mutual inverses and $\epsilon_{p'-1} = 1$ and $\epsilon_{p'} = -1$.   
If $z_{i+1}$ is of type \textit{i}, \textit{ii}${}\inv$, \textit{iii}${}^{\pm 1}$, or \textit{iv}, then the final letter of $z_{i+1}$ is $a_{i_{p'-1}}\inv$ which cannot be $a_{i_{p'}-1}\inv$ as that would contradict the maximality of $z_i$: prepending $\theta^{e_p'}(a_{i_{p'-1}}^{\epsilon_{p'-1}})$  to $z_i$ would give a longer type \textit{ii}${}\inv$ word.    

\emph{Case:   $z_{i+1}$  type ii, $z_i$  not type v.}    Similarly, there can be no cancellation between $z'_{i+1}$ and $z'_i$.  In the cases where $z_i$ is of type  \textit{i}, \textit{ii}, \textit{iii}${}^{\pm 1}$, or iv appending $\theta^{e_{p+1}}(a_{i_{p+1}}^{\epsilon_{p+1}})$ to $z_{i+1}$ would give a longer type \textit{ii} word, contradicting the definition of $z_i$ as a type~\textit{v} word.  

\emph{Case:  $z_{i+1}$ not type \textit{ii}, $z_i$ type \textit{v}.}  Then $z_{i+1}$ cannot be of type \textit{v}, else $z_{i+1} z_i$ would be of type \textit{v} contrary to maximality of $z_i$.  So $z_{i+1}$ is of type \textit{i}, \textit{ii}${}\inv$, \textit{iii}${}^{\pm 1}$ or \textit{iv}, and therefore $i_{q} \geq 3$ and $\epsilon_q =-1$, and by Proposition~\ref{types prop}, the final letter of $z'_{i+1}$ is $a\inv_{i_{q}}$.      So if there is cancellation between $z'_{i+1}$ and $z'_i$, then the first letter of $z'_i$ must be $a_{i_{q}}$.  But then,   there is a subword 
$$\pi'' \ := \ \theta^{e_q}(a^{-1}_{i_q})\theta^{e_{p'}}(a^{\epsilon_{p'}}_{i_{p'}}) \cdots \theta^{e_{m}}(a^{\epsilon_{m}}_{i_{m}})$$ of $z_{i+1}z_i$ such  that  the    $a^{-1}_{i_q}$ in $\theta^{e_q}(a^{-1}_{i_q})$  cancels with some $a_{i_q}$  in  $\theta^{e_{m}}(a^{\epsilon_{m}}_{i_{m}})$ on free reduction and $i_{p'}, \cdots, i_{m-1} \leq 2$---otherwise there would be some intervening letter of rank at least $3$ which would have to cancel away on freely reducing this subword and hence on freely reducing $z_i$,  contrary to  Proposition~\ref{types prop} in type \textit{v}.  

Suppose $\epsilon_m =1$.  Then  $\theta^{e_{m}}(a^{\epsilon_{m}}_{i_{m}})$ is $a_{i_m}$ times a word on lower rank letters.   So, as the    $a^{-1}_{i_q}$ in $\theta^{e_q}(a^{-1}_{i_q})$  cancels away when $\pi''$ is freely reduced,     $a^{\epsilon_{m}}_{i_{m}} = a_{i_q}$.  
But then the intervening subword $\theta^{e_{p'}}(a^{\epsilon_{p'}}_{i_{p'}}) \cdots \theta^{e_{m-1}}(a^{\epsilon_{m-1}}_{i_{m-1}})$ has rank at most $2$ and freely reduces to the empty word, and so is empty  by Lemma~\ref{Lem: H cancellation}.  
So  $p'=m$ and, as $a^{\epsilon_{m}}_{i_{m}} = a_{i_q}$, that contradicts  the $x = q$ instance of $a^{\epsilon_x}_{i_x} \neq a^{-\epsilon_{x+1}}_{i_{x+1}}$.

Suppose, on the other hand, $\epsilon_m =-1$.  If $e_m \geq 0$, then 
$\theta^{e_m}(a^{-1}_{i_{m}})$ contains no positive letters and so cannot supply a letter to cancel with 
$a\inv_{i_{q}}$.  If  $e_m < 0$ and $i_m =3$, then the only letter in $\theta^{e_m}(a^{-1}_{i_{m}})$  of rank at least three is a single $a^{-1}_{3}$, and that cannot cancel with $a\inv_{i_{q}}$.  
If  $e_m < 0$ and $i_m >3$, then the first letter of $\theta^{e_m}(a^{\epsilon_{m}}_{i_{m}})$  is  $a_{i_m -1}$  (Lemma~\ref{Lem: Theta breakdown}) and  this could only cancel with the  $a\inv_{i_{q}}$ were the intervening subword $\theta^{e_{p'}}(a^{\epsilon_{p'}}_{i_{p'}}) \cdots \theta^{e_{m-1}}(a^{\epsilon_{m-1}}_{i_{m-1}})$ empty (as before) and $p' =  m=q+1$, but in that case $z_i$ has prefix $\theta^{e_{p'}}(a^{\epsilon_{p'}}_{i_{p'}}) = \theta^{e_{q+1}}(a^{-1}_{i_{q}+1})$, violating the definition of a type \textit{v} subword because $\theta^{e_q}(a_{i_q}\inv)\theta^{e_{q+1}}(a^{-1}_{i_{q}+1})$ is type ii$\inv$.
 
\emph{Case:   $z_{i+1}$ type \textit{v}, $z_i$ not type ii${}\inv$.} As in the previous case, $z_i$ cannot be of type \textit{v}, so $z_i$ is type~\textit{i}, \textit{ii}, \textit{iii}${}^{\pm 1}$ or \textit{iv} and $i_{q+1} \geq 3$. 
The same arguments as the previous case apply to tell us that cancellation is impossible.  The final case concludes with the maximality of the type \textit{i}, \textit{ii}${}\inv$, \textit{iii}${}^{\pm 1}$ or \textit{iv} word $z_i$ being contradicted.

\emph{Case:   $z_{i+1}$  type \textit{v}, $z_i$ type \textit{ii}${}\inv$.} We have that $$z_{i+1}  \ = \   \theta^{e_p}(a^{\epsilon_p}_{i_p})   \cdots \theta^{e_q}(a^{\epsilon_q}_{i_q}) \qquad \text{ and } \qquad z'_i \ = \     \theta^{e_{p'}}(a_{i_{p'} -1}) \theta^{e_{q'}+1}(a_{i_{q'}}\inv)$$   by definition and by Proposition~\ref{types prop} in type \textit{i}, respectively, and $e_{q'} <0$, $i_{q'} \geq 3$, and $i_{p'} \geq 2$.   Moreover, the first letter of $z'_i$ is  $a_{i_{p'} -1}$  by Proposition~\ref{types prop} in type~\textit{ii}. Suppose  $i_{p'}$ is $2$ or $3$.  Then  $z_{i+1}$ has suffix    $\theta^{e_q}(a^{\epsilon_q}_{i_q}) = \theta^{e_q} (a_{i_{p'}-1}\inv)$ or something of rank at most $2$ which could be prepended to $z_i$   contradicting its maximality.       Suppose, on the other hand,  $i_{p'} > 3$.  If there is cancellation between $z'_{i+1}$ and $z'_i$, then  a letter of rank at least $3$ in $z_{i+1}$ cancels with the first letter   $a_{i_{p'} -1}$ of $z'_i$. 
As in the preceding cases, conclude that $a_{i_q}^{\epsilon_q}$ must cancel with the first letter of $z_i'$, so $i_q=i_{p-1}$ and $\epsilon_q=-1$, contradicting maximality of $z_i$.

\emph{Case:   $z_{i+1}$  type \textit{ii}, $z_i$  type~v.}  
This case is essentially the same as the preceding one. Follow the steps from the previous case, except instead of appealing to maximality of $z_{i+1}\inv$, observe that the last letter of $z_{i+1}$ and $z_i$ form a type~\textit{ii} subword which is forbidden by the definition of a type~\textit{v} subword.

Having established that there is no cancellation between $z'_{i+1}$ and $z'_i$ for  $i=1, \ldots, l-1$, all that remains is to argue that $a_m z'_l$ has prefix $\theta^{r-1}(a_m)$, for it will then follow that $a_m \pi'$ has the same prefix.  

 But $z_l$ is type~\emph{i}, \emph{iii} or \emph{v}  because $e_0 = r>0$.    It has prefix $\theta^{e_0}(a_{i_0}^{\epsilon_0})= \theta^{r}(a_m)$ and $r > 0$, so as $i_0 =m \geq 3$,   Proposition~\ref{types prop} in types \emph{i}, \emph{iii} and \emph{v},  tells us that $\theta^{r-1}(a_m)$ is a prefix of $z'_l$, and hence of $\pi = a_mu$.
  \end{proof}

We are now ready for the Piece Criterion.  
It concerns only the case where the rank (denoted by $m$) is at least $3$.  In the cases $m=1$ and  $m=2$  our algorithms are straightforward and  the Piece  Criterion is not required to prove correctness.

\begin{prop}[The Piece Criterion]\label{Prop: Piece Criterion}
Suppose $m\ge 3$ and $r\in \Z$, and suppose $\pi = a_m^{\epsilon_1} u a_m^{-\epsilon_2}$ is a freely reduced word such that $u=u(a_1,\ldots,a_{m-1})$ and $\epsilon_1, \epsilon_2\in\{0,1\}$.
Define 
\begin{align*}
 x_l  \ &  
 :=   \  a_m\inv \theta^{l}(a_m)  \quad \text{ for } l\in \Z,  \\
  x \ & := \ \begin{cases} x_r & \text{if }r>0 \text{ and } \epsilon_1 =1 \\ \textit{empty word} & \text{otherwise}, \end{cases} \\
 \delta \ & := \  \begin{cases} r & \text{if} \  \ \epsilon_1 = 0 \\ 
\psi_m(r) &  \text{if} \  \   \epsilon_1=1 \text{ and } \,r\le 0 \\
r-1 &  \text{if} \  \  \epsilon_1=1 \text{ and } r>0. \end{cases} 
\end{align*}
Suppose $s \in \Z$. Let   $\pi'$  be the freely reduced form of $x^{-\epsilon_1} ua_m^{-\epsilon_2}$. Consider the following conditions.    
\begin{enumerate}
\renewcommand{\labelenumi}{\textup{(}\roman{enumi}\textup{)} }
\def\theenumi{\roman{enumi}}
\item $\epsilon_1 = 0$. \label{Condition: Piece 1}
\item $\epsilon_1 =1$ and $r\le 0$. \label{Condition: Piece 2}
\item $\epsilon_1= 1$, $r>0$ and $\theta^{r-1}(a_m)$ is a prefix of $\pi$. \label{Condition: Piece 3} 
\end{enumerate}
\begin{enumerate}
\renewcommand{\labelenumi}{\textup{(}\alph{enumi}\textup{)} }
\def\theenumi{\alph{enumi}}
\item $\epsilon_2=0$ and $t^\delta x^{-\epsilon_1} u \in H_kt^{s}$. \label{Condition: Piece a}
\item $\epsilon_2=1$, $s  \le  0$ and $t^\delta x^{-\epsilon_1} u\in H_kt^{\psi_m(s)}$. \label{Condition: Piece b}
\item $\epsilon_2=1$, $s >0$ and  $t^\delta x^{-\epsilon_1} u x_{s} \in H_kt^{s-1}$ and $\theta^{s-1}(a_m\inv)$ is a suffix of $\pi$. \label{Condition: Piece c} 
\end{enumerate}
\def\theenumi{\roman{enumi}}
\renewcommand{\labelenumi}{\roman{enumi}. }

We have  $t^r \pi\in H_kt^s$ if and only if ((\ref{Condition: Piece 1},  \ref{Condition: Piece 2}  or  \ref{Condition: Piece 3})  and $t^{\delta} \pi' \in H_k t^s$).   
Moreover,  $t^{\delta} \pi' \in H_k t^s$ if and only if    (\ref{Condition: Piece a}, \ref{Condition: Piece b} or \ref{Condition: Piece c}).   
\end{prop}
 \begin{proof}  Suppose $s \in \Z$.
First suppose that $t^r\pi\in Ht^s$.  Then  (\emph{\ref{Condition: Piece 1}}, \emph{\ref{Condition: Piece 2}} or \emph{\ref{Condition: Piece 3}}) holds because if   $\epsilon_1= 1$ and $r>0$, then $\theta^{r-1}(a_m)$ is a prefix of $\pi$ by Proposition~\ref{Prop: set prefix}.  So  
$t^\delta x^{-\epsilon_1}u a_m^{-\epsilon_2} \in H_k t^s$ for the same $s\in \Z$.

Next we will prove that     $t^r\pi\in Ht^s$  is equivalent to  $t^\delta \pi' \in H_k t^s$ under the assumption that (\emph{\ref{Condition: Piece 1}}, \emph{\ref{Condition: Piece 2}} or \emph{\ref{Condition: Piece 3}}) holds.

 Under  \emph{\ref{Condition: Piece 1}},   $\epsilon_1=0$,  $x$ is the empty word, and $\delta = r$.  So $\blue{t^\delta \pi' =} t^\delta x^{-\epsilon_1}u a_m^{-\epsilon_2} =  t^r  u a_m^{-\epsilon_2} = t^r \pi$ and the equivalence is immediate.  

Under  \emph{\ref{Condition: Piece 2}}, $\epsilon_1=1$, $r \leq 0$,  $x$ is the empty word, and $\delta =  \psi_m(r)$.  
So $t^\delta \pi' =  t^\delta x^{-\epsilon_1}u a_m^{-\epsilon_2} =  t^{\psi_m(r)}  u a_m^{-\epsilon_2}$, giving the third of the following equivalences.   The first equivalence holds simply because  $\pi = a_m u a_m^{-\epsilon_2}$.  For the second,  $r$  is in the domain of $\psi_m$ because $r\leq 0$,   so $t^ra_m \in H_kt^{\psi_m(r)}$ by Proposition \ref{Prop: Hydra Game}, and so 
$t^{\psi_m(r)} a_m\inv t^{-r} \in H_k$.     
\begin{align*}
 t^r \pi & \in H_k t^s  \\
&  \Leftrightarrow \   t^r a_m u a_m^{-\epsilon_2}   \in H_k t^s \\ 
&   \Leftrightarrow \  t^{\psi_m(r)}  u a_m^{-\epsilon_2} \in H_kt^s \\
&   \Leftrightarrow \  t^{\delta} \pi' \in H_kt^s. 
\end{align*}

Under \emph{\ref{Condition: Piece 3}},   $\epsilon_1=1$, $r > 0$,  $x = x_r$, and $\delta = r-1$.   Observe that 
$$t^\delta \pi' =  t^{r-1} x\inv_r u a_m^{-\epsilon_2}  \in H_kt^s   \  \Leftrightarrow \    t^r \pi =  t^r a_m u  a_m^{-\epsilon_2}  \in H_kt^s$$ 
because $t^{r-1} x_r\inv a_m\inv t^{-r} = t^{r-1} \theta^r(a_m\inv) t^{-r} = (a_m t)\inv \in H_k$.  

So, assuming (\emph{\ref{Condition: Piece 1}}, \emph{\ref{Condition: Piece 2}} or \emph{\ref{Condition: Piece 3}}) holds,  $t^r \pi \in H_kt^s$ if and only if $t^\delta \pi'  \in H_k t^s$, as required.

Next we will prove that $t^{\delta} \pi' \in H_k t^s$ if and only if    (\emph{\ref{Condition: Piece a}}, \emph{\ref{Condition: Piece b}} or \emph{\ref{Condition: Piece c}}) holds.

Suppose $\epsilon_2=0$.  Then  $t^\delta \pi' = t^\delta x^{-\epsilon_1}u a_m^{-\epsilon_2} = t^\delta x^{-\epsilon_1}u$ and   so  $t^{\delta} \pi' \in H_k t^s$ is the same as Condition~\emph{\ref{Condition: Piece a}}.

Suppose, on the other hand, that $\epsilon_2=1$.   Suppose further that $s\le 0$.      Proposition~\ref{Prop: Hydra Game} tells us that $t^s a_m \in H_k t^{\psi_m(s)}$ since $s \leq 0$ and so is in the domain of $\psi_m$.  
So $t^{\delta}\pi' = t^{\delta} x^{-\epsilon_1} u a_m\inv \in  H_k  t^s$ if and only if $ t^{\delta}  x^{-\epsilon_1} u   \in H_kt^{\psi_m(s)}$. So $t^{\delta} \pi' \in H_k t^s$  is equivalent to Condition~\emph{\ref{Condition: Piece b}}.  

Finally, observe that
\begin{align*} 
& t^\delta \pi' = t^\delta x^{-\epsilon_1} u a_m\inv \in H_kt^s \\ 
 & \  \Leftrightarrow \ t^\delta x^{-\epsilon_1} u a_m\inv t^{-s} \in H_k \\      
	  & \  \Leftrightarrow \ t^\delta  x^{-\epsilon_1}  u  a_m\inv t^{-s}   (   t^{s} a_m   x_s t^{-(s-1)})   \in H_k \\ 
	  & \  \Leftrightarrow \ t^\delta  x^{-\epsilon_1}  u     x_s \in H_kt^{s-1} 
 \end{align*}
 because $   t^{s} a_m   x_s t^{-(s-1)}  = a_mt  \in \H_k$.  Suppose now that  $s>0$.  The   part of Condition~\emph{\ref{Condition: Piece c}} concerning the suffix of $\pi$ follows from Proposition~\ref{Prop: set prefix} (applied to $h^{-1}$).     So   $t^{\delta} \pi' \in H_k t^s$  is equivalent to Condition~\emph{\ref{Condition: Piece c}}.

We conclude that $t^r\pi \in Ht^s$ implies  (\emph{\ref{Condition: Piece 1}}, \emph{\ref{Condition: Piece 2}},  or \emph{\ref{Condition: Piece 3}}) and  (\emph{\ref{Condition: Piece a}}, \emph{\ref{Condition: Piece b}}, or  \emph{\ref{Condition: Piece c}}).
\end{proof}

\subsection{Our algorithm in detail}\label{Member construction}

Here we construct  $\alg{Member}_k$, where $k$ is, as usual, any integer greater than or equal to $1$, and is kept fixed.  $\alg{Member}_k$ inputs a word $w = w(a_1, \ldots, a_k, t)$ and    declares whether or not  $w$ represents an element of $H_k$.

Most of the workings of  $\alg{Member}_k$ are contained in a subroutine $\alg{Push}_k$, which   
inputs a valid $\psi$-word $f$ and a reduced word $v=v(a_1, \ldots, a_k)$, and declares whether or not $t^{f(0)}v \in H_k t^s$ for some  $s \in \Z$ and, if so, returns  a $\psi$-word $f'$ with $s= f'(0)$.  (If such an $s$ exists, it is unique by Lemma~6.1 in \cite{DR}.)  The key subroutine for $\alg{Push}_k$  when $k \geq 2$  is $\alg{Piece}_k$ which handles  the special case in which $w$ is a rank-$m$ piece.   $\alg{Piece}_k$  calls a subroutine $\alg{Back}_{k}$, which in turn calls a subroutine $\alg{Push}_{k-1}$.  So the construction of these three families of subroutines is inductive.     

Additionally, subroutines  $\alg{Prefix}_m$, and $\alg{Front}_m$  (where $3 \leq m \leq k$) are used.   These do not require an inductive construction, so we will give them first.  The designs of $\alg{Prefix}_m$,  $\alg{Front}_m$ (and also $\alg{Back}_m$) are motivated by  the Piece Criterion  (Proposition~\ref{Prop: Piece Criterion}).

\begin{algorithm*}[h]
\caption{ --- \alg{Prefix}$_m$, $m \geq 3$. \newline
$\circ$ \ Input  a rank-$m$ piece $\pi = a_m u a_m^{-\epsilon_2}$ (so,  $u=u(a_1, \ldots, a_{m-1})$ is reduced and $\epsilon_2 \in \set{0,1}$).  \newline 
$\circ$ \ Return the largest   integer  $i > 0$ (if any) such that $\theta^{i-1}(a_m)$ is a prefix of $\pi$.  \newline
$\circ$ \  Halt  in time in $O(\ell(\pi)^{2})$. 
}
\label{Alg: CheckPrefix}
\begin{algorithmic}[3]
\State \textbf{construct} $\theta^{i-1}(a_m)$ for $i=1, 2, \ldots$ until $\ell(\theta^{i-1}(a_m)) > \ell(\pi)$, and \textbf{compare} to $\pi$   
\State \textbf{return} the maximum $i$ encountered (if any) such that   $\theta^{i-1}(a_m)$ is a prefix of $\pi$ 
\end{algorithmic}
\end{algorithm*}

\begin{proof}[Correctness of $\alg{Prefix}_m$.]  
As  $\ell(\theta^{i-1}(a_m)) \geq i$  for $i  =1, 2, \ldots$,  the algorithm returns the appropriate $i$ in time  $O(\ell(\pi)^{2})$.
\end{proof}

$\alg{Front}_m$ takes a rank-$m$ piece $\pi$ and $\psi$-word $f$ and reduces the task of determining  whether $t^{f(0)} \pi \in Ht^s$   to performing a similar determination: specifically whether  $t^{f'(0)} \pi' \in Ht^s$  where $f'(0) = \delta$ and   $\pi'$ and $\delta$  are as per the Piece Criterion.   This will represent progress because $\pi'$ is a piece of rank-$m$  that does not begin with $a_m$, and because we are able to give good bounds on $\ell(\pi')$ and $\ell(f')$. 
   \begin{algorithm}[h]
    \caption{ --- \alg{Front}$_m$, $m \geq 3$.
     \newline
$\circ$ \  Input a rank-$m$ piece $\pi = a_m^{\epsilon_1} u a_m^{-\epsilon_2}$ with ${\epsilon_1},\epsilon_2\in \{0,1\}$, and a valid $\psi$-word $f = f(\psi_1, \ldots, \psi_k)$.   Let $r := f(0)$.  \newline 
$\circ$ \ Declare whether or not  (\emph{\ref{Condition: Piece 1}}, \emph{\ref{Condition: Piece 2}} or \emph{\ref{Condition: Piece 3}}) of the Piece Criterion holds. 
If so, output   $\pi'$ of the Criterion and a valid $\psi$-word $f'=f'(\psi_1, \ldots, \psi_k)$ such that $f'(0)$ equals $\delta$ of the Criterion.   These satisfy $\ell(\pi') \leq \ell(\pi)$  and  $\ell(f') \leq \ell(f) +1$, and  $t^r \pi \in H_k t^s$ if and only if $t^{f'(0)} \pi' \in H_k t^s$. \newline
$\circ$ \  Halt in  time  $O((\ell(w)+\ell(f))^{k+4} )$.}
    \label{Alg: Front_m}
    \begin{algorithmic}[3]
    \State \textbf{if} $\epsilon_1=0$ (so \emph{\ref{Condition: Piece 1}} holds), \textbf{output} $\pi' := u a_m^{-\epsilon_2}$ and $f' := f$, and \textbf{halt} \label{first line of front}
    \State \textbf{run} $\alg{Psi}(f)$ to determine whether or not $r \leq 0$  \label{Psi run here}
    \State \textbf{if} $\epsilon_1=1$ and $r \leq 0$ (so \emph{\ref{Condition: Piece 2}} holds),  \textbf{output} $\pi' := u a_m^{-\epsilon_2}$ and $f' := \psi_m f$, and \textbf{halt} \label{bounds are obvious}
\State
   \State we now have that   $\epsilon_1=1$ and $r > 0$  (so \emph{\ref{Condition: Piece 1}} and  \emph{\ref{Condition: Piece 2}} both fail, and it remains to test \emph{\ref{Condition: Piece 3}})   \label{take stock}
    \State \textbf{run} $\alg{Prefix}_m$ on $\pi$  \label{use checkprefix}
 \State \qquad \textbf{if} it fails to return an $i$  \textbf{declare} that  \emph{\ref{Condition: Piece 1}}, \emph{\ref{Condition: Piece 2}} and \emph{\ref{Condition: Piece 3}} all fail and \textbf{halt}
\State \qquad \textbf{else}  it returns some  some $i$  
\State \textbf{run}  \alg{Psi} on input $\psi_1^{i}f$ to check whether  $i < r$    \label{run sign}
\State \qquad \textbf{if}  $i < r$,  \textbf{then} \textbf{declare} that   \emph{\ref{Condition: Piece 1}}, \emph{\ref{Condition: Piece 2}} and \emph{\ref{Condition: Piece 3}} all fail 
 \State \qquad \textbf{else}  \emph{\ref{Condition: Piece 3}} holds, so   \textbf{return} the reduced form  $\pi'$ of  $\theta^r(a_m\inv)\pi$ and $f' := \psi_1f$   \label{return the answer} 
       \end{algorithmic}
  \end{algorithm}

	\begin{proof}[Correctness of $\alg{Front}_m$.] \quad \\ \vspace*{-6mm}
	\begin{itemize}
	\item[\ref{Psi run here}:] In was established in Section~\ref{Psi in detail} that  $\alg{Psi}$ on input $f$ halts in time $O(  \ell(f)^{k+4})$. 
	\item[\ref{take stock}:]  Whether \emph{\ref{Condition: Piece 3}}  holds depends on whether $\theta^{r-1}(a_m)$ is   a prefix of $\pi$, so that is what the remainder of the algorithm examines.  
\item[\ref{use checkprefix}:]  $\alg{Prefix}_m$ halts in time $O(\ell(\pi)^2)$.
\item[\ref{run sign}:]  At this point we know that  $\theta^{i-1}(a_m)$ is a prefix of $\pi$, and so $i \leq \ell(\pi)$.  Therefore,  $\ell(\psi_1^{i}f) \leq \ell(\pi)  + \ell(f)$, and so, by the bounds established in Section~\ref{Psi in detail},   \alg{Psi} halts  in time $O( (\ell(\pi)  + \ell(f))^{k+4})$.  
\item[\ref{return the answer}:]
For all $0 \leq p \leq q$,   $\theta^{p}(a_m)$ is a prefix of $\theta^{q}(a_m)$: after all, for $q \geq 0$,  $\theta^{q+1}(a_m) =  \theta^{q}(a_m) \theta^{q}(a_{m-1})$.  So, given that we know at this point that   $\theta^{i-1}(a_m)$ is a prefix of $\pi$ and  $r \leq i$, it is the case that $\theta^{r-1}(a_m)$ is also a prefix of $\pi$.   Note that $\theta^r(a_m\inv)\pi$ is $\theta^{r-1}(a_{m-1}\inv) ua_m^{-\epsilon_2}$  of the Criterion when  \emph{\ref{Condition: Piece 3}} holds. 
\end{itemize}
In  lines \ref{first line of front}, \ref{bounds are obvious}  and \ref{return the answer}, the claimed bound  $\ell(f') \leq \ell(f) +1$ is immediate, as is $\ell(\pi') \leq \ell(\pi)$   in lines \ref{first line of front}  and \ref{bounds are obvious}.   In line~\ref{return the answer}, $\pi'$ is the reduced form    of  $\theta^r(a_m\inv)\pi$
and   $\theta^{r-1}(a_m)$ is a prefix of $\pi$.   Now $\theta^r(a_m\inv)  = \theta^{r-1}(a_{m-1}\inv)  \theta^{r-1}(a_m\inv)$ and the length of $\theta^{r-1}(a_m\inv)$ is at least half  that  of $\theta^r(a_m\inv)$ (as $r>0$), and the last letter of  $\theta^{r-1}(a_{m-1}\inv)$  is $a_{m-1}\inv$.   So all of the prefix $\theta^{r-1}(a_m)$  of $\pi$ is cancelled away when $\theta^r(a_m\inv)\pi$ is freely reduced to give $\pi'$, and  $\ell(\pi') \leq \ell(\pi) $, as claimed.

The algorithm halts in   time   $O( (\ell(\pi)  + \ell(f))^{k+4})$  by our comments on lines \ref{take stock}, \ref{use checkprefix} and \ref{run sign} and the fact that  $\theta^r(a_m\inv)\pi$  in the final line has length at most $3\ell(\pi)$: after all, 
$\theta^r(a_m\inv) = \theta^{r-1}(a_{m-1}\inv)\theta^{r-1}(a_{m}\inv)$ and $\ell(\theta^{r-1}(a_{m-1}\inv))$ is at most $\ell(\theta^{r-1}(a_{m}\inv))$, and $\theta^{r-1}(a_{m}\inv)$ is the inverse of a prefix of  $\pi$.  	
\end{proof}

Next we construct $\alg{Back}_m$, $\alg{Piece}_m$ and    $\alg{Push}_{m}$.

For a  rank-$m$ piece $\pi$ which does not start with the letter  $a_m$,  $\alg{Back}_m$  determines  whether $t^{f(0)} \pi \in Ht^s$ for some $s\in \Z$, and if so it outputs a $\psi$-word $f'$ with  $f'(0) =s$.  Initially, it works  similarly to $\alg{Front}_m$ in that it reduces its task to performing a similar determination without the final letter   $a_m\inv$.  But then it calls $\alg{Push}_{m-1}$ to find out whether the $s$ exists, and, if so, to output a $\psi$-word $f'$ with  $f'(0) =s$.   A crucial feature of this algorithm is that  the lengths of the input data  to $\alg{Push}_{m-1}$ (specifically  $u'$ and $f$)  is carefully bounded in terms of the length of the inputs to $\alg{Back}_m$, and so does not blow up course of the induction.

\begin{algorithm}[h]
\label{Alg: Back_m}
\caption{ --- $\alg{Back}_m$, $m \geq 3$.     \newline
$\circ$ \ Input  a rank-$m$ piece $\pi  = u a_m^{-\epsilon_2}$ (so $u=u(a_1, \ldots, a_{m-1})$ is reduced and $\epsilon_2 \in \set{0,1}$)  and a valid $\psi$-word $f=f(\psi_1, \ldots, \psi_k)$. Let $r := f(0)$. \newline 
$\circ$ \ Declare whether or not $t^{r} \pi \in \bigcup_{s \in \Z}  H_k t^s$.  And, if it is,  
return a valid $\psi$-word $f'$ such that $t^{f(0)}\pi\in H_kt^{f'(0)}$,   $\ell(f')\le\ell(f) + 2(m-1)\ell(\pi) + 1$   and $\rank(f')\le \max\{\rank(f),m\}$. \newline
$\circ$ \ Halt  in   time   $O((\ell(\pi)+\ell(f))^{2m+k})$.} 
	\begin{algorithmic}[3]
		\State   \textbf{run} $\alg{Push}_{m-1}(u,f)$     to test whether or not $t^r u \in \bigcup_{s \in \Z} H_k t^s$ \label{first line}
		\State if it is, let $g$ be the valid $\psi$-word it outputs such that $t^r u \in  H_k t^{g(0)}$ \label{second line}
		\State
		\State \textbf{if} $\epsilon_2=0$,    \label{Line: Back trivial case1}
		\State \qquad \textbf{if}  $t^r u \in  H_k t^{g(0)}$  (so, (a) of the Criterion holds with $s=g(0)$), \textbf{return}  $f' := g$   \label{Line: Back trivial case2}
		\State \qquad \textbf{else} \textbf{declare}    $t^{r} \pi \notin \bigcup_{s \in \Z}  H_k t^s$  \label{Line: Back trivial case3}
		\State \qquad \textbf{halt} 
		\State 
		\State we now have that $\epsilon_2=1$
		\State  \textbf{run} $\alg{Psi}(\psi_m\inv g)$ to check   validity of  $\psi_m\inv g$ (so whether $g(0) \in \textup{Img} \, \psi_m$)  \label{start b}
		\State and, if so, to check  $\psi_m\inv g(0) \leq 0$ 	 (so, whether (\textit{\ref{Condition: Piece b}}) of the Criterion holds with $s=\psi_m\inv g(0)$) 
		 \State \qquad \textbf{if} so,  \textbf{halt} and \textbf{return} $f' := \psi_m\inv g$    \label{Line: Back inner word check} \label{Line: Back Option 2 Determination}
		\State
		 
		\State \textbf{run} $\alg{Prefix}_m(\pi\inv)$ to determine the maximum $i$ (if any) such that $a_{m-1}\inv \theta^{i-1}(a_m\inv)$ is a suffix of $\pi$  \label{run prefix} 
		\State \qquad if there is no such $i$ \textbf{halt} and \textbf{declare}  $t^{r} \pi \notin \bigcup_{s \in \Z}  H_k t^s$ \label{exhausted pos}
		\For{$s=1$ to $i$} \label{Line: Back pos check loop}
 			\State \textbf{run}  $\alg{Push}_{m-1}(u',f)$  where $u'$ is the freely reduced word representing $u a_m\inv\theta^s(a_m)$
			\State \qquad \textbf{if} it  outputs a $\psi$-word $h$, \textbf{run} $\alg{Psi}(\psi_1^{s-1} h)$  to check if $h(0) =s-1$   
				\State \qquad \qquad \textbf{if} so \textbf{halt} and \textbf{return} $f' := \psi_1  h$  \label{Line: Back 3 output}
 		\EndFor \label{end for loop}
		\State
		\State \textbf{declare} that $t^{f(0)} w \notin \bigcup_{s \in \Z}  H_k t^s$  \label{Line: Back reject}
	\end{algorithmic}
\end{algorithm}

\begin{proof}[For $m \geq 3$, correctness of $\alg{Push}_{m-1}$ (as specified below) implies correctness of \alg{Back}$_m$.] 

The idea is to employ the Piece Criterion in the instance  when $\epsilon_1 =0$, and therefore $\delta =r$, $\pi'=\pi$ and Condition~\emph{i} holds.   In this circumstance, the Criterion tells us that $t^r \pi\in H_kt^s$ (that is, $t^{\delta} \pi' \in H_k t^s$) if and only if (\emph{\ref{Condition: Piece a}}, \emph{\ref{Condition: Piece b}} or \emph{\ref{Condition: Piece c}}) holds.  

\begin{itemize}
\item[\ref{second line}:] Referring to the specifications of \alg{Push}$_{m-1}$, we see that $\ell(g) \leq  \ell(u) + \ell(f)$ and $\rank(g)\le \max\{\rank(f),m\}$.  

\item[\ref{Line: Back trivial case1}--\ref{Line: Back trivial case3}:]     \alg{Push}$_{m-1}$ in line  lines~\ref{first line}--\ref{second line}  tests whether or not $t^{\delta}x^{-\epsilon_1}u$ (that is, $t^r u$) is in $\bigcup_{s \in \Z} H_k t^s$ and, if so, it identifies the $s$ such that  $t^{\delta}x^{-\epsilon_1}u \in H_k t^s$.  The Piece Criterion  then tells us  that the answer to whether  $t^r \pi \in \bigcup_{s \in \Z} H_k t^s$  is the same, and if affirmative the $s$ agrees. (This instance of the   Criterion has no real content because $t^{\delta}x^{-\epsilon_1}u = t^r \pi$.  The other two instances that follow are more substantial but will follow the same pattern of reasoning.)  By our comment on line~\ref{second line},  $\ell(f') \le  \ell(f) + \ell(u) = \ell(f) + \ell(\pi)$, and $\rank(f')\le \max\{\rank(f),m\}$, as required.

\item[\ref{start b}--\ref{Line: Back Option 2 Determination}:]  
Again, we refer back to lines~\ref{first line}--\ref{second line} for whether or not $t^{\delta}x^{-\epsilon_1}u$ (that is, $t^r u$) is in $\bigcup_{s_0 \in \Z} H_k t^{s_0}$.  Assuming that it is, in fact, it is in $H_k t^{g(0)}$, and then Condition~\emph{\ref{Condition: Piece b}},  is satisfied if and only if $g(0) = \psi_m(s)$ for some $s \leq 0$.  And that is checked in line  \ref{start b}.  The Piece Criterion  then tells us  that the answer to this   is the same as the answer to whether $t^r \pi \in \bigcup_{s \in \Z} H_k t^s$, and, if affirmative, the $s$ agrees.  
By our comment on line~\ref{second line},  $\ell(f') = \ell(g) +1 \le \ell(f)+\ell(u)+1 = \ell(f)+\ell(\pi)$ and $\rank(f')\le \max\{\rank(f),m\}$, as required.

\item[\ref{run prefix}--\ref{end for loop}:] The aim here is to determine whether Condition~\emph{\ref{Condition: Piece c}} holds---that is, whether $$t^r u a_m\inv \theta^s(a_m) \in H_k t^{s-1}$$ and $a_{m-1}\inv \theta^{s-1}(a_m\inv)$ is a suffix of $\pi$ for some $s >0$---and, if so, output a $\psi$-word $f'$ such that $f'(0) = s$. (This $s$ must be unique, if it exists, because, by the Criterion, it is the $s$ such that $t^r \pi \in H_k t^s$, and we know that is unique.)   

The possibilities for $s$ are limited to the range $1, \ldots, i$ by the suffix condition and the requirement that $s >0$, where $i$ is as found in line~\ref{run prefix} and must be at most $\ell(\pi)$.    If there is such a suffix  $a_{m-1}\inv \theta^{i-1}(a_m\inv)$   of $\pi$, then  $a_{m-1}\inv \theta^{s-1}(a_m\inv)$ is a suffix of $\pi$ for all $s \in \set{1, \ldots, i}$.   If there is no such suffix, then Condition~\emph{\ref{Condition: Piece c}} fails, and, as we know at this point that Conditions~\emph{\ref{Condition: Piece a}} and \emph{\ref{Condition: Piece b}} also fail, we declare in line~\ref{exhausted pos} that (by the Criterion),  $t^{r} \pi \notin \bigcup_{s \in \Z}  H_k t^s$.   

For each $s$ in the range $1, \ldots, i$, lines \ref{Line: Back pos check loop}--\ref{end for loop}  address the question of whether or not $t^r u a_m\inv \theta^s(a_m) \in H_k t^{s-1}$.  First   $\alg{Push}_{m-1}$ is called, which can be done because  on freely reducing  $u a_m\inv \theta^s(a_m)$, the $a_m\inv$ cancels with the $a_m$ at the start of $ \theta^s(a_m)$ to give a word of rank at most $m-1$.  $\alg{Push}_{m-1}$ either tells us that $t^r u a_m\inv \theta^s(a_m) \notin \bigcup_{s' \in \Z} H_k t^{s'}$, or it gives a $\psi$-word $h$ such that  
$t^r u a_m\inv \theta^s(a_m) \in  H_k t^{h(0)}$.  In the latter case, $\alg{Psi}$ is then used to test whether or not $h(0) = s-1$.  

By  the  specifications of \alg{Push}$_{m-1}$,  $\ell(h) \le  \ell(f)+2(m-1)\ell(u')$.  And, as $\pi = u a_m\inv$ has suffix 
$\theta^{s-1}(a_m\inv)$, when we form $u'$ by freely reducing $u a_m\inv\theta^s(a_m)$,  at least half of $\theta^s(a_m) = \theta^{s-1}(a_m)\theta^{s-1}(a_{m-1})$ cancels into $\pi$.  So $\ell(u') \le \ell(\pi)$, and  
 $$\ell(f')  \ = \    \ell(h) +1  \ \le \    \ell(f)+2(m-1)\ell(u') +1 \ \leq \     \ell(f) + 2(m-1)\ell(\pi)+1$$   
as required.  Also, it is immediate that $\rank(f')\le \max\{\rank(f),m\}$, as required.

\item[\ref{Line: Back reject}: ] At this point, we  know \emph{\ref{Condition: Piece a}}, \emph{\ref{Condition: Piece b}} and \emph{\ref{Condition: Piece c}}   fail for all $s \in \Z$, so $t^r \pi \notin \bigcup_{s \in \Z} H_kt^s$. 
\end{itemize}

\alg{Back}$_m$ runs $\alg{Push}_{m-1}(u,f)$ once (with $\ell(u) \leq \ell(\pi)$),  $\alg{Psi}(\psi_m\inv g)$ at most once (with $\ell(g) \leq \ell(\pi) + \ell(f)$),   $\alg{Prefix}_m(\pi\inv)$ at most once,  $\alg{Push}_{m-1}(u',f)$ at most $i \leq \ell(\pi)$ times (with $\ell(u') < \ell(\pi)$), and  $\alg{Psi}(\psi_1^{s-1} h)$ at most $i \leq \ell(\pi)$ times (with $1 \leq s \leq \ell(\pi)$ and $\ell(h) < \ell(f) + \ell(\pi)$).  Other operations such as free reductions of words etc.\ do not contribute significantly to the running time.  
Referring to the specifications of  
$\alg{Push}_{m-1}$, $\alg{Psi}$, and $\alg{Prefix}_m$, we see that they (respectively) contribute:
\begin{align*}
\ell(\pi) O((\ell(\pi)+\ell(f))^{2(m-1) +k +1})  & +  \ell(\pi) O((\ell(f)+2\ell(\pi))^{4+k})  + O(\ell(\pi)^2)  \\
 & \ = \  O((\ell(\pi)+\ell(f))^{2m+k}) 
\end{align*}
which is the claimed bound on the halting time of   $\alg{Back}_m$.  
\end{proof}

   \begin{algorithm}[h]
    \caption{ --- \alg{Piece}$_m$, $k \geq m\ge 2$.      \newline
$\circ$ \ Input  a rank-$m$ piece $\pi$ and a valid $\psi$-word $f=f(\psi_1, \ldots, \psi_k)$. \newline 
$\circ$ \ Declare  whether or not $t^{f(0)} \pi \in \bigcup_{s \in \Z} H_kt^s$ and, if it is, return  a valid $\psi$-word $g$ such that   $t^{f(0)} \pi \in H_kt^{g(0)}$, $\rank(g)\le \max \set{m,\,\rank(f)}$, and $\ell(g)\le  \ell(f) +2(m-1)\ell(\pi)+2$.  \newline
$\circ$ \ Halt in time $O((\ell(\pi)+\ell(f))^{2m+k})$. }
    \label{Alg: Piece_m}
    \begin{algorithmic}[3]
     \State \textbf{if} $m=2$
     \State \qquad $\pi$ is $a_2^{\epsilon_1} a_1^l a_2^{-\epsilon_2}$ for some  $l\in\Z$ and some $\epsilon_1, \epsilon_2 \in \{0,1\}$
    \State  \qquad set $g = \psi_2^{-\epsilon_2} \psi_1^l \psi_2^{\epsilon_1} f$  
    \State  \qquad \textbf{run} $\alg{Psi}(g)$
			\State  \qquad \textbf{if}  it declares that $g$ is invalid, \textbf{then} \textbf{declare} that     $t^{f(0)} \pi \notin \bigcup_{s \in \Z} H_kt^s$
			\State  \qquad \textbf{else return} $g$
			\State  \qquad \textbf{halt}
    \State
    \State  \textbf{if}  $m >2$
		\State \qquad \textbf{run}    $\alg{Front}_m(\pi,f)$
		\State \qquad \textbf{if} it declares that  \emph{\ref{Condition: Piece 1}},  \emph{\ref{Condition: Piece 2}}  and  \emph{\ref{Condition: Piece 3}}   of the Piece Criterion all fail
		\State \qquad \qquad  \textbf{declare} that $t^{f(0)} \pi \notin \bigcup_{s \in \Z} H_kt^s$ and \textbf{halt} 
		\State \qquad \textbf{else} \textbf{run} $\alg{Back}_m$ on the output $(\pi',f')$ of $\alg{Front}_m$ and \textbf{return} the result    
    \end{algorithmic}
  \end{algorithm}

\begin{proof}[The correctness of \alg{Piece}$_2$.]
By applying Proposition \ref{Prop: Hydra Game} repeatedly, we see that    $t^{f(0)} \pi \in H_kt^s$   if and only if  
$t^{\psi_1^l \psi_2^{\epsilon_1} f(0)} a_2^{-\epsilon_2} \in H_k t^s$, since $\psi_1^l \psi_2^{\epsilon_1} f$ is valid as  the domains of $\psi_1$ and $\psi_2$ are $\Z$.   So, by Corollary~\ref{pushing cor},   $t^{f(0)} \pi \in H_kt^s$   if and only if  $g = \psi_2\inv \psi_1^l \psi_2^{\epsilon_1} f$ is  valid and $s= \psi_2\inv\psi_1^l \psi_2^{\epsilon_1} f(0)$.  

 It halts in   time   $O(\ell(w) + \ell(f)^{6})$ because  \alg{Psi} halts in time  $O(\ell(f)^{6})$ on input  $\psi_2\inv f$ by the bounds established in Section~\ref{Psi in detail}, given that $f$ is of rank $2$.  
\end{proof}

	\begin{proof}[For   $k \geq m \geq 3$, correctness of $\alg{Back}_{m}$ implies correctness of $\alg{Piece}_m$.]  It   follows from the specifications of $\alg{Front}_m$ and $\alg{Back}_m$, that they combine in the manner of  $\alg{Piece}_m$ to declare whether or not $t^{f(0)} \pi \in \bigcup_{s \in \Z} H_kt^s$, and if it is to return a $g$ with the claimed properties.

	Using  that  $\ell(\pi') \leq \ell(\pi)$ and $\ell(f')\le \ell(f) +1$, we can add the halting-time estimates for 
	 $\alg{Front}_m$ and $\alg{Back}_m$, to deduce that   $\alg{Piece}_m$ halts in time 
	 \begin{equation*}
	 O((\ell(w)+\ell(f))^{\max\set{k+4, 2m+k}}) \ = \ O((\ell(w)+\ell(f))^{2m+k}).  \qedhere
	 \end{equation*}
	 	\end{proof}

  \begin{algorithm}[h]
    \caption{ --- $\alg{Push}_m$, $k \geq m\ge 1$.      \newline
$\circ$ \   Input  a reduced word $v = v(a_1, \ldots, a_m)$  and a valid $\psi$-word $f=f(\psi_1, \ldots, \psi_k)$. \newline 
$\circ$ \ Declare whether or not $t^{f(0)}v\in \bigcup_{s \in \Z} H_kt^s$.  If it is, return  a valid $\psi$-word $g$ with $\ell(g)\le \ell(f)+2m\ell(v)$, $\rank(g)\le \max\set{m,\,\rank(f)}$ and $t^{f(0)}v\in H_kt^{g(0)}$. \newline
$\circ$ \ Halt  time $O((\ell(v)+\ell(f))^{2m+k+1})$. }
    \label{Alg: rank m word}
    \begin{algorithmic}[3]
    			\State \textbf{if} $m=1$ (and so $v = a_1^{l}$ for some $l\in\Z$) \label{first line came m=1}
			\State \qquad \textbf{declare} \textbf{yes}, \textbf{output}  $g := \psi_1^l f$ and \textbf{halt} \label{line came m=1}
			\State 
			\State \textbf{if} $m>1$
			\State \qquad let $\pi_1\cdots \pi_p$ be the rank-$m$ decomposition of $v$ into pieces as per Section~\ref{MP problem intro} 
			\State \qquad \textbf{set} $f_0:= f$
			\State \qquad \textbf{for}  $i=1$ to $p$ 
				\State \qquad \qquad \textbf{run} $\alg{Piece}_m(\pi_i,f_{i-1})$
				\State \qquad \qquad \textbf{if}  it declares   $t^{f_{i-1}(0)} \pi_i \notin \bigcup_{s \in \Z} H_kt^s$, \textbf{declare}   $t^{f(0)}w \notin \bigcup_{s \in \Z} H_kt^s$ and \textbf{halt}		
						\State \qquad \qquad \textbf{else}  set $f_i$ to be its output  \label{Line: Word for check}	
      \State \qquad \textbf{end for}
      \State \qquad   \textbf{return} $g := f_p$
    \end{algorithmic}
  \end{algorithm}

\begin{proof}[The correctness of $\alg{Push}_1$.]  The case $m=1$ is handled in lines~\ref{first line came m=1}--\ref{line came m=1}.  The point is that in $G_k$ we have  $t^{f(0)}a_1^l = (a_1t)^l t^{f(0)-l} \in H_k t^{g(0)}$ since $g(0) =  \psi_1^l f = f(0)-l$.   That it halts within the time bound is clear.   
\end{proof}

\begin{proof}[For $k \geq m\ge 2$, correctness of $\alg{Piece}_m$ implies correctness of $\alg{Push}_m$.]

This algorithm runs in accordance with Lemma~6.2 of \cite{DR} as we described in Section~\ref{MP problem intro}.

By the specifications of $\alg{Piece}_m$, after the $i$th iteration of the \textbf{for} loop, $$\ell(f_i) \  \le \  \ell(f) + \sum_{j=1}^i (2(m-1)\ell(\pi_j) +2 ) \ \le \ \ell(f)+2(m-1)\ell(v)+2i \ \leq \ \ell(f) +2m\ell(v),$$ as $i \leq \ell(v)$, and $\rank(f_i)\le \max\set{m,\,\rank(f)}$.  In particular,   $\rank(g)\le \max\set{m,\,\rank(f)}$, as claimed. 

 $\alg{Piece}_m(\pi_i,f_{i-1})$ halts in time $O( (\ell(\pi_i) + \ell(f_{i-1}))^{2m+k})$  and $p \leq \ell(\pi)$, so   for $1 \leq i \leq p$ , $$\ell(\pi_i) + \ell(f_{i-1}) \  \leq \ \ell(\pi_i) +  \ell(\pi_1) + \cdots + \ell(\pi_{i-1}) + \ell(f) + i -1 \ = \  O( (\ell(v) + \ell(f))).$$
    So  $\alg{Push}_m$ halts in time $O( (\ell(v) + \ell(f))^{2m+k +1})$. 
\end{proof}

\begin{proof}[Correctness of $\alg{Piece}_m$ for $2 \leq m \leq k$, of $\alg{Push}_m$ for $1 \leq m \leq k$, and of $\alg{Back}_m$ for $3 \leq m \leq k$.]  We established the correctness of  $\alg{Push}_1$ and   $\alg{Piece}_2$ individually. The implications proved above give the correctness of the others by induction in the order:
\begin{equation*}
\alg{Piece}_2 \implies \alg{Push}_2 \implies \alg{Back}_3 \implies \alg{Piece}_3 \implies \alg{Push}_3 \implies  \alg{Back}_4 \implies \cdots. \qedhere
\end{equation*}
\end{proof}

Finally, we are ready for:

  \begin{algorithm}[h]
    \caption{ --- $\alg{Member}_k$, $k\ge 1$.      \newline
$\circ$ \ Input    a  word $w=w(a_1, \ldots, a_k, t)$. \newline 
$\circ$ \  Declare whether or not $w \in H_k$. \newline
$\circ$ \  Halt in   time $O(  \ell(w)^{3k^2 +k})$.}
    \label{Alg: rank m member}
    \begin{algorithmic}[3]
    			\State convert $w$ to normal form $t^r v$ where $v=v(a_1, \ldots, a_k)$ is reduced, $r \in \Z$, and $t^rv=w$ in $G_k$,  
 as described at the start of Section~\ref{MP problem intro}
			\State \textbf{set} $f = \psi_1^{-r}$
			\State \textbf{run} $\alg{Push}_k(v,f)$ 
			\State{\textbf{if} it outputs a (necessarily valid) $\psi$-word $g$}
			\State \qquad   \textbf{then run} $\alg{Psi}(g)$ to test whether $g(0) =0$
			\State \qquad  \qquad \textbf{if} so, \textbf{declare}  $w \in H_k$ and \textbf{halt}
				\State  \textbf{declare}  $w \notin H_k$
    \end{algorithmic}
  \end{algorithm}

\begin{proof}[Correctness of \alg{Member}$_k$.]
The process set out at the start of Section~\ref{MP problem intro} produces  $t^r v$  in time $O(\ell(w)^k)$.  Moreover, $\ell(f) = \abs{r} \leq  \ell(w)$ and
$\ell(v)  \leq   \ell(w) (\ell(w) +1)^{k-1}$.

The algorithm calls $\alg{Push}_k(v,f)$, which halts in time $$O((\ell(v)+\ell(f))^{2k+k+1}) \ = \ O((\ell(w)^k+\ell(w))^{2k+k+1}) \ = \ O(\ell(w)^{3k^2 +k}).$$ 
It either declares that $t^rv\notin \bigcup_{s \in \Z} H_kt^s$, and so $w\notin H_k$, or it  
returns a valid $\psi$-word $g$ such that $w\in H_kt^{g(0)}$ and $\ell(g) \leq \ell(f)+ 2 k \ell(v) \leq \ell(w) +  2 k \ell(w) (\ell(w) +1)^{k-1} = O(\ell(w)^{k})$.   But then
  $w\in H_k$ if and only if $g(0) = 0$ (by Lemma~6.1  of \cite{DR}), which is precisely what the algorithm uses $\alg{Psi}(g)$  to check.  This call on  $\alg{Psi}$ halts in time $O( (\ell(w)^{k})^{k+4} ) = O(\ell(w)^{k^2+4k})$ when $k>1$ and in time $O(\ell(w))$ when $k=1$.  So, as $\max \set{k^2 +4k, 3k^2 +k} = 3k^2+k$ for all $k >1$,  \alg{Member}$_k$ halts in time $O(\ell(w)^{3k^2 +k})$, as required.  
\end{proof}

\section{Conclusion} \label{MP to WP} \label{Conclusion}

The construction and analysis of  $\alg{Member}_k$ in the last section solves the membership problem for $H_k$ in $G_k$ in polynomial time, indeed in $O(n^{3k^2+k})$ time, where $n$ is the length of the input word, and so proves Theorem~\ref{MP theorem}.  

Here is why a polynomial time (indeed $O(n^{3k^2+k+2})$ time)  solution to the word problem for $\Gamma_k$   follows, giving Theorem~\ref{WP theorem}.

Suppose we have a word $x = x(a_1, \ldots, a_k, p, t)$ of length $n$ on the generators of $\Gamma_k$. Recall that $\Gamma_k$ is the HNN-extension of $G_k$ with stable letter $p$ commuting with all elements of $H_k$.   Britton's Lemma (see, for example, \cite{BrH, Lyndon, Stillwell}) tells us that if $x=1$ in $\Gamma_k$, then it has a subword $p^{\pm 1} w p^{\mp 1}$ such that $w = w(a_1, \ldots, a_k, t)$ and represents an element of $H_k$.  

There are fewer than $n$  subwords $p^{\pm 1} w p^{\mp 1}$ in $x$ such that $w = w(a_1, \ldots, a_k, t)$.  As discussed above, $\alg{Member}_k$ checks whether  such a $w \in H_k$  in time in $O(n^{3k^2+k})$.  If none  represents an element of $H_k$, we conclude that $x \neq 1$ in $G_k$.  If,   for some such subword $p^{\pm 1} w p^{\mp 1}$, we find $w \in H_k$, then we can remove the $p^{\pm 1}$ and $p^{\mp 1}$ to give a word of length $n-2$ representing the same element of $G_k$. 

 This repeats at most $n/2$ times until we have either determined that $x \neq 1$ in $\Gamma_k$, or no $p^{\pm 1}$ remain.  In the latter case, we then have a word on $a^{\pm 1}_1, \ldots, a^{\pm 1}_k, t^{\pm 1}$ of length at most $n$, which represents an element of the subgroup $G_k$.  But $G_k$ is automatic (Theorem 1.3 of \cite{DR}) and so there is an algorithm  solving its word problem in  $O(n^2)$ time (Theorem~2.3.10 of \cite{Epstein}).  
 
In all, we have called $\alg{Member}_k$ at most $n^2/2$ times and an algorithm solving the word problem in $G_k$ once, in every case with input of length at most $n$.  It follows that the whole process can be completed in time $O(n^{3k^2+k+2})$.

\bibliographystyle{plain}
\bibliography{$HOME/Dropbox/Bibliographies/bibli}

\end{document}